\documentclass[11pt]{amsart}
\usepackage{geometry}
\usepackage[dvips]{graphicx}

\usepackage{amscd,amsfonts,amssymb,amsmath,amsthm,latexsym}

\usepackage{graphics}
\usepackage[all]{xy}
\xyoption{curve}
\xyoption{import}
\xyoption{arc}
\xyoption{ps}

\numberwithin{equation}{section}

\theoremstyle{plain}
    \newtheorem{thm}{Theorem}[section]
    \newtheorem{lemma}[thm]{Lemma}
    \newtheorem{coro}[thm]{Corollary}
    \newtheorem{prop}[thm]{Proposition}

    \newtheorem{question}[thm]{Question}
\theoremstyle{definition}
    \newtheorem{defi}[thm]{Definition}
    \newtheorem{ex}[thm]{Example}
    \newtheorem{remark}[thm]{Remark}
\theoremstyle{remark}

\newcommand{\suchthat}{\ | \ }

\newcommand{\RA}[1]{R\langle\hspace{-0.05cm}\langle #1\rangle\hspace{-0.05cm}\rangle}
\newcommand{\RAextension}[2]{R_{#2}\langle\hspace{-0.05cm}\langle #1_{#2}\rangle\hspace{-0.05cm}\rangle}
\newcommand{\comp}[3]{[#1\hspace{0.032cm}#2#3]}
\newcommand{\compbwa}{\comp{b}{\omega}{a}}

\newcommand{\Gal}{\operatorname{Gal}}

\newcommand{\field}{F}

\newcommand{\lcm}{\operatorname{lcm}}
\newcommand{\Hom}{\operatorname{Hom}}

\newcommand{\Z}{\mathbb{Z}}
\newcommand{\diag}{\operatorname{diag}}
\newcommand{\dtuple}{\mathbf{d}}
\newcommand{\triv}{\operatorname{triv}}
\newcommand{\red}{\operatorname{red}}
\newcommand{\maxid}{\mathfrak{m}}
\newcommand{\Min}{\operatorname{in}}
\newcommand{\Mout}{\operatorname{out}}
\newcommand{\B}{\mathcal{B}}

\newcommand{\image}{\operatorname{im}}

\newcommand{\Q}{\mathbb{Q}}
\newcommand{\pathalgQd}{R\langle Q\rangle}
\newcommand{\completeQdcoprime}{\RA{A}}
\newcommand{\completeQd}{\completeQdcoprime}
\newcommand{\pathalgQdcoprime}{R\langle A\rangle}
\newcommand{\completeQpdcoprime}{\RA{A'}}

\newcommand{\mult}{\operatorname{mult}}
\newcommand{\Mat}{\operatorname{Mat}}

\newcommand{\coker}{\operatorname{coker}}
\newcommand{\DR}{\rho}

\begin{document}

\title{Strongly primitive species with potentials I: Mutations}
\author[Daniel Labardini-Fragoso and Andrei Zelevinsky]{Daniel Labardini-Fragoso\\ Andrei Zelevinsky $\dag$}
\address{Mathematisches Institut, Universit\"{a}t Bonn, Germany 53115}
\email{labardini@math.uni-bonn.de}
\address{Department of Mathematics, Northeastern University, Boston, MA 02115}
\subjclass[2010]{Primary 16G10, Secondary 16G20,13F60}
\keywords{Species, potential, mutation, decorated representation, unfolding}
\thanks{$\dag$ The second author passed away on April 10, 2013.}
\maketitle

\begin{abstract} Motivated by the mutation theory of quivers with potentials developed by Derksen-Weyman-Zelevinsky, and the representation-theoretic approach to cluster algebras it provides, we propose a mutation theory of species with potentials for species that arise from skew-symmetrizable matrices that admit a skew-symmetrizer with pairwise coprime diagonal entries. The class of skew-symmetrizable matrices covered by the mutation theory proposed here contains a class of matrices that do not admit global unfoldings, that is,
unfoldings compatible with all possible sequences of mutations.
\end{abstract}

\tableofcontents

\section{Introduction}\label{sec:intro}

The mutation theory of quivers with potentials developed by H. Derksen, J. Weyman and the second author of this note (cf. \cite{DWZ1,DWZ2}), has proven useful not only in cluster algebra theory, but in other areas of mathematics as well. On the cluster algebra side, one can mention the proofs of several conjectures of S. Fomin and the second author's ``\emph{Cluster algebras IV: Coefficients}" \cite{FZ4}, using representation-theoretic machineries based on QPs and their mutations (e.g. \cite{CKLP,DWZ2,Nagao,Plamondon}). Outside cluster algebras, one can mention B. Keller's proof \cite{Keller} of Zamolodchikov's periodicity conjecture using cluster categories associated to quivers with potentials, or the use of the QPs associated to triangulated surfaces \cite{Labardini1,Labardini4} in T. Bridgeland and I. Smith's realization \cite{BS} of spaces of stability conditions as spaces of quadratic differentials on Riemann surfaces.

However, quivers with potentials are not sufficient to treat all cluster algebras. The reason is that not all cluster algebras are skew-symmetric. In fact, most cluster algebras are not skew-symmetric, and even the standard procedure of folding/unfolding does not allow to obtain all cluster algebras from skew-symmetric ones. Keeping this in mind, in this paper we propose a mutation theory of species with potentials that covers the species that arise from what we call \emph{strongly primitive} skew-symmetrizable matrices, that is, matrices that admit a skew-symmetrizer with pairwise coprime diagonal entries.

Strongly motivated by 
\cite{DWZ1}, the mutation theory proposed here generalizes the mutation theory of quivers with potentials, and although it is not general enough to cover species arising from all possible skew-symmetrizable matrices, it does cover several cases where it is not possible to perform unfolding from the cluster algebra point of view.

Recall that an $n\times n$ integer matrix $B$ is called \emph{skew-symmetrizable} if there exist positive integers $d_1,\ldots,d_n$, such that $DB$ is skew-symmetric, where $D=\diag(d_1,\ldots,d_n)$. The matrix $D$ is then called a \emph{skew-symmetrizer} of $B$. It is well-known that every skew-symmetrizable integer matrix matrix $B$ gives rise to a \emph{valued quiver}, that is, a finite directed graph whose arrows have been labeled by pairs of integers satisfying certain conditions.  
However, in this paper we shall not work directly with valued quivers, but rather with some closely related objects which we call \emph{weighted quivers}. A \emph{weighted quiver} is a pair $(Q,\dtuple)$, where $Q=(Q_0,Q_1,t,h)$ is a quiver (possibly with 2-cycles or multiple arrows), and $\dtuple=(d_{i})_{i\in Q_0}$ is a tuple that attaches a positive integer $d_i$ to each vertex $i$ of $Q$. If $B$ is an $n\times n$ skew-symmetrizable integer matrix and $D$ is a skew-symmetrizer for $B$, one can define a weighted quiver $(Q,\dtuple)$ on the vertex set $Q_0=\{1,\ldots,n\}$ as follows. For every pair of vertices $i,j\in Q_0$ such that $b_{ij}\geq 0$, the quiver $Q$ has
$$
\frac{b_{ij}\gcd{(d_i,d_j)}}{d_j} \ \ \ \text{arrows from $j$ to $i$.}
$$
The tuple $\dtuple$ is defined in the obvious way, namely, by attaching the $i^{\operatorname{th}}$ diagonal entry of $D$ to each vertex $i\in Q_0$. It can be easily seen that for a fixed integer matrix $D=\diag(d_1,\ldots,d_n)$ with positive diagonal entries, the assignment $B\mapsto(Q,\dtuple)$ is a bijection between the set of skew-symmetrizable matrices that can be skew-symmetrized by $D$, and the set of 2-acyclic weighted quivers on the vertex set $Q_0=\{1,\ldots,n\}$ and with weight tuple $\dtuple=(d_1,\ldots,d_n)$. 
This implies in particular that the operation of \emph{matrix mutation} can be translated from the language of skew-symmetrizable matrices to the language of weighted quivers. The mutation rule for 2-acyclic weighted quivers takes a rather neat form that can be described as a three-step procedure as follows. Given a 2-acyclic weighted quiver $(Q,\dtuple)$ and a vertex $k\in Q_0$, define the \emph{mutation} of $(Q,\dtuple)$ with respect to $k$ as the the weighted quiver $\mu_k(Q,\dtuple)$ obtained from $(Q,\dtuple)$ by the performance of the following three steps:
\begin{itemize}
\item[(Step 1)] For each pair of arrows $j\rightarrow k$, $k\rightarrow i$, of $Q$, add $\frac{\gcd(d_i,d_j)d_k}{\gcd(d_i,d_k)\gcd(d_k,d_j)}$ ``composite'' arrows from $j$ to $i$;
\item[(Step 2)] reverse all arrows incident to $k$;
\item[(Step 3)] choose a maximal collection of disjoint 2-cycles and delete it.
\end{itemize}
The weighted quiver $\mu_k(Q,\dtuple)$ has vertex set $Q_0$ and weight tuple $\dtuple$. The underlying quiver of $\mu_k(Q,\dtuple)$ will be denoted by $\mu_k(Q)$. If $\dtuple=(1,\ldots,1)$, then $\mu_k(Q)$ coincides with the quiver obtained from $Q$ by ordinary quiver mutation, but this is not the case for arbitrary values of the $d_i$s.

Every weighted quiver 
has \emph{species realizations} over finite or $p$-adic fields. Over finite fields, such realizations\footnote{which also go under the name of \emph{modulations of valued quivers}} date back at least to the works of Dlab-Ringel \cite{DR} and Gabriel \cite{Gabriel}.
Species realizations of weighted quivers give rise to (complete) path algebras, a natural algebraic framework for potentials and related algebraic and representation-theoretic objects. One of the main aims of the present paper is to lift the three mutation steps of the previous paragraph from the combinatorial level 
to the algebraic and representation-theoretic levels by defining and studying \emph{mutations of species with potentials and their representations}.

Let us describe the contents of the paper in more detail. In Section \ref{sec:weighted-quivers} we show how to associate a weighted quiver to a skew-symmetrizable matrix, and translate the operation of matrix mutation to the language of weighted quivers.

In Section \ref{sec:assumptions} we explicitly state the assumptions under which we will work throughout the paper. The first of these assumptions is at the level of weighted quivers, and the second and third assumptions are at the level of the fields underlying the construction of path algebras from weighted quivers. To be more precise, we restrict our attention to \emph{strongly primitive} weighted quivers, that is, weighted quivers $(Q,\dtuple)$ whose weight tuple $\dtuple=(d_i)_{i\in Q_0}$ satisfies $\gcd(d_i,d_j)=1$ for all vertices $i,j\in Q_0$ with $i\neq j$.

At the level of fields, our assumptions are as follows. Let $E/F$ be a degree-$d$ field extension, where $d$ is the least common multiple of the integers $d_i$ conforming the tuple $\dtuple=(d_i)_{i\in Q_0}$. We assume that $F$ contains a primitive $d^{\operatorname{th}}$ root of unity, and that $E$ is a cyclic Galois extension of $F$ (that is, a Galois extension 
with cyclic Galois group). 
These assumptions on $F$ and $E$ guarantee the existence of an \emph{eigenbasis} $\B_E$ of $E/F$, that is, an $F$-vector space basis of $E$ consisting of eigenvectors of all elements of the Galois group $\Gal(E/F)$. The eigenbasis $\B_E$ possesses a very useful multiplicative property, namely, the product of any two elements of $\B_E$ is an $F$-multiple of some element of $\B_E$. Moreover, for every $i\in Q_0$, the intersection of $\B_E$ with the unique degree-$d_i$ extension $F_i$ of $F$ is an eigenbasis $\B_i$ of $F_i/F$, with the same multiplicative property as $\B_E$.

With the data $(Q,\dtuple)$, $F$ and $E$ at hand, in Section \ref{sec:path-algs} we introduce
an algebraic setup for constructing path algebras of weighted quivers. Let
$$
R=\bigoplus_{i\in Q_0}F_i \ \ \ \text{and} \ \ \ A=\bigoplus_{(i,j)\in Q_0\times Q_0}\bigoplus_{a:j\to i}F_i\otimes_F F_j,
$$
where, for each $i\in Q_0$, $F_i/F$ is the unique degree-$d_i$ field subextension of $E/F$. Then $R$ is a commutative semisimple $F$-algebra under the usual addition an multiplication defined componentwise, and $A$ is endowed with a natural structure of $R$-$R$-bimodule. The \emph{complete path algebra of $(Q,\dtuple)$ over the extension $E/F$} is defined to be the complete tensor algebra of $A$ over $R$, that is,
$$
\RA{A}=\prod_{\ell\geq0}A^{\ell},
$$
where $A^\ell$ denotes the $\ell$-fold tensor product of $A$ with itself as an $R$-$R$-bimodule.
Given our assumptions on $\dtuple$, $F$ and $E$, it is easy to see that $A^\ell$ has an $F$-vector space basis given by all elements of the form $\omega_0 a_1\omega_1\ldots\omega_{\ell-1}a_\ell\omega_\ell$, where the $a_k$ are arrows of $Q$, and each $\omega_k$ is an element of the corresponding eigenbasis $\B_{t(a_k)}$. We call such elements \emph{paths} on $(Q,\dtuple)$. Thus, every element of the complete path algebra $\RA{A}$ can be uniquely expressed as a possibly infinite $F$-linear combination of paths on $(Q,\dtuple)$ (the uniqueness of such an expression is directly due to the assumption $\gcd(d_i,d_j)=1$ for $i\neq j$).

In Section \ref{sec:potentials-and-Jacobians} we introduce some of our main objects of study: potentials and
their cyclic derivatives. A \emph{potential on $(Q,\dtuple)$ over the extension $E/F$} is an element of $\RA{A}$ which is a possibly infinite $F$-linear combination of cyclic paths. We refer to the pair $(A,S)$ as a \emph{strongly primitive species with potential over $E/F$}, or \emph{SP} for short. Two SPs over $E/F$, $(A,S)$ and $(A',S')$, are \emph{right-equivalent} if there exists an algebra isomorphism $\varphi:\RA{A}\to \RA{A'}$ fixing pointwise the underlying semisimple ring $R$, such that $\varphi(S)$ is \emph{cyclically-equivalent} to $S'$. The \emph{cyclic derivatives} of a potential $S$ are labeled by the arrows of the quiver $Q$. The topological closure of the two-sided ideal they generate in $\RA{A}$ is called the \emph{Jacobian ideal} $J(S)$. If $\varphi:(A,S)\rightarrow(A',S')$ is a right-equivalence, then $J(\varphi(S))=\varphi(J(S))$. To show this, we adapt to our setup the ``differential calculus" for cyclic derivatives developed in \cite{DWZ1} in the case of ordinary quivers.

In Section \ref{sec:change-of-base-field} we study the behavior of (complete) path algebras and Jacobian algebras when the ground field $F$ is replaced by an extension $K$. We show that if $K$ is a finite-degree extension of $F$ which is \emph{linearly disjoint} from the extension $E/F$, then there is a canonical inclusion of the complete path algebra over $E/F$ into the complete path algebra of over $KE/K$, where $KE$ is the \emph{compositum} of $K$ and $E$ inside an \emph{a priori} fixed algebraic closure of $F$. Furthermore, an $F$-basis of paths is also a $K$-basis of paths. We also show that replacing $F$ with $K$ corresponds to tensoring with $K$ over $F$, both at the level of (complete) path algebras and at the level of Jacobian algebras.

Section \ref{sec:reduction} deals with a crucial technical issue, namely, the reduction of a non-necessarily 2-acyclic SP. An SP $(A,S)$ is \emph{reduced} (resp. \emph{trivial}) if $S$ involves only cycles of length at least 3 (resp. only cycles of length 2 and, moreover, the cyclic derivatives of $S$ span $A$ as an $R$-$R$-bimodule). Note that the underlying quiver of a reduced SP is not necessarily 2-acyclic. For a general potential $S$ (not necessarily reduced), the arrows that appear in the degree-2 component of $S$ may appear in its higher-degree components, and this constitutes an obstacle to delete 2-cycles algebraically from an SP. The Splitting Theorem \ref{thm:trivial-reduced-splitting} states that every SP $(A,S)$ is right-equivalent to the direct sum of a reduced SP $(A_{\red},S_{\red})$ and a trivial SP $(A_{\triv},S_{\triv})$, and that the right-equivalence classes of such reduced and trivial SPs are uniquely determined by the right-equivalence class of $(A,S)$. Fortunately for us, the proof of \cite[Theorem 4.6]{DWZ1} can be taken, practically word by word, to show Theorem \ref{thm:trivial-reduced-splitting}, the crucial point being that the complete path algebra $\RA{A}$ is a \emph{$D$-algebra} (see \cite[Section 13]{DWZ1} for the definition of $D$-algebra). We should remark that the Splitting Theorem \ref{thm:trivial-reduced-splitting} is the reason why we work with complete path algebras instead of the more commonly used ``incomplete" path algebras. Indeed, the proof of the existence of a desired right-equivalence $(A,S)\rightarrow(A_{\red},S_{\red})\oplus(A_{\triv},S_{\triv})$ relies on a limit process, and the only way to ensure that such limit process converges is to work with complete path algebras. We close Section \ref{sec:reduction} with Theorem \ref{thm:reduction-is-regular-function}, which gives a determinantal criterion for the reduced part of an SP to be reduced. More specifically, for each maximal collection of disjoint 2-cycles of the underlying quiver $Q$ of $A$, we produce a polynomial $f$, given by a determinant, with the property that for any potential $S\in \RA{A}$ such that $f(S)\neq 0$, the reduced part $(A_{\red},S_{\red})$ is 2-acyclic. The main feature of this criterion is that, due to the considerations in Section \ref{sec:change-of-base-field}, the polynomial $f$ remains unchanged if we replace $F$ with a finite-degree extension $K$ of $F$ linearly disjoint from $E/F$. 

In Section \ref{sec:mutations} we define \emph{mutations of strongly primitive species with potentials} and show that they have the basic properties one should expect from a mutation: they take right-equivalent SPs to right-equivalent SPs, and they are involutive up to right-equivalence. It is in order to be able to define SP-mutations that we need the existence of reduced parts of SPs and that reduced parts be well-defined up to right-equivalence.

The delicate issue of existence of nondegenerate SPs is dealt with in Section \ref{sec:nondegeneracy}, where we focus in the situation where the arrow span $A$ of a strongly primitive weighted quiver $(Q,\dtuple)$ is defined over an extension $E/F$ that is either an extension of finite fields or an unramified extension of non-archimedian completions of finite-degree extensions of $\mathbb{Q}$. In the latter case, we show that for every 2-acyclic strongly primitive weighted quiver $(Q,\dtuple)$, the arrow span of $(Q,\dtuple)$ over $E/F$ admits a non-degenerate potential. In the case when $E/F$ is an extension of finite fields, we show that for any finite sequence $(k_1,\ldots,k_\ell)$ of vertices of a 2-acyclic strongly primitive weighted quiver $(Q,\dtuple)$, there exists a finite-degree extension $K$ of $F$, linearly disjoint from $E/F$, such that the arrow span of $(Q,\dtuple)$ over the extension $KE/K$, call it $A_{KE/K}$, admits a potential $S$ such that all the SPs $\mu_{k_1}(A_{KE/K},S)$, $\mu_{k_2}\mu_{k_1}(A_{KE/K},S)$, $\ldots$, $\mu_\ell\ldots\mu_{k_1}(A_{KE/K},S)$, are 2-acyclic.

In Section \ref{sec:mutation-invariants} we show that finite-dimensionality of Jacobian algebras is invariant under SP-mutations.

Decorated representations and their mutations are introduced and studied in Section \ref{sec:reps-and-repmutations}. We show that mutations of decorated representations have the basic properties one should expect: they send right-equivalent decorated representations to right-equivalent ones, and they are involutive up to right-equivalence.

In Section \ref{sec:Dlab-Ringel} we show that mutations of decorated representations at sinks or sources constitute a decorated version of the reflection functors defined and studied by Dlab-Ringel in \cite{DR}. Dlab-Ringel reflection functors are defined at sinks and sources of a species, whereas mutations of decorated representations are defined here with respect to arbitrary vertices of an SP. Hence, the mutations of species with potentials we have defined here constitute a generalization of Dlab-Ringel reflection functors which is analogous to the generalization of Bernstein-Gelfand-Ponomarev reflection functors provided by the mutations of quivers with potentials in the simply laced case. Furthermore, the way we generalize the QP-mutation theory of \cite{DWZ1} is analogous to the way that Dlab-Ringel reflection functors generalize Bernstein-Gelfand-Ponomarev reflection functors; in this sense, we Dlab-Ringel's work \cite{DR} has served as a motivating source to our approach.

In Section \ref{sec:unfoldings} we show that the class of strongly primitive skew-symmetrizable matrices contains infinitely many instances of skew-symmetrizable matrices that do not admit global unfoldings and hence cannot be unfolded from the cluster algebra perspective. Roughly speaking, saying that a skew-symmetrizable matrix $B$ does not admit a global unfolding means that such that any unfolding of $B$ will be incompatible with mutations inside cluster algebras.

Finally, in Section \ref{sec:other-developments}, we mention previous approaches by other authors to non--skew-symmetric cluster algebras via representations of quivers or species.

A note to the reader familiar with the mutation theory of quivers with potentials. The influence of \cite{DWZ1} on this paper should be evident. Indeed, we have made a conscious effort to follow as much as possible the guidelines provided by \cite{DWZ1}, to the extent that for several statements, rather than transcribing arguments without change, we refer the reader to \cite{DWZ1} to find proofs that need either minor or no modifications whatsoever.
Nevertheless, many of our results require either constructions that are plainly different from the corresponding ones in \cite{DWZ1}, or non-trivial refinements of arguments from \cite{DWZ1}.
None of these two facts should surprise the reader: on the one hand, several arguments and techniques from \cite{DWZ1} can be applied in our setup with minor or no further modifications. On the other hand, not all the tools available in the skew-symmetric case can be applied in the skew-symmetrizable case as are or with only slight modifications; indeed, the passage from simply laced to non-simply laced situations often becomes rather subtle and requires non-trivial refinements of arguments and constructs, or independent observations.


\section*{Acknowledgements}

The collaborations \cite{DWZ1,DWZ2} of Harm Derksen and Jerzy Weyman with the second author have served as an inspiring and guiding
source throughout the development of the present manuscript. We have made a conscious effort to follow \cite{DWZ1} as much as possible, both in the structure of the paper, and in the form and proofs of the results.

We thank Laurent Demonet, Claus Michael Ringel and Dylan Rupel for helpful discussions. We are particularly grateful to David Speyer for his collaboration \cite{SZ} with the second author that showed the existence of skew-symmetrizable matrices without global unfoldings.

The first author thanks the Representation theory group of the National Autonomous University of Mexico, especially Michael Barot and Christof Geiss, for their hospitality and financial support during his several visits to UNAM's Institute of Mathematics in Mexico City. Parts of the work presented here were completed in these visits.

This project started several years ago, when the first author was a Ph.D. student of the second author at Northeastern University's Department of Mathematics (Boston, MA, USA). Unfortunately, Professor Andrei Zelevinsky passed away during the last stage of the preparation of this paper. The first author is sincerely grateful to him for his enthusiasm and deep insights and contributions throughout the collaboration that led to the present paper. It is needless to say that any possible inaccuracies or mistakes should be attributed to the first author alone.



\section{The weighted quiver of a skew-symmetrizable matrix}\label{sec:weighted-quivers}

\begin{defi}
An $n\times n$ matrix $B$ with integer entries is called \emph{skew-symmetrizable} if there exist positive integers $d_1,\ldots,d_n$, such that the matrix $DB$ is skew-symmetric, where $D=\diag(d_1,\ldots,d_n)$. The diagonal matrix $D$ is then called a \emph{skew-symmetrizer} of $B$.
If $B$ admits a skew-symmetrizer $D=\diag(d_1,\ldots,d_n)$ such that $\gcd(d_i,d_j)=1$ for all $i\neq j$, we say that $B$ is a \emph{strongly primitive skew-symmetrizable matrix}.
\end{defi}

\begin{defi} A 
\emph{weighted quiver} is a pair $(Q,\dtuple)$, where $Q$ is a loop-free quiver and
$\dtuple=(d_i)_{i\in Q_0}$ is a tuple that assigns a positive integer $d_i$ to each vertex $i$ of $Q$. We will often refer to $\dtuple$ as a \emph{weight tuple}.
If the weight tuple $\dtuple$ consists entirely of pairwise coprime integers, we say that $(Q,\dtuple)$ is a \emph{strongly primitive weighted quiver}.
\end{defi}

Let $B$ be an $n\times n$ skew-symmetrizable integer matrix, and fix a skew-symmetrizer $D$ of $B$. We associate a weighted quiver $(Q,\dtuple)=(Q_B,\dtuple)$ to $B$ as follows. First, define a matrix $C$ given by the entries
\begin{equation}
c_{ij}=\frac{\gcd(d_i,d_j)b_{ij}}{d_j}.
\end{equation}
A straightforward check shows that $C$ is a skew-symmetric integer matrix. We set $Q=Q_B$ to be the quiver whose vertex is
$Q_0=\{1,\ldots,n\}$, and whose arrow set is given by placing exactly $c_{ij}$ arrows from $j$ to $i$ whenever $c_{ij}\geq 0$. We set $\dtuple$ to be the tuple consisting of the diagonal entries of $D$, that is, the positive integer attached to a vertex $i$ of $Q$ is defined to be the $i^{\operatorname{th}}$ diagonal entry of $D$. Note that the quiver $Q=Q_B$ is uniquely determined by $B$, while the weight tuple $\dtuple$ is uniquely determined by $D$.

The following lemma is a straightforward consequence of the definitions.

\begin{lemma}\label{lemma:B<->Q} Fix an $n$-tuple $\dtuple=(d_1,\ldots,d_n)$ of positive integers, and let $D=\diag(d_1,\ldots,d_n)$. The assignment $B\mapsto (Q_B,\dtuple)$ we have just described
is a bijection between the set of $n\times n$ integer matrices that can be skew-symmetrized by $D$, and the set of 2-acyclic weighted quivers
on the vertex
set $Q_0=\{1,\ldots,n\}$, and with weight tuple $\dtuple$.
\end{lemma}

In particular, strongly primitive skew-symmetrizable matrices are in one-to-one correspondence with strongly primitive 2-acyclic weighted quivers. 

Let us recall the operation of \emph{matrix mutation}, the main ingredient
in the definition of \emph{cluster algebras} (cf. \cite{FZ1},
\cite{FZ2}, \cite{BFZ3}, \cite{FZ4}).

\begin{defi}\label{def:matrixmutation} Let $B$ be an
$n\times n$ skew-symmetrizable matrix. For $k\in\{1,\ldots,n\}$, the
\emph{mutation of $B$ with respect to $k$} is the matrix $B'=\mu_k(B)$ defined
by
$$
b'_{ij}=\begin{cases}
  -b_{ij} & \text{if} \ i=k \ \text{or} \ k=j,\\
  b_{ij}+\frac{b_{ik}|b_{kj}|+|b_{ik}|b_{kj}}{2} & \text{if} \ i\neq k\neq j.
\end{cases}
$$
\end{defi}

By Lemma \ref{lemma:B<->Q}, it is possible to define mutations of
2-acyclic weighted quivers by simply translating Definition
\ref{def:matrixmutation}. 

\begin{defi}\label{def:weightedquivermutation} Let $(Q,\dtuple)$ be a
weighted quiver and $k$ a vertex of $Q$. Define the \emph{mutation of $(Q,\dtuple)$ with respect to $k$} as the weighted quiver $\mu_k(Q,\dtuple)$ on the vertex set
$Q_0$, and with the same weight tuple $\dtuple$, as the result of performing the following 3-step procedure:
\begin{itemize}
\item[(Step 1)] For each pair $a,b\in Q_1$ such that $h(b)=k=t(a)$, introduce
$\frac{\gcd(d_i,d_j)d_k}{\gcd(d_i,d_k)\gcd(d_k,d_j)}$ ``composite'' arrows from $t(b)$ to
$h(a)$;
\item[(Step 2)] replace each $a\in Q_1$ incident to $k$ with an arrow $a^*$ in the
direction oposite to that of $a$;
\item[(Step 3)] choose a maximal collection of disjoint 2-cycles and remove them.
\end{itemize}
The underlying quiver of $\mu_k(Q,\dtuple)$ will be denoted by $\mu_k(Q)$.
\end{defi}

\begin{remark} \begin{itemize}\item The integer $\gcd(d_i,d_k,d_j)=\gcd(\gcd(d_i,d_k),\gcd(d_k,d_j))$ certainly divides $\gcd(d_i,d_j)$, whereas the least common multiple of the integers $\gcd(d_i,d_k)$ and $\gcd(d_k,d_j)$ obviously divides $d_k$. Since $\gcd(d_i,d_k)\gcd(d_k,d_j)=\gcd(d_i,d_k,d_j)\lcm[\gcd(d_i,d_k),\gcd(d_k,d_j)]$, this means that the rational number $\frac{\gcd(d_i,d_j)d_k}{\gcd(d_i,d_k)\gcd(d_k,d_j)}$ is indeed an integer.
\item The underlying quiver $\mu_k(Q)$ of the weighted quiver $\mu_k(Q,\dtuple)$ depends both on $Q$ and $\dtuple$. Since the weight tuple $\dtuple$ will always be fixed throughout the paper, the fact that the dependence on $\dtuple$ is not apparent in the notation $\mu_k(Q)$ should not be a cause of confusion.
\item If $\dtuple=(1,\ldots,1)$, the bijection from Lemma \ref{lemma:B<->Q} turns out to be the well-known bijection between the set of skew-symmetric matrices and the set of 2-acyclic quivers. Moreover, in this case the underlying quiver of $\mu_k(Q,\dtuple)$ coincides with the quiver obtained from $Q$ by \emph{ordinary quiver mutation}.
\item It is not true in general that the underlying quiver of $\mu_k(Q,\dtuple)$ always coincides with the quiver obtained from $Q$ by ordinary quiver mutation.
\item Up to isomorphisms of weighted quivers, $\mu_k$ is an involution on the class of 2-acyclic weighted quivers, that is,
$\mu_k^2(Q,\dtuple)\cong (Q,\dtuple)$, in spite of the fact that the choice of a maximal collection of
disjoint 2-cycles is not canonical.
\end{itemize}
\end{remark}

The next lemma is an easy consequence of the definitions.

\begin{lemma}\label{lemma:muk(B)<->muk(Q,d)} Let $B$ be a skew-symmetrizable matrix, and let $D$ be a skew-symmetrizer of $B$. Then  $\mu_k(Q_B,\dtuple)$ and $(Q_{\mu_k(B)},\dtuple)$ are isomorphic as weighted quivers.
\end{lemma}

\begin{ex}\label{ex:1213-cycle} The matrix
$$
B=\left[\begin{array}{cccc}0 & -2 & 0  & 3 \\
                                                                       1 & 0  & -1 & 0\\
                                                                       0 & 2  & 0  & -3\\
                                                                        -1& 0 & 1 & 0\end{array}\right]
$$
is skew-symmetrizable, and $D=\diag(1,2,1,3)$ is a skew-symmetrizer of $B$. Hence, $B$ is strongly primitive. The (strongly primitive) weighted quiver associated to $B$ is
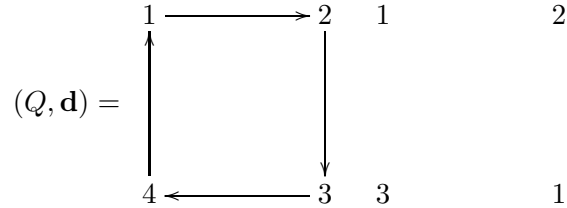
\begin{center}
$(Q,\dtuple)=$\begin{tabular}{cc}
$\xymatrix{1  \ar[rr]
& & 2 \ar[dd]
\\
 & & \\
4 \ar[uu]
& & 3 \ar[ll]
}$ &
$\xymatrix{1  & & 2 \\
 & & \\
3 & & 1 }$
\end{tabular}
\end{center}
If we respectively perform the matrix mutation $\mu_4$, and the weighted quiver mutation $\mu_4$, we obtain
$$
\mu_4(B)=\left[\begin{array}{cccc}0 & -2 & 3  & -3 \\
                                  1 & 0  & -1 & 0\\
                                  -3 & 2  & 0  & 3\\
                                  1& 0 & -1 & 0
\end{array}\right] \ \ \ \text{and} \ \ \
\text{$\mu_k(Q,\dtuple)=$\begin{tabular}{cc}
$$\xymatrix{1 \ar[rr] \ar[dd] & & 2 \ar[dd]\\
 & & \\
4 \ar[rr]& & 3 \ar[uull] \ar@<-1.0ex>[uull] \ar@<1.0ex>[uull]}$$
 & $\xymatrix{1  & & 2 \\
 & & \\
3 & & 1 }$
\end{tabular}
}
$$
It is readily seen that, as stated in Lemma \ref{lemma:muk(B)<->muk(Q,d)}, one can read $\mu_4(B)$ and $\mu_4(Q,\dtuple)$ off from each other via the bijection of Lemma \ref{lemma:B<->Q}.
\end{ex}



\section{Assumptions}\label{sec:assumptions}


Let $(Q,\dtuple)$ be a weighted quiver. From this point to the end Section \ref{sec:Dlab-Ringel}, we will permanently suppose, even without stating it explicitly, that $(Q,\dtuple)$ is a strongly primitive weighted quiver. In other words, that the tuple $\dtuple=(d_i)_{i\in Q_0}$ satisfies
\begin{equation}\label{eq:coprime-assumption}
\gcd(d_i,d_j)=1 \ \text{for all} \ i\neq j.
\end{equation}

Denote by $d$ the least common multiple of the integers conforming the tuple $\dtuple$. Throughout the whole paper we will suppose that
\begin{equation}\label{eq:d-root-of-unity}
\text{$F$ is a field containing a primitive $d^{\operatorname{th}}$ root of unity,}
\end{equation}
and that
\begin{equation}\label{eq:E/F-Galois-simple-spectrum}
\text{$E$ is a degree-$d$ cyclic Galois field extension of $F$.}
\end{equation}

The assumptions \eqref{eq:d-root-of-unity} and \eqref{eq:E/F-Galois-simple-spectrum} imply that
\begin{equation}\label{eq:Fi/F=degree-di-extension}
\text{for each $i\in Q_0$ there exists a unique degree-$d_i$ cyclic Galois subextension $F_i/F$ of $E/F$,}
\end{equation}
and that
\begin{equation}\label{eq:eigenbasis}
\text{there exists an element $v\in E$ such that $\B_E=\{1,v,\ldots,v^{d-1}\}$ is an eigenbasis of $E/F$,}
\end{equation}
that is, an $F$-vector space basis of $E$ consisting of eigenvectors of all elements of the Galois group $\Gal(E/F)$ (note that the eigenvalues are the $d^{\operatorname{th}}$ roots of unity, which lie inside $F$ by assumption).

Setting $v_i=v^{\frac{d}{d_i}}$ for $i\in Q_0$, we have that
\begin{equation}\label{eq:compatible-eigenbases-Bi}
\text{$\B_i=\{1,v_i,\ldots,v_i^{d_i-1}\}=\B_E\cap F_i$ is an eigenbasis of $F_i/F$.}
\end{equation}

Note that there are functions $f:\B_E\times\B_E\rightarrow F^\times$ and $m:\B_E\times\B_E\rightarrow\B_E$ such that
\begin{equation}\label{eq:mult-property-of-eigenbases}
uu'=f(u,u')m(u,u') \ \ \text{for all} \ u,u'\in\B_E.
\end{equation}
Moreover, $m(u,u')\in\B_i$ for $u,u'\in\B_i$.

\begin{remark}\label{rem:specific-choice-of-eigenbases} Althought it will be assumed throughout the paper that the eigenbases $\B_E$ and $\B_i$ ($i\in Q_0$) are given by \eqref{eq:eigenbasis} and \eqref{eq:compatible-eigenbases-Bi} for a specific $v\in E$ (so that, in particular, they all contain the element $1\in F$), our statements and constructions will be made in more invariant terms whenever possible. So, only when needed will we make reference to the specific way the eigenbases have been chosen.
\end{remark}

\begin{ex}\label{ex:finite-fields-satisfy} Let $p$ be a positive prime number congruent to $1$ modulo $d$ (by the famous theorem of Dirichlet on primes in arithmetic progressions, there are infinitely many such positive primes). If $F$ is a finite field of characteristic $p$, then $F$ contains a primitive $d^{\operatorname{th}}$ root of unity, and the unique degree-$d$ extension $E$ of $F$ (inside any \emph{a priori} fixed algebraic closure of $F$) is a cyclic Galois extension of $F$.
\end{ex}

\begin{ex}\label{ex:p-adic-fields-satisfy} Let $p$ be a positive prime number congruent to $1$ modulo $d$.
\begin{enumerate}
\item If $F=\mathbb{Q}_p$, the field of $p$-adic numbers, then $F$ contains a primitive $d^{\operatorname{th}}$ root of unity, and the unique degree-$d$ unramified extension $E$ of $F$ (inside any \emph{a priori} fixed algebraic closure of $F$) is a cyclic Galois extension of $F$. Note that $\mathbb{Q}_p$ is uncountable.
\item More generally, if $L$ is any finite-degree extension of the field $\mathbb{Q}$ of rational numbers, $\mathbb{O}_L$ is the ring of algebraic integers of $L$, $\mathfrak{p}$ is a prime ideal of $\mathbb{O}_L$ such that $p\in\mathfrak{p}$, and $F$ is the $\mathfrak{p}$-adic completion of $L$, then $F$ contains a primitive $d^{\operatorname{th}}$ root of unity, and the unique degree-$d$ unramified extension $E$ of $F$ (inside any \emph{a priori} fixed algebraic closure of $F$) is a cyclic Galois extension of $F$. 
    Note that, in this case, the field $F$ is uncountable.
\end{enumerate}
\end{ex}

\begin{ex}\label{ex:linealy-disjoint} If $E/F$ is a degree-$d$ cyclic Galois extension, with $F$ containing a primitive $d^{\operatorname{th}}$ root of unity, and $K/F$ is a finite-degree extension which is linearly disjoint from $E/F$, then $KE/K$ is again a degree-$d$ cyclic Galois extension, and $K$ certainly contains a primitive $d^{\operatorname{th}}$ root of unity. Furthermore, if $\B_E$ is an eigenbasis of $E/F$, then it is an eigenbasis of $KE/K$ as well. (In this example, $E$ and $K$ are assumed to be contained in an \emph{a priori} fixed algebraic closure $\overline{F}$ of $F$; $KE$ then denotes the \emph{compositum} of $K$ and $E$ inside $\overline{F}$, that is, the minimal subfield of $\overline{F}$ that simultaneously contains $K$ and $E$. Finally, we remind the reader that $K/F$ being \emph{linearly disjoint} from $E/F$ means, by definition, that the multiplication map $K\otimes_FE\to KE$ is injective or, equivalently, that the intersection $K\cap E$ is equal to $F$.)
\end{ex}



\section{Path algebras and complete path algebras}\label{sec:path-algs}


Let $(Q,\dtuple)$ be a weighted quiver satisfying \eqref{eq:coprime-assumption}, and let $F$, $E$, $v\in E$, $F_i$ and $v_i\in F_i$ ($i\in Q_0$) be as in \eqref{eq:E/F-Galois-simple-spectrum}, \eqref{eq:Fi/F=degree-di-extension}, \eqref{eq:eigenbasis} and \eqref{eq:compatible-eigenbases-Bi}. Define
\begin{equation}\label{eq:R=semisimple-ring}
R=\bigoplus_{i\in Q_0}F_i,
\end{equation}
which is a semisimple $F$-algebra, we call it the \emph{vertex span} of $(Q,\dtuple)$ over the extension $E/F$. For $i,j\in Q_0$ let
\begin{equation}\label{eq:Aij}
A_{ij}=\bigoplus_{a:j\to i}F_i\otimes_FF_j,
\end{equation}
which is obviously an $F_i$-$F_j$-bimodule. Set also
\begin{equation}
A=\bigoplus_{i,j\in Q_0}A_{ij}.
\end{equation}
Then $A$ is an $R$-$R$-bimodule, which we will call \emph{arrow span} of $(Q,\dtuple)$ over the extension $E/F$.
Following a standard convention, we identify each arrow $a\in Q_1$ with the element $1\otimes 1$ of the component corresponding to $a$ in \eqref{eq:Aij}.

\begin{defi}\label{def:path-algebra}\begin{itemize}\item The \emph{path algebra} of $(Q,\dtuple)$ over $E/F$, to be denoted by $R\langle A\rangle$, is the tensor algebra of $A$ over $R$. Thus, as an $R$-$R$-bimodule we have
\begin{equation}
R\langle A\rangle=\bigoplus_{\ell\geq 0} A^\ell,
\end{equation}
where $A^\ell$ denotes the $\ell$-fold tensor product $A\otimes_R \ldots\otimes_RA$ (as it is customary, $A^0=R$ and $A^1=A$).
\item The \emph{complete path algebra} of $(Q,\dtuple)$ over $E/F$, to be denoted by $\RA{A}$, is the complete tensor algebra of $A$ over $R$. Thus, as an $R$-$R$-bimodule we have
\begin{equation}
\RA{A}=\prod_{\ell\geq 0} A^\ell.
\end{equation}
\end{itemize}
If the dependence on the extension $E/F$ needs to be emphasized, we shall write $R_{E/F}=R$ and $A_{E/F}=A$.
\end{defi}


\begin{remark} Even though the action of
$R$ on $\completeQdcoprime$ is not central, it is compatible with
the multiplication of $\RA{A}$, in the sense that if $a$ and $b$ are paths in Q, then $e_{h(a)}ab =
ae_{t(a)}b = abe_{t(b)}$, where, for $i\in Q_0$, $e_i$ is the idempotent sitting in the $i^{\operatorname{th}}$ component of \eqref{eq:R=semisimple-ring}. We will thus say that $\pathalgQdcoprime$
and $\completeQdcoprime$ are \emph{$R$-algebras}. Accordingly, any $F$-algebra homomorphism $\varphi$ between (complete) path algebras will be said to be an \emph{$R$-algebra homomorphism} if the underlying quivers have the same set of vertices and the same weight tuple $\dtuple$, and $\varphi(r)=r$ for every $r\in R$.
\end{remark}

\begin{defi}\label{def:path} A \emph{path of length $\ell$} on $(Q,\dtuple)$ over $E/F$ is an element $\omega_0a_1\omega_1a_2\ldots \omega_{\ell-1}a_\ell\omega_\ell\in\completeQdcoprime$, where
\begin{itemize}
\item $a_1,\ldots,a_\ell$, are arrows of $Q$ such that $h(a_{r+1})=t(a_r)$ for $r=1,\ldots,\ell-1$;
\item $\omega_0\in\B_{h(a_1)}$ and $\omega_r\in\B_{t(a_r)}$ for $r=1,\ldots,\ell$.
\end{itemize}
Here we are assuming that the eigenbasis $\B_E$, and hence the eigenbases $\B_i$ for $i\in Q_0$ (see \eqref{eq:compatible-eigenbases-Bi}), have been \emph{a priori} fixed.
\end{defi}

Note that a path of length $0$ is just an element $\omega$ of an eigenbasis $\B_i$ (hence there are $d_i$ paths of length $0$ sitting at each vertex $i$ of $Q$). Note also that while every arrow is a path of length $1$, not every path of length $1$ is an arrow. (That every arrow is a path of length $1$ follows from the fact that the element $1\in F$ belongs to each one of the eigenbases $\B_i$).

\begin{remark}\label{rem:path-is-welldef-notion} The notion of path on $(Q,\dtuple)$ introduced in Definition \ref{def:path} does not depend on $(Q,\dtuple)$ alone, but also on the extension $E/F$ and the choice of the eigenbasis $\B_E$ (from which the eigenbases $\B_i$ are obtained). Two observations are worth making:
\begin{enumerate}
\item Given $E/F$ and $\B_E$, Definition \ref{def:path} gives us a notion of path on $(Q,\dtuple)$. If $K/F$ is a finite-degree field extension which is linearly disjoint from $E/F$, then $\B_E$ is an eigenbasis of $KE/K$ as well, and hence the notion of path on $(Q,\dtuple)$ over $E/F$ coincides with the notion of path on $(Q,\dtuple)$ over $KE/K$, provided we use the same eigenbasis $\B_E$ for both field extensions. In other words, once $\B_E$ is fixed, the notion of path on $(Q,\dtuple)$ is independent of $E/F$, in the sense that it does not change if we replace $E/F$ with $KE/K$ given any extension $K/F$ linearly disjoint from $E/F$.
\item For a fixed extension $E/F$, if $\B_E$ and $\B_E'$ are eigenbases of $E/F$, and $\B_i=\B_E\cap F_i$ and $\B_i'=\B_E'\cap F_i$ are the corresponding eigenbases of $F_i/F$ for $i\in Q_0$, then there exists a bijection $\pi:\B_E\rightarrow\B_E'$ with the following two properties: its restriction to each $\B_i$ is a bijection $\B_i\rightarrow\B_i'$, and $\pi(\omega)$ is an $F$-multiple of $\omega$ for every $\omega\in\B_E$. Hence, for $E/F$ fixed, the notion of path on $(Q,\dtuple)$ is independent of the choice of eigenbasis $\B_E$ up to multiplication by non-zero elements of $F$.
\end{enumerate}
\end{remark}

The proof of the following lemma is a straightforward exercise.

\begin{lemma}\label{lemma:paths-form-a-basis} The set of all length-$\ell$ paths on $(Q,\dtuple)$ constitutes a basis of $A^\ell$ as an $F$-vector space. Consequently, every element of the complete path algebra $\RA{A}$ can be expressed in a unique way as a possibly infinite $F$-linear combination of paths.
\end{lemma}

\begin{remark}
The product of two (concatenable) paths is not necessarily a path, but an $F$-multiple of a path. More precisely, the product $(\omega_0a_1\ldots a_\ell\omega_\ell)\cdot(\varpi_0b_1\ldots b_l\varpi_l)$ is equal to
    $$
\begin{cases}
f(\omega_\ell,\varpi_0)\omega_0a_1\omega_1\ldots a_\ell m(\omega_\ell,\varpi_0)b_1\varpi_1\ldots b_l\varpi_l & \textrm{if $t(a_\ell)=h(b_1)$;}\\
0 & \textrm{otherwise};
\end{cases}
$$
where $f:\B_E\times\B_E\rightarrow F^\times$ and $m:\B_E\times\B_E\rightarrow\B_E$ are the functions satisfying \eqref{eq:mult-property-of-eigenbases}.
\end{remark}

\begin{ex}\label{ex:1213-cycle-species} Consider the weighted quiver $(Q,\dtuple)$ from Example \ref{ex:1213-cycle}. Since $\lcm(1,2,1,3)=6$, we have $[E:F]=6$, $F_1=F=F_3$, $[F_2:F]=2$ and $[F_4:F]=3$ (note that here, just as in \eqref{eq:Fi/F=degree-di-extension}, the subindex $i$ in $F_i$ does not refer to the degree of the extension $F_i/F$, but rather to the vertex $i$ of $Q$ to which $F_i$ is attached). Furthermore, the semisimple algebra $R$ and the bimodules $A_{ij}$ that comprise the arrow span $A$ can be visualized as follows
$$
\xymatrix{F  \ar[rr]^{F_2\otimes F}_{\delta}
& & F_2 \ar[dd]^{F\otimes F_2}_{\gamma}
\\
 & & \\
F_4 \ar[uu]^{F\otimes F_4}_{\alpha}
& & F \ar[ll]^{F_4\otimes F}_{\beta}
}
$$
where all tensor products are taken over $F$. Take an eigenbasis $\B_E=\{1,v,v^2,v^3,v^4,v^5\}$ of $E=F_2F_4$ over $F$, so that $\B_2=\{1,v^3\}$ and $\B_4=\{1,v^2,v^4\}$ are eigenbases of $F_2$ and $F_4$ over $F$, respectively. Then for each $\ell\geq 1$ the set
$$
\{\alpha v^{2m_1}\beta\gamma v^{3n_1}\delta\alpha v^{2m_2}\beta\gamma v^{3n_2}\delta\ldots\alpha v^{2m_{\ell}}\beta\gamma v^{3n_{\ell}}\delta\suchthat m_1,\ldots,m_\ell\in\{0,1,2\},n_1,\ldots,n_\ell\in\{0,1\}\}
$$
is an $F$-basis of the space of all length-$4\ell$ paths starting and ending at vertex $1\in Q_0=\{1,2,3,4\}$.
\end{ex}

Let ${\mathfrak m} = {\mathfrak m}(A)$ denote the (two-sided) ideal
of $\completeQdcoprime$ given by
\begin{equation}\label{eq:max-ideal}
{\mathfrak m}  = {\mathfrak m}(A) = \prod_{\ell=1}^\infty A^\ell.
\end{equation}
Thus the powers of~${\mathfrak m}$ are given by
$$
{\mathfrak m}^n = \prod_{\ell=n}^\infty A^\ell.
$$
We view $\completeQdcoprime$ as a topological $F$-algebra
via the \emph{${\mathfrak m}$-adic topology} having
the powers of ${\mathfrak m}$ as a basic system of open
neighborhoods of~$0$.
Thus, the closure of any subset $W \subseteq \completeQdcoprime$
is given by
\begin{equation}\label{eq:closure}
\overline W = \bigcap_{n=0}^\infty (W + {\mathfrak m}^n).
\end{equation}
It is clear that $\pathalgQdcoprime$ is a  dense subalgebra
of $\completeQdcoprime$.

The ideal $\maxid$ satisfies the basic properties one would expect (cf. \cite[Section 2]{DWZ1}). For instance, $\maxid$ is
maximal amongst the two-sided ideals of $\completeQdcoprime$
that have zero intersection with $R$.
Moreover, $\maxid$
is invariant under any $R$-algebra automorphism of $\completeQdcoprime$.
Thus, such an automorphism is continuous as a map between topological spaces.
More generally, if $(Q,\dtuple)$ and $(Q',\dtuple)$ are weighted quivers on the same vertex set and with the same weight function $\dtuple$, then any $R$-algebra homomorphism $\varphi: \completeQdcoprime\rightarrow \RA{A'}$ sends $\maxid=\maxid(A)$ into $\maxid'=\maxid(A)$, and is hence
continuous.
Furthermore, such a ~$\varphi$ is uniquely determined by its
restriction to $A$, which is an $R$-$R$-bimodule homomorphism
$A \to \maxid' = A' \oplus (\maxid')^2$.
We write $\varphi|_{A} = (\varphi^{(1)}, \varphi^{(2)})$, where
$\varphi^{(1)}:A \to A'$ and $\varphi^{(2)}:A \to (\maxid')^2$ are
$R$-$R$-bimodule homomorphisms.

\begin{prop} \label{prop:automorphisms}
Any pair $(\varphi^{(1)}, \varphi^{(2)})$
of $R$-$R$-bimodule homomorphisms
$\varphi^{(1)}:A \to A'$ and $\varphi^{(2)}:A \to (\maxid')^2$
gives rise to a unique continuous $R$-algebra homomorphism
$\varphi: \completeQdcoprime \to \RA {A'}$
such that $\varphi|_{A} = (\varphi^{(1)}, \varphi^{(2)})$.
Furthermore, $\varphi$ is an isomorphism
if and only if $\varphi^{(1)}$ is an $R$-$R$-bimodule isomorphism.
\end{prop}

\begin{proof}
A suitable but minor modification of the proof of Proposition 2.4 of \cite{DWZ1} applies here.
\end{proof}

\begin{defi}\label{def:automorphisms}
Let $\varphi$ be an automorphism of $\completeQdcoprime$, and let
$(\varphi^{(1)}, \varphi^{(2)})$ be the corresponding pair of $R$-$R$-bimodule homomorphisms.
If~$\varphi^{(2)} = 0$, then we call $\varphi$ a \emph{change of arrows}.
If $\varphi^{(1)}$ is the identity automorphism of~$A$, we say that~$\varphi$
is a \emph{unitriangular} automorphism; furthermore, we
say that $\varphi$ is of \emph{depth} $\delta \geq 1$, if $\varphi^{(2)}(A)
\subset \maxid^{\delta+1}$.
\end{defi}

The following property of unitriangular automorphisms is immediate
from the definitions:
\begin{align}
\label{eq:unitriangular-d}
&\text{If~$\varphi$ is an unitriangular automorphism of $\completeQdcoprime$ of depth~$\delta$,}\\
\nonumber
&\text{then
$\varphi(u) - u \in  \maxid^{n+\delta}$ for $u \in \maxid^{n}$.}
\end{align}


\section{Potentials and their cyclic derivatives}\label{sec:potentials-and-Jacobians}

In what follows we shall
introduce some of our main objects of study:
potentials and their cyclic derivatives in the corresponding
complete path algebras.

\begin{defi}[Potentials, cyclical equivalence, cyclic derivatives, Jacobian algebras]\
\label{def:cyclic-stuff}
\begin{itemize}
\item For each $\ell \geq 1$, we define the \emph{cyclic part} of $A^\ell$ to be
$A^\ell_{\rm cyc} = \bigoplus_{i \in Q_0} A^\ell_{i,i}$.
Thus, $A^\ell_{\rm cyc}$ is the $F$-span of all paths $\omega_0 a_1\omega_1 \cdots a_\ell\omega_\ell$
with $h(a_1) = t(a_d)$; we call such paths \emph{cyclic}.
\item We define a closed $F$-vector subspace
$\completeQdcoprime_{\rm cyc} \subseteq \completeQdcoprime$
by setting
$$
\completeQdcoprime_{\rm cyc} =
\prod_{\ell=1}^\infty A^\ell_{\rm cyc},
$$
and call the elements of $\completeQdcoprime_{\rm cyc}$
\emph{potentials}.
\item\label{item:cyclic-equivalence}
Two potentials $S$ and $S'$ are \emph{cyclically equivalent}
if $S - S'$ lies in the closure of the $F$-span of
all elements of the form $\omega_0 a_1\omega_1 a_2\omega_2  \cdots a_\ell\omega_\ell - \omega_1 a_2\omega_2  \cdots a_\ell\omega_\ell\omega_0a_1$,
where $\omega_0a_1\omega_1 a_2\omega_2  \cdots a_\ell\omega_\ell$ is a cyclic path on $(Q,\dtuple)$.
\item For each $a\in Q_1$, we define the \emph{cyclic
derivative} $\partial_a$ as the continuous $F$-linear map
$\completeQdcoprime_{\rm cyc} \to \completeQdcoprime$
acting on paths by
\begin{equation}
\label{eq:cyclic-derivative}
\partial_a (\omega_0 a_1\omega_1 a_2\cdots a_\ell\omega_\ell) =
\sum_{k=1}^\ell \delta_{a,a_k} \omega_k a_{k+1} \cdots a_\ell\omega_\ell\omega_0 a_1 \cdots a_{k-1}\omega_{k-1},
\end{equation}
where $\delta_{a,a_k}$ is the \emph{Kronecker delta} between $a$ and $a_k$.
\item For every potential~$S$, we
define its \emph{Jacobian ideal} $J(S)$ as the closure of
the (two-sided) ideal in $\completeQdcoprime$
generated by the elements $\partial_a(S)$ for all $a\in Q_1$
(see \eqref{eq:closure}); clearly, $J(S)$ is a two-sided
ideal in $\completeQdcoprime$.
\item We call the quotient $\completeQdcoprime/J(S)$
the \emph{Jacobian algebra} of~$S$, and denote
it by ${\mathcal P}(A,S)$.
\end{itemize}
\end{defi}

\begin{remark} Let $\xi=\omega_0a_1\omega_1\ldots a_\ell\omega_\ell$ be a cyclic path. It is easy to see that Definition \ref{def:cyclic-stuff} implies that $\xi$ is cyclically equivalent to $a_1\omega_1\ldots a_\ell\omega_\ell\omega_0$. Actually, $u\xi$ is cyclically equivalent to $\xi u$ for any $u\in F_{h(a_1)}$.
\end{remark}

\begin{ex}\label{ex:1213-cycle-potential} Let $(Q,\dtuple)$ and $A=A_{E/F}$ be as in Examples \ref{ex:1213-cycle} and \ref{ex:1213-cycle-species}.
Then $S=abcd-av^2bcv^3d$ is a potential on $A$, and its cyclic derivatives are
$$
\partial_a(S)=bcd-v^2bcv^3d, \ \ \partial_b(S)=cda-bcv^3dav^2, \ \ \partial_c(S)=dab-v^3dav^2b, \ \ \partial_d(S)=abc-av^2bcv^3.
$$
\end{ex}

\begin{prop}
\label{prop:cyclic-equivalence}
If two potentials $S$ and $S'$ are cyclically equivalent,
then $\partial_a (S) = \partial_a(S')$ for all $a\in Q_1$, hence $J(S) = J(S')$ and ${\mathcal P}(A,S) = {\mathcal P}(A,S')$.
\end{prop}

\begin{proof}
It suffices to show that $\partial_a(\omega_0 a_1 \cdots a_\ell\omega_\ell)=\partial_a(\omega_1a_2 \cdots a_\ell\omega_\ell\omega_0a_1)$. We have $\omega_1a_2 \cdots a_\ell\omega_\ell\omega_0a_1=f(\omega_\ell\omega_0)\omega_1a_2 \cdots a_\ell m(\omega_\ell,\omega_0)a_1$. Then, by definition,
$$
\partial_a (\omega_0 a_1\omega_1 a_2\cdots a_\ell\omega_\ell) =
\sum_{k=1}^\ell \delta_{a,a_k}\omega_k a_{k+1} \cdots a_\ell\omega_\ell\omega_0a_1\omega_1 \cdots \omega_{k-2}a_{k-1}\omega_{k-1}=
$$
$$
f(\omega_\ell,\omega_0)\sum_{k=1}^\ell \delta_{a,a_k}\omega_k a_{k+1} \cdots a_\ell m(\omega_\ell,\omega_0)a_1\omega_1 \cdots \omega_{k-2}a_{k-1}\omega_{k-1}=
$$
$$
f(\omega_\ell\omega_0)\partial_a(\omega_1a_2 \cdots a_\ell m(\omega_\ell,\omega_0)a_1)=\partial_a(\omega_1a_2 \cdots a_\ell\omega_\ell\omega_0a_1),
$$
where $f:\B_E\times\B_E\rightarrow F^\times$ and $m:\B_E\times\B_E\rightarrow\B_E$ are the functions satisfying \eqref{eq:mult-property-of-eigenbases}.
\end{proof}

Let $(Q,\dtuple)$ and $(Q',\dtuple)$ be weighted quivers with arrow spans $A$ and $A'$. Clearly, every $R$-algebra homomorphism
$\varphi: \completeQdcoprime \to \completeQpdcoprime$, sends potentials to potentials. Thus it makes sense to ask what
is the relation between $J(\varphi(S))$ and $\varphi(J(S))$ for a given potential $S$.

\begin{prop}
\label{prop:automorphism-respects-jacobian}
Every $R$-algebra isomorphism $\varphi:\completeQdcoprime\to \completeQpdcoprime$,
sends $J(S)$ onto $J(\varphi(S))$, thus inducing an isomorphism of Jacobian algebras
${\mathcal P}(A,S) \to {\mathcal P}(A',\varphi(S))$.
\end{prop}

The proof of Proposition \ref{prop:automorphism-respects-jacobian} is a consequence of the fact that the ``differential calculus" for cyclic derivatives developed in \cite{DWZ1} can be adapted to our setup. Let us be more explicit.
Set
$$
\completeQdcoprime \widehat{\otimes} \completeQdcoprime
= \prod_{\ell_1,\ell_2 \geq 0} \left(A^{\ell_1} \underset{F}{\otimes} A^{\ell_2}\right),
$$
Notice that $\pathalgQdcoprime \otimes_F \pathalgQdcoprime$ is canonically embedded as an $F$-vector subspace in
$\completeQdcoprime \widehat{\otimes} \completeQdcoprime$.
Now, for each $a\in Q_1$, we define an $F$-linear map
$$
\Delta_a: \completeQdcoprime\to
\completeQdcoprime \widehat{\otimes} \completeQdcoprime
$$
by setting $\Delta_a (e) = 0$ for $e \in R = A^0$, and
\begin{equation}
\label{eq:delta-xi}
\Delta_a(\omega_0a_1\omega_1 \cdots a_\ell\omega_\ell) = \sum_{k=1}^\ell \delta_{a,a_k} \omega_0a_1\omega_1 \cdots
a_{k-1}\omega_{k-1} \otimes \omega_{k}a_{k+1} \cdots a_\ell\omega_\ell
\end{equation}
for any path $\omega_0a_1\omega_1 \cdots a_\ell\omega_\ell$ of length $\ell \geq 1$ (in particular, a path of length one, say $\omega_0a_1\omega_1$, is mapped either to $0$ or to $\omega_0\otimes\omega_1\in R\otimes_FR$).

Next, denote by $(f,g) \mapsto f \square g$
the $F$-bilinear map
$(\completeQdcoprime \widehat{\otimes} \completeQdcoprime) \times \completeQdcoprime
\to \completeQdcoprime$
induced by the rule
\begin{equation}
\label{eq:square}
(u \otimes v) \square g= v g u
\end{equation}
for $u, v \in \pathalgQd$.
The Leibniz and chain rules for cyclic derivatives are then identical to those in \cite{DWZ1}. Explicitly, we have:

\begin{lemma}[Cyclic Leibniz rule]
Let $h \in \completeQdcoprime_{i,j}$ and
$g \in \completeQdcoprime_{j,i}$ for some
vertices $i$ and $j$.
Then for every $a\in Q_1$, we have
\begin{equation}
\label{eq:leibniz}
\partial_a (hg)=\Delta_a(h)\square g + \Delta_a(g)\square h.
\end{equation}
More generally, for any finite sequence of vertices
$i_1, \dots, i_d, i_{d+1} = i_1$ and for any $h_1, \dots h_\ell$ such that
$h_k \in \completeQdcoprime_{i_k,i_{k+1}}$,
we have
\begin{equation}
\label{eq:leibniz-several}
\partial_a (h_1 \cdots h_\ell)= \sum_{k=1}^\ell
\Delta_a(h_k)\square (h_{k+1}\cdots h_\ell h_1 \cdots h_{k-1}).
\end{equation}
\end{lemma}

\begin{proof}
The proof of \cite[Lemma 3.8]{DWZ1} applies here \emph{mutatis mutandis}.
\end{proof}

\begin{lemma}[Cyclic chain rule]
\label{lem:chain-rule}
Suppose that $\varphi:\completeQdcoprime\to \completeQpdcoprime$ is an $R$-algebra homomorphism.
Then, for every potential $S \in\completeQdcoprime)_{\rm cyc}$
and $a\in Q_1'$, we have:
\begin{equation}
\label{eq:chain-rule}
\partial_{a}(\varphi(S)) =
\sum_{b \in Q_1} \Delta_a (\varphi(b)) \square
\varphi(\partial_{b}(S)).
\end{equation}
\end{lemma}

\begin{proof}
The proof of \cite[Lemma 3.9]{DWZ1} applies here \emph{mutatis mutandis}.
\end{proof}

\begin{proof}[Proof of Proposition \ref{prop:automorphism-respects-jacobian}]
This is now identical to the proof of \cite[Proposition 3.7]{DWZ1}.
\end{proof}


\section{Change of base field}\label{sec:change-of-base-field}

Let $(Q,\dtuple)$ be a strongly primitive weighted quiver. Let us study now the behavior of (complete) path algebras and Jacobian algebras under certain extensions of the ground field $F$. Specifically, let $K/F$ be a finite-degree extension which is \emph{linearly disjoint} from $E/F$, that is, such that the multiplication map $K\otimes_FE\to KE$ is bijective. 
Then for every $i\in Q_0$ the extension $K/F$ is linearly disjoint from  $F_i/F$ as well. Consequently, for all $i,j\in Q_0$, the $F_i$-$F_j$-bimodule
$$
(A_{E/F})_{ij}=\bigoplus_{a:j\to i}F_i\otimes_FF_j
$$
is canonically embedded in the $KF_i$-$KF_j$-bimodule
$$
(A_{KE/K})_{ij}=\bigoplus_{a:j\to i}KF_i\otimes_KKF_j.
$$
Therefore, the $R_{E/F}$-$R_{E/F}$-bimodule $A_{E/F}=\bigoplus_{i,j\in Q_0}(A_{E/F})_{ij}$ is canonically embedded in the $R_{KE/K}$-$R_{KE/K}$-bimodule $A_{KE/K}=\bigoplus_{i,j\in Q_0}(A_{KE/K})_{ij}$, where
$R_{E/F}=\bigoplus_{i\in Q_0}F_i$ and $R_{KE/K}=\bigoplus_{i\in Q_0}KF_i$. All this implies that there exists a canonical embedding of $F$-algebras
$$
\iota:\RAextension{A}{E/F}\hookrightarrow \RAextension{A}{KE/K}.
$$
This embedding obviously sends paths to paths (see the part (1) of Remark \ref{rem:path-is-welldef-notion}).
Moreover, multiplication of paths inside $\RAextension{A}{E/F}$ coincides with multiplication of paths inside $\RAextension{A}{KE/K}$. 

Since $K/F$ is linearly disjoint from $E/F$, we can identify $R_{KE/K}$ with $K\otimes_FR_{E/F}$ by means of the bijective map $K\otimes_FR_{E/F}\rightarrow R_{KE/K}$ induced by the multiplication maps $K\otimes_FF_i\rightarrow KF_i$ (which are bijective by linear disjointness). Hence, $K\otimes_FA_{E/F}$ and $K\otimes_F\RAextension{A}{E/F}$ have a natural structure of $R_{KE/K}$-$R_{KE/K}$-bimodule. Under the identification $R_{KE/K}=K\otimes_FR_{E/F}$, the left and right actions of $R_{KE/K}$ on $K\otimes_FA_{E/F}$ and $K\otimes_F\RAextension{A}{E/F}$ can be written as
\begin{equation}\nonumber
(k\otimes r)(x\otimes y)=kx\otimes ry \ \ \
(x\otimes y)(k\otimes r)=kx\otimes yr.
\end{equation}
Note that $K\otimes_F\RAextension{A}{E/F}$ also has a natural $K$-algebra structure, given by the multiplication
$$
(k_1\otimes u_1)(k_2\otimes u_2)=k_1k_2\otimes u_1u_2.
$$

\begin{lemma}\label{lemma:extension-of-scalars} The arrow span $A_{KE/K}$ is isomorphic to $K\otimes_FA_{E/F}$ as an $R_{KE/K}$-$R_{KE/K}$-bimodule. The algebras $\RAextension{A}{KE/K}$ and $K\otimes_F\RAextension{A}{E/F}$ are isomorphic as $R_{KE/K}$-algebras.
\end{lemma}

Now suppose that $S$ is a potential on the arrow span $A_{E/F}$ of $(Q,\dtuple)$. The $K$-vector subspace $K\otimes_FJ_{E/F}(S)$ of $K\otimes_F\RAextension{A}{E/F}$ is actually an ideal. Under the isomorphism $K\otimes_F\RAextension{A}{E/F}\rightarrow\RAextension{A}{KE/K}$ of Lemma \ref{lemma:extension-of-scalars}, $K\otimes_FJ_{E/F}(S)$ maps bijectively onto the Jacobian ideal $J_{KE/K}(S)\subseteq\RAextension{A}{KE/K}$. We thus have

\begin{coro}\label{coro:Jacobian-ext-of-scalars} The isomorphism $K\otimes_F\RAextension{A}{E/F}\rightarrow\RAextension{A}{KE/K}$ from Lemma \ref{lemma:extension-of-scalars} induces an isomorphism of $K$-algebras $K\otimes_F {\mathcal P}(A_{E/F},S)\rightarrow{\mathcal P}(A_{KE/K},S)$.
\end{coro}


\section{Reduction}\label{sec:reduction}


\begin{defi}
\label{def:SP}
Suppose
$S \in  \completeQdcoprime_{\rm cyc}$ is a potential.
We say that the pair $(A,S)$ is a \emph{strongly primitive species with potential}
(SP for short) if
\begin{equation}
\label{eq:no-loops}
\text{the quiver~$Q$ has no loops, i.e., $A_{i,i} = 0$ for all $i \in Q_0$; and}
\end{equation}
\begin{equation}\label{eq:no-cycequiv}
\text{no two different terms of $S$ are cyclically equivalent.}
\end{equation}
\end{defi}

\begin{remark}\label{rem:why-we-call-it-species} Our use of the term \emph{species} comes from the fact that, for a strongly primitive skew-symmetrizable matrix $B$, the data $((F_i)_{i\in Q_0},(A_{ij})_{i,j\in Q_0},(A_{ji}^\star)_{i,j\in Q_0})$ constitutes a \emph{species realization}, or \emph{modulation} in Dlab-Ringel's nomenclature, of the valued quiver defined by $B$, where  $A_{ji}^{\star}=\bigoplus_{a:j\to i}F_j\otimes_F F_i\cong\Hom_{F_i}(A_{ij},F_i)\cong\Hom_{F_j}(A_{ij},F_j)$.
\end{remark}

\begin{defi}
\label{def:AS-isomorphism} Let $(Q,\dtuple)$ and $(Q',\dtuple)$ be weighted quivers on the same vertex set $Q_0$, and with the same weight tuple $\dtuple$. Let $A$ and $A'$ be their respective arrow spans over $E/F$.
By a \emph{right-equivalence} between SPs $(A,S)$ and $(A',S')$ we mean
an $R$-algebra isomorphism
$\varphi:  \completeQdcoprime\to\completeQpdcoprime$ such that
$\varphi(S)$ is
cyclically equivalent to~$S'$.
\end{defi}

Any $R$-algebra homomorphism sends cyclically equivalent potentials to cyclically
equivalent ones.
It follows that right-equivalences of SPs
have the expected properties: the composition of two right-equivalences,
as well as the inverse of a right-equivalence, is again a right-equivalence.
Notice also that by Proposition \ref{prop:automorphisms}, if $(A,S)$ and
$(A',S')$ are right-equivalent, then there is a bijection $Q_1\to Q_1'$ giving an isomorphism of weighted quivers $(Q,\dtuple)\rightarrow(Q',\dtuple)$. So in dealing with
right-equivalent SPs
we may as well assume that $(Q,\dtuple)=(Q',\dtuple)$ and $A=A'$.

In view of Propositions~\ref{prop:cyclic-equivalence}
and \ref{prop:automorphism-respects-jacobian},
any right-equivalence of SPs
$(A,S) \cong (A,S')$ induces
an isomorphism of the Jacobian ideals $J(S) \cong J(S')$
and of the Jacobian algebras ${\mathcal P}(A,S) \cong {\mathcal P}(A,S')$.

For the following definition we note that if $(A,S)$ is an SP such that $S\in A^2$, then the $R$-$R$-subbimodule $\partial S$ of $\completeQdcoprime$
generated by the set $\{\partial_a(S)\suchthat a\in Q_1\}$ is contained in $A$.

\begin{defi}
\label{def:trivial-&-reduced}
We say that an SP
$(A,S)$ is
\begin{itemize}\item \emph{trivial} if $S \in A^2$, and $\partial S = A$
;
\item \emph{reduced} if $S^{(2)}=0$, that is, if no path of length $2$ appears in the expression of $S$ as possibly infinite-linear combination of cyclic paths.
\end{itemize}
\end{defi}

\begin{ex} Consider the weighted quiver
\begin{center}
$(Q,\dtuple)=$\begin{tabular}{ccc}
$\xymatrix{1  \ar@<0.5ex>[r]^{b}
& 2 \ar@<0.5ex>[l]^{a}},$ & &
$\xymatrix{2  &  3 }$%
\end{tabular}
\end{center}
Let $\B_E=\{1,v,v^2,v^3,v^4,v^5\}$ be an eigenbasis of $E/F$, where $[E:F]=6$. Then $\B_1=\{1,v^3\}$ and $\B_2=\{1,v^2,v^4\}$ are eigenbases of $F_1/F$ and $F_2/F$, respectively. Let $A$ be the arrow span of $(Q,\dtuple)$ over $E/F$. Consider the potential $S=ab+v^3av^2b$, which certainly belongs to $A^2$. We claim that $(A,S)$ is trivial and right-equivalent to $(A,ab)$. Notice that the multiplication map $\mult:F_2\otimes_F F_1\to E$ (resp. $\mult:F_1\otimes_F F_2\to E$) is an isomorphism of $F_2$-$F_1$-bimodules (resp. $F_1$-$F_2$-bimodules), and that the $F_2$-$F_1$-subbimodules (resp. $F_1$-$F_2$-subbimodules) of $E$ are precisely the $E$-vector subspaces of $E$. Since $\mult(b+v^2bv^3)=\mult(a+v^3av^2)=1+v^5\neq 0$, it follows that the cyclic derivative $\partial_a(S)=b+v^2bv^3$ generates $A_{21}=F_2\otimes_FF_1$ as an $F_2$-$F_1$-bimodule, and that the cyclic derivative  $\partial_b(S)=a+v^3av^2$ generates $A_{12}=F_1\otimes_FF_2$ as an $F_1$-$F_2$-bimodule. Furthermore, since the $F_1$-$F_2$-bimodule endomorphisms of $F_1\otimes_FF_2$ correspond to the $E$-linear maps $E\to E$ under the bimodule isomorphism $\mult$, we see that there exists a unique $F_1$-$F_2$-bimodule endomorphism of $F_1\otimes_FF_2$ sending $a$ to $a+v^3av^2$, and this endomorphism is actually an isomorphism. We extend it to a right-equivalence $(A,ab)\to(A,S)$ by sending $b$ to itself.
\end{ex}

\begin{ex} Let $(Q,\dtuple)$ be the weighted quiver from the previous example, with corresponding arrow span $A$. Then any degree-2 potential on $A$ is cyclically equivalent to a potential of the form $S=\alpha_0ab+\alpha_1vab+\alpha_2av^2b+\alpha_3v^3ab+\alpha_4av^4b+\alpha_5v^3av^2b$ for some $\alpha_0,\alpha_1,\alpha_2,\alpha_3,\alpha_4,\alpha_5\in F$. Thus, the space of degree-2 potentials taken up to cyclical equivalence is isomorphic to $F^6$ as an $F$-vector space. Furthermore, $(A,S)$ is trivial if and only if $(\alpha_0,\alpha_1,\alpha_2,\alpha_3,\alpha_4,\alpha_5)\in F^6\setminus\{0\}$, in which case, $(A,S)$ is right-equivalent to $(A,ab)$.
\end{ex}

\begin{prop}\label{prop:trivial-potential} An SP
$(A,S)$ with $S \in A^2$ is trivial
if and only if the set $Q_1$
consists of $2N$ distinct arrows
$a_1, b_1, \dots, a_N, b_N$ such that
$a_k b_k$ is a $2$-cycle and
there is a change of arrows~$\varphi$ of $\completeQdcoprime$
(see Definition~\ref{def:automorphisms}) such that
$\varphi(S)$ is cyclically equivalent to $a_1b_1+\ldots a_Nb_N$.
\end{prop}

\begin{proof} Sufficiency is obvious, let us prove necessity. Suppose $(A,S)$ is trivial.  Fix a total ordering $<$ of the vertices of $Q$. Take two vertices $i,j\in Q_0$ with $i<j$. Up to cyclical equivalence, we can write the part of $S$ that involves
cycles passing through $i$ and $j$ as $S_{ij}=u_1b_1+\ldots+u_nb_{n}$, where $b_1,\ldots,b_n$, are all the arrows of $Q$ that go from $j$ to $i$, and $u_1,\ldots,u_{n}$, are $F$-linear combinations of paths of the form $\omega a\varpi$, with $a:i\rightarrow j$, $\omega\in\B_j$, and $\varpi\in\B_i$.

Let $a_1,\ldots,a_m$, be all arrows of $Q$ that go from $i$ to $j$. Since $A_{ij}$ can be generated by $\partial_{a_1}(S_{ij}),\ldots,\partial_{a_m}(S_{ij})$ as an $F_i$-$F_j$-bimodule, seeing $A_{ij}$ as an $F_iF_j$-vector space if necessary we deduce that $nd_id_j=\dim_F(A_{ij})\leq md_id_j$. Similarly, we have $md_jd_i=\dim_F(A_{ji})\leq nd_jd_i$. Hence $m=n$. 

Again seeing $A_{ji}$ as an $F_{j}F_i$-vector space if necessary, we deduce the existence of an $F_j$-$F_i$-bimodule homomorphism $\varphi_{(ij)}:A_{ji}\to A_{ji}$ sending $a_k$ to $u_k$ for each $k=1,\ldots,n$. This homomorphism is easily seen to be bijective. The proposition follows by assembling all homomorphisms $\varphi_{(ij)}$ for $i<j$.
\end{proof}

\begin{defi}\label{def:direct-sum-of-SPs} Let $A$ and $A'$ be as in Definition \ref{def:AS-isomorphism}. Given SPs
$(A,S)$ and $(A',S')$, we define their \emph{direct sum} to be the SP
$(A,S) \oplus (A,S')=(A\oplus A',S+S')$, where $A\oplus A'$ is the direct sum of $A$ and $A'$ as $R$-$R$-bimodules, and $S+S'$ is seen as an element of $\RA{A\oplus A'}$ through the embeddings of
$\completeQdcoprime$ and $\completeQpdcoprime$ as closed $R$-subalgebras of $\RA{A\oplus A'}$.
\end{defi}

Notice that the $R$-$R$-bimodule $A\oplus A'$ is the arrow span of the weighted quiver $(Q\oplus Q',\dtuple)$ where $Q\oplus Q'=(Q_0,Q_1\sqcup Q_1',h,t)$, with its head and tail functions defined in the obvious way in terms of the head and tail functions of $Q$ and $Q'$.

Taking direct sums with trivial ones
does not affect the Jacobian algebra. More precisely:

\begin{prop}
\label{prop:jacobian-algebra-invariant}
If $(A,S)$ is an arbitrary SP,
and $(A',T)$ is a trivial one, then the canonical embedding
$\completeQdcoprime\hookrightarrow \RA{A\oplus A'}$
induces an isomorphism of Jacobian algebras
${\mathcal P}(A, S) \to {\mathcal P}(A \oplus A',S + T)$.
\end{prop}

\begin{proof} The proof of Proposition 4.5 of \cite{DWZ1} applies here \emph{mutatis mutandis}.
\end{proof}

\begin{defi}\label{def:triv-&-red-spans} We define the \emph{trivial} and \emph{reduced} arrow spans of $(A,S)$
as the $R$-$R$-bimodules given by
\begin{equation}
\label{eq:A-triv-A-red}
A_{\rm triv} = A_{\rm triv}(S) = \partial (S^{(2)}), \quad
A_{\rm red} = A_{\rm red}(S) = A/A_{\rm triv} \ .
\end{equation}
\end{defi}

The following statement will play a crucial role in later sections.

\begin{thm}[Splitting Theorem]
\label{thm:trivial-reduced-splitting}
For every SP $(A,S)$ there exist a trivial
SP $(A_{\triv},S_{\rm triv})$ and a reduced SP
$(A_{\red},S_{\rm red})$ such that $(A,S)$
is right-equivalent to the direct sum
$(A_{\triv},S_{\rm triv}) \oplus (A_{\red}, S_{\rm red})$.
Furthermore, the right-equivalence class of
each of the SPs $(A_{\triv}, S_{\rm triv})$ and $(A_{\red}, S_{\rm red})$
is determined by the right-equivalence class of $(A,S)$.
\end{thm}

\begin{proof} The proof of Theorem 4.6 of \cite{DWZ1} applies here \emph{mutatis mutandis}. More precisely, Lemmas 4.7 and 4.8, Proposition 4.9 and 4.10 and Lemmas 4.11 and 4.12 from \cite{DWZ1} remain valid in our current context of SPs. (Crucial is the fact that the complete path algebra $\RA{A}$ is a \emph{D-algebra}).
\end{proof}

\begin{defi}
\label{def:reduced-part}
We call (the right-equivalence class of) the SP $(A_{\rm red}, S_{\rm red})$ Theorem~\ref{thm:trivial-reduced-splitting} the \emph{reduced part of} $(A,S)$. Likewise, we call (the right-equivalence class of) $(A_{\triv}, S_{\rm triv})$ the \emph{trivial part} of $(A,S)$.
\end{defi}

Reduced parts and trivial parts are stable under extensions $K/F$ linearly disjoint from $E/F$, in the following sense:

\begin{prop}\label{prop:red-commutes-with-ext-of-scalars} Let $(A_{E/F},S)$ be an SP, and let $K/F$ be a finite-degree extension linearly disjoint from $E/F$. If $((A_{E/F})_{\triv},S_{\rm triv})$ and
$((A_{E/F})_{\red},S_{\rm red})$ are a trivial and a reduced SP such that $(A_{E/F},S)$
is right-equivalent to the direct sum
$((A_{E/F})_{\triv},S_{\rm triv}) \oplus ((A_{E/F})_{\red}, S_{\rm red})$, then $((A_{KE/K})_{\triv},S_{\rm triv})$ is a trivial SP, $((A_{KE/K})_{\red}, S_{\rm red})$ is a reduced SP, and $(A_{KE/K},S)$
is right-equivalent to the direct sum
$((A_{KE/K})_{\triv},S_{\rm triv}) \oplus ((A_{KE/K})_{\red}, S_{\rm red})$.
\end{prop}

\begin{proof} We have the following obvious facts:
\begin{itemize}
\item An SP $(A_{E/F},W)$ is trivial if and only if $(A_{KE/K},W)$ is trivial;
\item an SP $(A_{E/F},W)$ is reduced if and only if $(A_{KE/K},W)$ is reduced;
\item any $R_{E/F}$-algebra isomorphism $\varphi:\RAextension{A}{E/F}\rightarrow\RAextension{A'}{E/F}$ can be uniquely extended to an $R_{KE/K}$-algebra isomorphism $\widetilde{\varphi}:\RAextension{A}{KE/K}\rightarrow\RAextension{A'}{KE/K}$. Furthermore, if $\varphi$ is a right-equivalence $(A_{E/F},W_1)\rightarrow(A'_{E/F},W_2)$, then its extension $\widetilde{\varphi}$ is a right-equivalence $(A_{KE/K},W_1)\rightarrow(A'_{KE/K},W_2)$.
\end{itemize}
Proposition \ref{prop:red-commutes-with-ext-of-scalars} is an immediate consequence of these facts together with the observation made right after Definition \ref{def:direct-sum-of-SPs}.
\end{proof}

\begin{defi}
\label{def:2-acyclic}
We call a weighted quiver $(Q,\dtuple)$ \emph{$2$-acyclic}
if the underlying quiver $Q$ has no oriented $2$-cycles, that is, if the arrow span $A$ of $(Q,\dtuple)$ satisfies the following
condition:
\begin{equation}
\label{eq:no-2-cycles}
\text{For every pair of vertices $i \neq j$, either $A_{i,j} = \{0\}$ or $A_{j,i} = \{0\}$.}
\end{equation}
We shall also say that $A$ is 2-acyclic when \eqref{eq:no-2-cycles} is satisfied. Accordingly, any SP $(A,S)$ on $A$ will be called 2-acyclic if $A$ is 2-acyclic.
\end{defi}

In the rest of this section we study the conditions on an SP
$(A,S)$ guaranteeing that its reduced part is $2$-acyclic.
We need some preparation.

\begin{lemma}
Let $\mathcal C = {\mathcal C}(Q,\dtuple)$ be a set of cycles on $(Q,\dtuple)$ (see Remark \ref{rem:specific-choice-of-eigenbases} and the first observation in Remark \ref{rem:path-is-welldef-notion}), satisfying the following two conditions:
\begin{eqnarray}\label{eq:form-of-representative-cycle}
&&\text{Every element of $\mathcal{C}$ has the form $\omega_0a_1\omega_1a_2\omega_2\ldots a_\ell$;}\\
\label{eq:representative-cycle}
&&\text{for every cycle of the form $\omega_0a_1\omega_1a_2\omega_2\ldots a_\ell$ on $(Q,\dtuple)$, exactly one element of the set}\\
\nonumber
&&\text{$\{\omega_{k-1}a_k\omega_{k}\ldots a_\ell\omega_0a_1\omega_1\ldots \omega_{k-2}a_{k-1}\suchthat 1\leq k\leq \ell\}$ belongs to $\mathcal{C}$}.
\end{eqnarray}
For each finite-degree extension $K/F$ which is linearly disjoint from $E/F$, let $K^{\mathcal{C}}=\prod_{c\in\mathcal{C}}K$
and identify $K^{\mathcal{C}}$ with the topological closure of the $K$-vector subspace of $\RAextension{A}{KE/K}$ generated by $\mathcal{C}$ (this can be done because $\mathcal{C}$ is linearly independent over $K$). Then:
\begin{eqnarray}\label{eq:class-of-representative-cycles}
&&\text{Every potential on $A_{KE/K}$ is cyclically-equivalent to an element of $K^{\mathcal{C}}$;}\\
\label{eq:only-0-is-cyc-equiv-to-0}
&&\text{no two different elements of $K^{\mathcal{C}}$ are cyclically-equivalent.}
\end{eqnarray}
\end{lemma}

\begin{proof} The statement \eqref{eq:class-of-representative-cycles} is obvious. To prove \eqref{eq:only-0-is-cyc-equiv-to-0}, it suffices to show that no non-zero element of $K^{\mathcal{C}}$ is cyclically-equivalent to $0$. Suppose that $S\in K^{C}$ is cyclically-equivalent to $0$. Then $S$ is a possibly infinite $K$-linear combination of elements of the form
$$
\omega_0a_1\omega_1a_2\ldots\omega_{\ell-1}a_\ell\omega_\ell-\omega_1a_2\ldots\omega_{\ell-1}a_\ell\omega_\ell\omega_0a_1,
$$
with $\omega_0a_1\omega_1a_2\ldots\omega_{\ell-1}a_\ell\omega_\ell$ a cyclic path on $(Q,\dtuple)$. That is, we can write
\begin{equation}\label{eq:S-cyc-equiv-to-0}
S=\sum_{\omega_0a_1\omega_1a_2\ldots\omega_{\ell-1}a_\ell\omega_\ell}x_{\omega_0a_1\omega_1a_2\ldots\omega_{\ell-1}a_\ell\omega_\ell}
\left(\omega_0a_1\omega_1a_2\ldots\omega_{\ell-1}a_\ell\omega_\ell-\omega_1a_2\ldots\omega_{\ell-1}a_\ell\omega_\ell\omega_0a_1\right),
\end{equation}
with the sum running over all possible cyclic paths on $(Q,\dtuple)$. When we collect similar terms on the right-hand-side of \eqref{eq:S-cyc-equiv-to-0} so as to express it as a non-redundant $K$-linear combination of paths, we see that the coefficient that appears with every cyclic path $\omega_0a_1\omega_1a_2\ldots\omega_{\ell-1}a_\ell\omega_\ell$ such that $\omega_\ell\neq 1$, is precisely $x_{\omega_0a_1\omega_1a_2\ldots\omega_{\ell-1}a_\ell\omega_\ell}$. From \eqref{eq:S-cyc-equiv-to-0}, \eqref{eq:form-of-representative-cycle} and Lemma \ref{lemma:paths-form-a-basis}, we deduce that every such $x_{\omega_0a_1\omega_1a_2\ldots\omega_{\ell-1}a_\ell\omega_\ell}$ is equal to $0$. Therefore, \eqref{eq:S-cyc-equiv-to-0} reduces to
\begin{equation}\label{eq:S-cyc-equiv-to-0-simplified}
S=\sum_{\omega_0a_1\omega_1a_2\ldots\omega_{\ell-1}a_\ell}x_{\omega_0a_1\omega_1a_2\ldots\omega_{\ell-1}a_\ell}
\left(\omega_0a_1\omega_1a_2\ldots\omega_{\ell-1}a_\ell-\omega_1a_2\ldots\omega_{\ell-1}a_\ell\omega_0a_1\right)
\end{equation}
(note the absence of $\omega_\ell$). For each $\xi=\in\omega_0a_1\omega_1a_2\ldots\omega_{\ell-1}a_\ell\in\mathcal{C}$, let $t_{\xi}$ be the smallest element of $\{2,\ldots,\ell+1\}$ such that $\xi=\omega_{t_\xi-1}a_{t_\xi}\omega_{t_\xi}\ldots \omega_{\ell-1}a_\ell\omega_0a_1\ldots\omega_{t_\xi-2}a_{t_\xi-1}$. From \eqref{eq:representative-cycle} and \eqref{eq:S-cyc-equiv-to-0-simplified} we deduce that it is possible to write
\begin{eqnarray}\nonumber
\nonumber
S & =&\sum_{\xi=\omega_0a_1\omega_1a_2\ldots\omega_{\ell-1}a_\ell\in\mathcal{C}}\Big(
y_{\xi,1}(\omega_0a_1\omega_1a_2\ldots\omega_{\ell-1}a_\ell-\omega_1a_2\ldots\omega_{\ell-1}a_\ell\omega_0a_1)\\
\nonumber
&&+  y_{\xi,2}(\omega_1a_2\omega_2a_3\ldots\omega_{\ell-1}a_\ell\omega_0a_1-\omega_2a_3\ldots\omega_{\ell-1}a_\ell\omega_0a_1\omega_1a_2)\\
\nonumber
&& + \ldots\\
\nonumber
&& + y_{\xi,t_\xi-2}(\omega_{t_\xi-3}a_{t_\xi-2}\omega_{t_\xi-2}a_{t_\xi-1}\ldots \omega_{t_\xi-4}a_{t_\xi-3}-
\omega_{t_\xi-2}a_{t_\xi-1}\ldots \omega_{t_\xi-4}a_{t_\xi-3}\omega_{t_\xi-3}a_{t_\xi-2})
\Big)
\end{eqnarray}
for some scalars $y_{\xi,k}\in K$. By \eqref{eq:representative-cycle} and Lemma \ref{lemma:paths-form-a-basis}, we deduce that $y_{\xi,1}=y_{\xi,2}=\ldots=y_{\xi,\ell-1}=0$. Therefore, $S=0$.
\end{proof}

Let $\mathcal{C}=\mathcal{C}(Q,\dtuple)$ be as above. Given any extension $L$ of $F$, we denote by $L[X_{c}\suchthat c\in\mathcal{C}]$ the commutative ring of polynomials with coefficients in $L$ on a set of indeterminates indexed by the elements of $\mathcal{C}$. By definition, the elements of $L[X_{c}\suchthat c\in\mathcal{C}]$ are precisely the finite $L$-linear combinations of monomials $\prod_{c\in\mathcal{C}}X_{c}^{a_c}$ with finite support. The field of fractions of $L[X_{c}\suchthat c\in\mathcal{C}]$ will be denoted $L(X_{c}\suchthat c\in\mathcal{C})$, it consists of rational functions that involve only finitely many variables. Clearly, every $f\in E[X_{c}\suchthat c\in\mathcal{C}]$ induces a function $f^K:K^{\mathcal{C}}\to KE$ for any finite-degree extension $K/F$.

\begin{defi}\label{def:polynomial-&-reg-functions} Let $K/F$ be a finite-degree extension linearly disjoint from $E/F$.
\begin{itemize}\item A function $g:K^{\mathcal{C}}\rightarrow KE$ is called \emph{polynomial map} if there exists a polynomial $f\in E[X_{c}\suchthat c\in\mathcal{C}]$ such that $g=f^K$;
\item given a polynomial $f\in E[X_{c}\suchthat c\in\mathcal{C}]$, we denote $U_K(f)=\{S\in K^{\mathcal{C}}\suchthat f(S)\neq 0\}$;
\item a function $H:U_K(f)\rightarrow K$ is a \emph{regular map} if there are polynomials $g,h\in F[X_{c}\suchthat c\in\mathcal{C}]$, such that $h$ vanishes nowhere in $U_K(f)$ and $H=g^K/h^K$ on $U_K(f)$;
\item if $(Q',\dtuple)$ is another weighted quiver, we will say that a function $H:U_K(f)\rightarrow K^{\mathcal{C}(Q',\dtuple)}$ is a \emph{regular map} if its every component is a regular map $U_K(f)\rightarrow K$.
\end{itemize}
\end{defi}

The following theorem gives a determinantal criterion for the 2-acyclicity of the reduced part of an SP. It can be seen as a (non-trivial) refinement of \cite[Proposition 4.15]{DWZ1}.

\begin{thm}\label{thm:reduction-is-regular-function} Let $(Q,\dtuple)$ be a weighted quiver satisfying \eqref{eq:coprime-assumption}, and let $E/F$ be a field extension satisfying \eqref{eq:d-root-of-unity} and \eqref{eq:E/F-Galois-simple-spectrum}.
\begin{itemize}
\item[(a)]
Let $a_1,b_1,\ldots,a_N,b_N\in Q_1$ be $2N$ distinct arrows such that each $a_kb_k$ is a 2-cycle on $(Q,\dtuple)$. Let $Q'$ denote the quiver obtained from $Q$ by deleting the arrows $a_1,b_1,\ldots,a_N,b_N$, and let $(Q',\dtuple)$ be the corresponding weighted quiver. Then:
\begin{enumerate}\item There exists a non-zero polynomial $f=f_{a_1,b_1,\ldots,a_N,b_N}\in E[X_{c}\suchthat c\in\mathcal{C}]$ such that for every finite-degree extension $K/F$ linearly disjoint from $E/F$, the polynomial map $f^K=f_{a_1,b_1,\ldots,a_N,b_N}^K:K^{\mathcal{C}}\rightarrow KE$ is non-zero and has the property that for any potential $S\in U_K(f)$, the reduced part of $(A_{KE/E},S)$ is 2-acyclic.
\item There exist rational functions $H_{\gamma}\in F(X_{c}\suchthat c\in\mathcal{C})$, $\gamma\in\mathcal{C}'$, such that for every finite-degree extension $K/F$ linearly disjoint from $E/F$, the map $H_\gamma^K:U_K(f)\to K$ is well-defined and hence regular, and the regular map $H^K=H^K_{a_1,b_1,\ldots,a_N,b_N}=(H^K_\gamma)_{\gamma\in\mathcal{C}'}:U_K(f)\to K^{\mathcal{C}'}$ has the property that the reduced part of of $(A_{KE/K},S)$ is right-equivalent to $(A'_{KE/K},H^K(S))$ for every $S\in U_K(f)$.
\end{enumerate}
\item[(b)] If $S\in K^{\mathcal{C}}$ is a potential such that the reduced part of $(A_{KE/K},S)$ is 2-acyclic, then $S\in U_K(f_{a_1,b_1,\ldots,a_N,b_N})$ for some set $\{a_1,b_1,\ldots,a_N,b_N\}$ of $2N$ distinct arrows such that each $a_kb_k$ is a 2-cycle on $(Q,\dtuple)$.
\end{itemize}
\end{thm}

\begin{proof} Assume, without loss of generality, that the arrows $a_1,b_1,\ldots,a_N,b_N$, have been ordered in such a way that
\begin{itemize}
\item for every pair of vertices $i,j\in Q_0$, all arrows in $\{a_1,b_1,\ldots,a_N,b_N,\}$ connecting $i$ and $j$ appear listed consecutively;
\item for every pair of vertices $i,j\in Q_0$, if there is an arrow $a_k:j\to i$ for some index $k\in\{1,\ldots,N\}$, then none of the arrows $b_1,b_2,\ldots,b_N$ goes from $j$ to $i$;
\item for every pair of vertices $i,j\in Q_0$, if there is an arrow $a_k:j\to i$ for some index $k\in\{1,\ldots,N\}$, then the number of arrows of $Q$ that go from $i$ to $j$ is not larger than the number of arrows of $Q$ that go from $j$ to $i$.
\end{itemize}

Take vertices $i,j\in Q_0$ and suppose that the number of arrows from $j$ to $i$ is greater or equal than the number of arrows from $i$ to $j$. 
Suppose that $\alpha_1,\ldots,\alpha_{n_{i,j}},\alpha_{n_{i,j}+1},\ldots,\alpha_{m_{i,j}}$ are all the arrows of $Q$ that go from $j$ to $i$, and that $\beta_1,\ldots,\beta_{n_{i,j}}$ are all arrows that go from $i$ to $j$. For each $S\in F^{\mathcal{C}}$ and each $\ell=1,\ldots,n_{i,j}$, we have
$$
\partial_{\beta_\ell}(S^{(2)})=
\sum_{t=1}^{m_{i,j}}\sum_{\omega_1\in\B_i,\omega_2\in\B_j}x_{t,\ell}^{\omega_1,\omega_2}\omega_1\alpha_{t}\omega_2
$$
where $x_{t,\ell}^{\omega_1,\omega_2}\in F$ is the coefficient of the 2-cycle $\omega_1\alpha_{t}\omega_2\beta_{\ell}$ in the expansion of $S\in F^{\mathcal{C}(A)}$. Let $\Lambda_{i,j}(S)$ be the $m_{i,j}\times n_{i,j}$ matrix whose $(t,\ell)$-entry is $\sum_{\omega_1\in\B_i,\omega_2\in\B_j}x_{t,\ell}^{\omega_1,\omega_2}\omega_1\omega_2$. Notice that the entries of $\Lambda_{i,j}(S)$ are defined by polynomial functions given by linear polynomials with coefficients in $F_iF_j\subseteq E$.

The intersection $\{\alpha_1,\ldots,\alpha_{n_{i,j}},\alpha_{n_{i,j}+1},\ldots,\alpha_{m_{i,j}}\}\cap\{a_1,a_2,\ldots,a_N\}$ is an $n_{i,j}$-element subset of $\{\alpha_1,\ldots,\alpha_{n_{i,j}},\alpha_{n_{i,j}+1},\ldots,\alpha_{m_{i,j}}\}$. Each such $n_{i,j}$-element subset gives rise to a $n_{i,j}\times n_{i,j}$ submatrix $\Delta_{i,j}^{a_1,b_1,\ldots,a_N,b_N}(S)$ of $\Lambda_{i,j}(S)$. We set
$$
f_{a_1,b_1,\ldots,a_N,b_N}^F(S)=\prod_{\{i,j\}} \det(\Delta_{i,j}^{a_1,b_1,\ldots,a_N,b_N}(S)),
$$
where the product is taken over the set of all unordered pairs of distinct vertices of $Q$. Note that $f_{a_1,b_1,\ldots,a_N,b_N}^F=f^F$ with $f\in E[X_{c}\suchthat c\in\mathcal{C}]$.

To show that for any potential $S\in U_F(f)$ the reduced part of $(A,S)$ is 2-acyclic, let us suppose, without loss of generality, that for every $i,j$, we have $\{\alpha_1,\ldots,\alpha_{n_{i,j}}\}\subseteq\{a_1,a_2,\ldots,a_N\}$, and that, moreover, this inclusion is an inclusion of ordered sets. Then $\Lambda_{i,j}(S)$ has the form
$$
\left[\begin{array}{c}
\Delta_{i,j}^{a_1,b_1,\ldots,a_N,b_N}(S)\\
X(S)\end{array}\right],
$$
where $X(S)$ is an $(m_{i,j}-n_{i,j})\times n_{i,j}$ matrix. Let $\Gamma_{i,j}(S)$ be the following $n_{i,j}\times n_{i,j}$ matrix
$$
\Gamma_{i,j}(S)=
\left[\begin{array}{cc}
\Delta_{i,j}^{a_1,b_1,\ldots,a_N,b_N}(S) & 0\\
X(S) & \operatorname{id}\end{array}\right],
$$
If $S\in U_F(f)$, then $\det(\Delta_{i,j}^{a_1,b_1,\ldots,a_N,b_N}(S))\neq 0$ for all $i,j$, hence $\Gamma_{i,j}(S)$ is invertible, and its inverse is
$$
\Gamma_{i,j}(S)^{-1}=
\left[\begin{array}{cc}
(\Delta_{i,j}^{a_1,b_1,\ldots,a_N,b_N}(S))^{-1} & 0\\
-X(S)(\Delta_{i,j}^{a_1,b_1,\ldots,a_N,b_N}(S))^{-1} & \operatorname{id}
\end{array}\right].
$$
Since $\Gamma_{i,j}(S)$ has entries in $F_iF_j$, so does $\Gamma_{i,j}(S)^{-1}$. This readily implies the existence of an $F_i$-$F_j$-bimodule automorphism $\varphi_{i,j}^S$ of $A_{ij}$ such that
$$
\varphi_{i,j}^S\left(\alpha_\ell\right)=
\sum_{t=1}^{m_{i,j}}\sum_{\omega_1\in\B_i,\omega_2\in\B_j}x_{t,\ell}^{\omega_1,\omega_2}\omega_1\alpha_{t}\omega_2
\ \ \ \text{for} 1\leq\ell\leq n_{i,j}
$$
and $\varphi_{i,j}^S(\alpha_\ell)=\alpha_\ell$ for $n_{i,j}+1\leq \ell\leq m_{i,j}$. We can certainly see $\varphi_{i,j}^S$ only as an $F$-linear map $A_{ij}\to A_{ij}$. The entries of the matrix
that represents $\varphi_{i,j}^S$ with respect to the basis $\{\omega_1\alpha_{t}\omega_2\suchthat 1\leq\ell\leq m_{i,j} \omega_1\in\B_i,\omega_2\in B_j\}$ of $A_{ij}$ over $F$, are precisely the scalars $x_{t,\ell}^{\omega_1,\omega_2}$. This readily implies that each of the entries of the matrix
representing $\psi_{i,j}^S=(\varphi_{i,j}^S)^{-1}$ as an $F$-linear map, is given by a regular map $U_F(f)\rightarrow F$.

Define $\psi_{j,i}^S$ to be the identity of $A_{j,i}$ (recall that the roles of $i$ and $j$ is not symmetric at the moment --the number of arrows from $j$ to $i$ has been assumed to be greater or equal than the number of arrows from $i$ to $j$).

Now we collect the $\psi_{i,j}$ over all pair of vertices $i,j$, letting $\psi^S$ be the $R$-algebra automorphism of $\completeQdcoprime$ such that $\psi^S|_{A_{i,j}}=\psi_{i,j}^S$ for all $i,j$. By construction, $\psi^S$ satisfies:
\begin{itemize}
\item the degree-2 component of the potential $\psi^S(S)$ is precisely $\sum_{k=1}^Na_kb_k$;
\item the assignment $S\mapsto\psi^S(S)$ is a regular map $U_F(f)\to F^{\mathcal{C}(A)}$.
\end{itemize}
The first of these two properties implies that the reduced part of $(A,S)$ is 2-acyclic (this can be seen by applying the limit process from the proofs of \cite[Lemmas 4.7 and 4.8]{DWZ1}).

The limit process from the proofs of Lemmas 4.7 and 4.8 of \cite{DWZ1} can be applied to any potential whose degree-2 component is precisely $\sum_{k=1}^Na_kb_k$. From the fact that the product of paths is an $F$-multiple of a path, it is not hard to see that this limit process yields a function $G$ from the set of all potentials with degree-2 component $\sum_{k=1}^Na_kb_k$, to $F^{\mathcal{C}(A')}$, such that:
\begin{itemize}
\item $(A',G(W))$ is the reduced part of $(A,W)$ for every potential $W$ with degree-2 component $\sum_{k=1}^Na_kb_k$;
\item each of the components of $G(W)$ depends polynomially on (finitely many of) the entries of $W$, and the polynomials defining the components of $G(W)$ have coefficients in $F$.
\end{itemize}
Consequently, the assignment $S\mapsto G(\psi^S(S))$ defines a regular map $H^F=H^F_{a_1,b_1,\ldots,a_N,b_N}:U_F(f)\to F^{\mathcal{C}'}$ such that the reduced part of $(A,S)$ is right-equivalent to $(A',H^F(S))$ for every $S\in U_F(f)$.

The considerations made in Section \ref{sec:change-of-base-field} and a second reading of the argument we have given so far, make it clear that the polynomial defining $f^F$ and the rational functions defining the components of $H^F_{a_1,b_1,\ldots,a_N,b_N}$ do not change if we replace $F$ with a finite-degree extension field $K$ linearly disjoint from $E/F$.

Part (a) of Theorem \ref{thm:reduction-is-regular-function} is now proved. Part (b) follows from the fact that if $S\in K^{\mathcal{C}}$ is such that the reduced part of $(A_{KE/K},S)$ is 2-acyclic, then for every $i,j$, the matrix $\Lambda_{i,j}(S)$ defined above has full rank.
\end{proof}


\section{Mutations of species with potentials}
\label{sec:mutations}

In this section we define mutations of SPs and establish two basic properties one should expect from such mutations: well-definedness up to right-equivalence and involutivity up to right-equivalence. We remind the reader that the species with potentials we working with are always assumed to be strongly primitive, that is, the underlying arrow spans are always supposed to be associated with strongly primitive weighted quivers.

Let $k\in Q_0$. Throughout this section we will suppose that
\begin{equation}
\label{eq:no-2-cycles-thru-k}
\text{no oriented 2-cycle of $Q$ is incident to $k$.}
\end{equation}

Replacing $S$ if necessary with a cyclically equivalent
potential, we shall assume that
\begin{equation}
\label{eq:no-start-in-k}
\text{No cyclic path occurring in the expansion of~$S$
starts (nor ends) at~$k$.}
\end{equation}
Under these conditions, we shall associate to $(A,S)$ another SP
$\widetilde{\mu}_k(A,S) = (\widetilde{\mu}_k(A),\widetilde{\mu}_k(S))$ on the
same set of vertices~$Q_0$ and with the same weight function $\dtuple$.
The weighted quiver $(\widetilde{\mu}_k(Q),\dtuple)$ is obtained from $Q$ by means of a two-step procedure:
    \begin{enumerate}\item for each incoming arrow $a:j\to k$ and each outgoing arrow $b:k\to i$,
add $d_k$ ``composite" arrows $\compbwa$, with the index $\omega$ running in the eigenbasis $\B_{k}$;
\item replace each arrow $c$ incident to $k$ by an arrow $c^*$ going in the direction opposite to that of $c$.
    \end{enumerate}
We denote by $\widetilde{\mu}_k(A)$ the arrow-span of $(\widetilde{\mu}_k(Q),\dtuple)$. As for the potential, we define \begin{equation}\label{eq:mu-k-S}
\widetilde{\mu}_k(S)=[S]+\triangle_k(A)\in \RA{\widetilde{\mu}_k(A)},
\end{equation}
\begin{equation}
\text{where $\triangle_k(A)=\sum_{a,b}\sum_{\omega\in\B_{k}}\omega^{-1}b^*\compbwa a^*$},
\end{equation}
the outer sum running over all pairs of arrows $a,b\in Q_1$ such that $h(a)=k$ and $t(b)=k$.

Let us say some words on how the composite arrows $\compbwa$ sit inside $\widetilde{\mu}_k(A)$ and how the potential $[S]$ is obtained from $S$. First of all, for every pair of arrows $a,b\in Q_1$ such that $h(a)=k$ and $t(b)=k$, the $d_k$ arrows $\compbwa$ of $(\widetilde{\mu}_k(Q),\dtuple)$ give rise to the $F_{h(b)}$-$F_{t(a)}$-bimodule
\begin{equation}\label{eq:bimod-of-composite-arrows}
\bigoplus_{\omega\in\B_k}F_{h(b)}\otimes_FF_{t(a)},
\end{equation}
which sits canonically as a direct summand of the component $\widetilde{\mu}_k(A)_{h(b)t(a)}$ of the arrow span $\widetilde{\mu}_k(A)$. Furthermore, for each $\omega\in\B_k$, the arrow $\compbwa$ of $(\widetilde{\mu}_k(Q),\dtuple)$ is identified with the element $1\otimes 1$ of the component corresponding to $\omega$ in the direct sum \eqref{eq:bimod-of-composite-arrows} (see the line preceding Definition \ref{def:path-algebra}). Given our coprimality assumption on the tuple $\dtuple$, for every pair of arrows $a,b\in Q_1$ such that $h(a)=k$ and $t(b)=k$, there is an $F_{h(b)}$-$F_{t(a)}$-bimodule isomorphism
$$
[\bullet]:F_{h(b)}\otimes_{F}F_k\otimes_{F_k}F_k\otimes_F F_{t(a)}\longrightarrow \bigoplus_{\omega\in\B_k}(F_{h(b)}\otimes_FF_{t(a)})
$$
sending each element $b\omega a\in F_{h(b)}\otimes_FF_k\otimes_{F_k}F_k\otimes_FF_{t(a)}$ to the element $\compbwa=1\otimes 1$ in the component corresponding to $\omega\in\B_k$. Assembling these isomorphisms over all pairs $a,b$ with $t(b)=k=h(a)$ we obtain a well-defined map from the closed subspace of $\RA{A}$ consisting of all (possibly infinite) linear combinations of paths not starting nor ending at $k$ to $\RA{\widetilde{\mu}_k(A)}$. The potential $[S]$ is the image of $S$ under this map.

\begin{remark}\label{rem:[bua]-always-defined} Let $a,b\in Q_1$ be arrows such that $h(a)=k=t(b)$. Any $u\in F_k$ yields a well-defined element $[bua]$ of $\oplus_{\omega\in\B_k}(F_{h(b)}\otimes_FF_{t(a)})$. That is, $u$ does not need to belong to $\B_k$ for $[bua]$ to be well-defined. Note that $[bua]$ is an $F$-linear combination of the composite arrows $\comp{b}{\omega}{a}$,
but not necessarily an $F$-multiple of a single one such.
\end{remark}

The following proposition is a direct consequence of the definitions.

\begin{prop}
\label{prop:mu-tilde-trivial-summand}
Suppose an SP $(A,S)$ satisfies \eqref{eq:no-2-cycles-thru-k}
and \eqref{eq:no-start-in-k}, and an SP $(A',S')$ is such that
$e_k A' = A' e_k = \{0\}$.
Then we have
\begin{equation}
\widetilde \mu_k(A \oplus A',S+S') = \widetilde \mu_k(A,S) \oplus
(A',S').
\end{equation}
\end{prop}

\begin{thm}
\label{thm:mu-tilde-preserves-isomorphisms}
The right-equivalence class of the SP $(\widetilde A, \widetilde S) = \widetilde \mu_k(A,S)$
is determined by the right-equivalence class of $(A,S)$.
\end{thm}

\begin{proof} 
Let $\widehat A$ be the finite-dimensional $R$-$R$-bimodule given by
\begin{equation}
\label{eq:hat-A}
\widehat A= A \oplus (e_k A)^\star \oplus (A e_k)^\star.
\end{equation}
The natural embedding $A \to \widehat A$ identifies
$\RA{A}$ with a closed subalgebra
in $\RA{\widehat A}$.
We also have a natural embedding $\widetilde{\mu}_k(A) \to
\RA{ \widehat A}$ (sending each arrow $\compbwa$ to the product $b\omega a$).
This allows us to identify
$\RA{\widetilde{\mu}_k(A)}$ with another closed
subalgebra in $\RA{\widehat A}$. 
Under this identifications, the equality
$$
\sum_{u\in\B_k}\sum_{\omega\in\B_k}u^{-1}\omega^{-1}b^\star bu\omega=d_k\sum_{\omega\in\B_k}\omega^{-1}b^\star b\omega,
$$
whose proof is left to the reader,
implies that the potential $\widetilde{\mu}_k(S)$ given by
\eqref{eq:mu-k-S} and viewed as an element of
$\RA{ \widehat A}$ is cyclically equivalent to
the potential
$$
S + \frac{1}{d_k}\left(\sum_{b \in Q_1 \cap A e_k}\sum_{\omega\in\B_k}  \omega^{-1}b^\star b\omega\right)
\left(\sum_{a \in Q_1 \cap  e_k A}\sum_{\omega\in\B_k}\omega a a^\star\omega^{-1}\right).
$$
Taking this into account, we see that Theorem~\ref{thm:mu-tilde-preserves-isomorphisms}
becomes a consequence of the following lemma.

\begin{lemma}
\label{lemma:extension-of-varphi}
Every automorphism $\varphi$ of $\RA{A}$
can be extended to an automorphism $\widehat\varphi$ of
$\RA{\widehat A}$
satisfying
\begin{equation}
\label{eq:preserving-tildeA}
{\widehat \varphi}(\RA{\widetilde{\mu}_k(A)}) =
\RA{\widetilde{\mu}_k(A)},
\end{equation}
\begin{equation}
\label{eq:preserving-quadratic-terms-a}
{\widehat \varphi}(\sum_{a \in Q_1 \cap  e_k A}\sum_{\omega\in\B_k}\omega a a^\star\omega^{-1}) =
\sum_{a \in Q_1 \cap  e_k A}\sum_{\omega\in\B_k}\omega a a^\star\omega^{-1},
\end{equation}
and
\begin{equation}\label{eq:preserving-quadratic-terms-b}
{\widehat \varphi}(\sum_{b \in Q_1 \cap A e_k}\sum_{\omega\in\B_k}  \omega^{-1}b^\star b\omega) =
\sum_{b \in Q_1 \cap A e_k}\sum_{\omega\in\B_k}\omega^{-1}  b^\star b\omega.
\end{equation}
\end{lemma}

In order to extend $\varphi$ to an automorphism $\widehat\varphi$ of
$\RA{\widehat A}$, we need only to define
the elements $\widehat \varphi(a^\star)$ and $\widehat \varphi(b^\star)$ for all
arrows $a \in Q_1 \cap  e_k A$ and $b \in Q_1 \cap A e_k$.

We first deal with $\widehat \varphi(a^\star)$.
Let $Q_1 \cap e_k A = \{a_1, \dots, a_p\}$. We assume that these arrows have been ordered in such a way that for each $i\in Q_0$, all the arrows from $i$ to $k$ appear consecutively. For each $s=1,\ldots,p$, let $\B_ka_s$ denote the row matrix
$$
\B_ka_s=\left[\begin{array}{cccccc}a_s & v_ka_s & v_k^2a_s & \ldots & v_k^{d_k-2}a_s & v_k^{d_k-1}a_s\end{array}\right].
$$
We set $\B_k\boldsymbol{a}$ to be the row matrix
$$
\B_k\boldsymbol{a}=\left[\begin{array}{ccccc}\B_ka_1 & \B_ka_2 & \ldots & \B_ka_{p-1} & \B_ka_p\end{array}\right].
$$
The entries of $\B_k\boldsymbol{a}$ are elements of $\RA{A}$, whence we can apply $\varphi$ to each of them. Let $\B_k\varphi(\boldsymbol{a})$ be the row matrix whose entries are obtained by applying $\varphi$ componentwise to $\B_k\boldsymbol{a}$. The entries of $\B_k\varphi(\boldsymbol{a})$ are elements of $\RA{A}$ as well.
In view of Proposition~\ref{prop:automorphisms},
\begin{equation}
\label{eq:phi-on-a}
\B_k\varphi(\boldsymbol{a})  =  (\B_k\boldsymbol{a})(C_0+C_1),
\end{equation}
where:
\begin{enumerate}
\item $C_0$ is a $d_kp \times d_kp$ matrix with the following properties:
 \begin{itemize}
 \item it is block-diagonal with square blocks, the blocks given by grouping together the arrows $a_1,\ldots,a_p$ according to their tails;
 \item a block of $C_0$ corresponding to (the arrows whose tail equal) a given vertex $i$ has entries in $F_i$ and, moreover, such a block is an invertible matrix.
 \end{itemize}
\item $C_1$ is a $d_kp \times d_kp$ matrix with the following properties:
\begin{itemize}
\item it can be divided into $p^2$ blocks of size $d_k\times d_k$, each block being determined by an ordered pair $(s,r)$ with $s,r\in\{1,\ldots,p\}$;
\item for $s,r\in\{1,\ldots,p\}$, all entries of the block of $C_1$ determined by the ordered pair $(s,r)$ belong
to ${\mathfrak m}(A)_{t(a_s),t(a_r)}$.
\end{itemize}
\end{enumerate}

Since $\varphi(v^{\ell}a_r)=v^{\ell}\varphi(a_r)$ for all $\ell=0,\ldots,d_k-1$ and all $r=1,\ldots,p$, we can divide $C=C_0+C_1$ into $p^2$ blocks of size $d_k\times d_k$
\begin{equation}\label{eq:C-divided-in-blocks}
C=\left[\begin{array}{ccc}
C_{11} & \ldots & C_{1p}\\
 & \ddots & \\
C_{p1} & \ldots & C_{pp}\end{array}\right],
\end{equation}
the blocks being given by
\begin{equation}\label{eq:C-typical-block}
C_{sr}=
\left[\begin{array}{ccccccc}
c_{sr}^{0} & v^{d_k}c_{sr}^{d_k-1} & v^{d_k}c_{sr}^{d_k-2} &   & v^{d_k}c_{sr}^3 & v^{d_k}c_{sr}^2 & v^{d_k}c_{sr}^1 \\
c_{sr}^1 & c_{sr}^{0} & v^{d_k}c_{sr}^{d_k-1}  &\ldots&    v^{d_k}c_{sr}^4 & v^{d_k}c_{sr}^3 & v^{d_k}c_{sr}^2\\
c_{sr}^2 & c_{sr}^1 & c_{sr}^{0} &&  v^{d_k}c_{sr}^5 & v^{d_k}c_{sr}^4 & v^{d_k}c_{sr}^3\\
& \vdots            &          & \ddots & & \vdots & \\
c_{sr}^{d_k-3} & c_{sr}^{d_k-4} & c_{sr}^{d_k-5} && c_{sr}^{0} & v^{d_k}c_{sr}^{d_k-1} & v^{d_k}c_{sr}^{d_k-2}\\
c_{sr}^{d_k-2} & c_{sr}^{d_k-3} & c_{sr}^{d_k-4} & \ldots & c_{sr}^1 & c_{sr}^{0} & v^{d_k}c_{sr}^{d_k-1}\\
c_{sr}^{d_k-1} & c_{sr}^{d_k-2} & c_{sr}^{d_k-3} && c_{sr}^2 & c_{sr}^1 & c_{sr}^{0}
\end{array}\right].
\end{equation}


Notice that using the identification $F_i=\bigoplus_{j\in Q_0}\delta_{i,j}F_j$ for each vertex $i$, we can interpret $C_0$ as a matrix with entries in $R=\bigoplus_{j\in Q_0}F_j$. The sum $C=C_0 + C_1$ can then be seen as a matrix with entries in $\completeQd$. As such, it is invertible, and its inverse is of the same form: indeed, we have
$$(C_0 + C_1)^{-1} = (I + C_0^{-1} C_1)^{-1} C_0^{-1} =
(C_0^{-1} + \sum_{n=1}^\infty (-1)^n (C_0^{-1} C_1)^nC_0^{-1}) .$$

For each $s=1,\ldots,p$, let $a_s^\star\B_k^{-1}$ be the column matrix which is transpose to the row matrix
$$
\left[\begin{array}{cccccc}a_s^\star & a_s^\star v_k^{-1} & a_s^\star v_k^{-2} & \ldots & a_s^\star v_k^{-(d_k-2)} & a_s^\star v_k^{-(d_k-1)}\end{array}\right].
$$
Set $\boldsymbol{a^\star}\B_k^{-1}$ to be the column matrix
$$
\boldsymbol{a^\star}\B_k^{-1}=\left[\begin{array}{c}
a_1^\star\B_k^{-1}\\
a_2^\star\B_k^{-1}\\
\vdots\\
a_{p-1}^\star\B_k^{-1}\\
a_p^\star\B_k^{-1}\end{array}
\right]
$$

To define the elements ${\widehat\varphi}(a_s^\star)$ it is enough to define what the entries of the matrix ${\widehat\varphi}(\boldsymbol{a^\star})\B_k^{-1}$ are. That is, it suffices to define the matrix ${\widehat\varphi}(\boldsymbol{a^\star})\B_k^{-1}$. We do so through the following matrix product:
\begin{equation}\label{eq:phi(a*)}
{\widehat\varphi}(\boldsymbol{a^\star})\B_k^{-1}=(C_0+C_1)^{-1}(\boldsymbol{a^\star}\B_k^{-1})
\end{equation}
It is not hard to see that $C^{-1}=(C_0+C_1)^{-1}$ has a decomposition into blocks similar to that of $C$, and this readily implies that $\varphi(a_s^\star v_k^{-\ell})=\varphi(a_s^\star)v_k^{-\ell}$ for all $\ell=0,\ldots,d_k-1$ and $s=1,\ldots,p$. That is, there is no ambiguity in the definition \eqref{eq:phi(a*)} of $\varphi(a_s^\star v_k^{-\ell})$.
Now,
\begin{eqnarray}
{\widehat \varphi}\left(\sum_{a \in Q_1 \cap  e_k A}\sum_{l=0}^{d_k-1}v_k^{l} a a^\star v_k^{-l}\right) & = &
(\B_k\varphi(\boldsymbol{a}))({\widehat\varphi}(\boldsymbol{a^\star})\B_k^{-1})\\ \nonumber & = &
(\B_k\boldsymbol{a})(C_0+C_1)(C_0+C_1)^{-1}(\boldsymbol{a^\star}\B_k^{-1})\\ \nonumber &= &
\sum_{a \in Q_1 \cap  e_k A}\sum_{l=0}^{d_k-1}v_k^l a a^\star v_k^{-l}
\end{eqnarray}

For $b \in Q_1 \cap A e_k$, we define $\widehat \varphi(b^\star)$ in a
similar way.
Namely, let
$Q_1 \cap A e_k = \{b_1, \dots, b_q\}$ and assume that these arrows have been ordered in such a way that for each $j\in Q_0$, all the arrows from $k$ to $j$ appear consecutively. For each $r=1,\ldots,q$, let $b_r\B_k$ be the column matrix which is transpose to the matrix
$$
\left[\begin{array}{cccccc}b_r & b_rv_k & b_rv_k^2 & \ldots & b_rv_k^{d_k-2} & b_rv_k^{d_k-1}\end{array}\right].
$$
We set $\boldsymbol{b}\B_k$ to be the column matrix
$$
\boldsymbol{b}\B_k=\left[\begin{array}{c}b_1\B_k\\
b_2\B_k\\
\vdots\\
b_{q-1}\B_k\\
b_q\B_k\end{array}\right].
$$
The entries of $\boldsymbol{b}\B_k$ are elements of $\RA{A}$, whence we can apply $\varphi$ to each of them. Let $\varphi(\boldsymbol{b})\B_k$ be the row matrix whose entries are obtained by applying $\varphi$ componentwise to $\boldsymbol{b}\B_k$. The entries of $\varphi(\boldsymbol{b})\B_k$ are elements of $\RA{A}$ as well.
In view of Proposition~\ref{prop:automorphisms},
\begin{equation}
\label{eq:phi-on-b}
\varphi(\boldsymbol{b})\B_k=(D_0 + D_1)(\boldsymbol{b}\B_k),
\end{equation}
where:
\begin{enumerate}
\item $D_0$ is a $d_kq \times d_kq$ matrix with the following properties:
 \begin{itemize}
 \item it is block-diagonal with square blocks, the blocks given by grouping together the arrows $b_1,\ldots,b_q$ according to their heads;
 \item a block of $D_0$ corresponding to (the arrows whose heads equal) a given vertex $j$ has entries in $F_j$ and, moreover, such a block is an invertible matrix.
 \end{itemize}
\item $D_1$ is a $d_kq \times d_kq$ matrix with the following properties:
\begin{itemize}
\item it can be divided into $q^{2}$ blocks of size $d_k\times d_k$, each block being determined by an ordered pair $(s,r)$ with $s,r\in\{1,\ldots,q\}$
\item for $s,r\in\{1,\ldots,q\}$, all entries of the block of $D_1$ determined by the ordered pair $(s,r)$ belong to
to ${\mathfrak m}(A)_{h(b_s),h(b_r)}$.
\end{itemize}
\end{enumerate}
Since $\varphi(b_rv^{\ell})=\varphi(b_r)v^{\ell}$ for all $\ell=0,\ldots,d_k-1$ and all $r=1,\ldots,q$, we can divide $D=D_0+D_1$ into $q^2$ blocks of size $d_k\times d_k$
\begin{equation}\label{eq:D-divided-in-blocks}
D=\left[\begin{array}{ccc}
D_{11} & \ldots & D_{1q}\\
 & \ddots & \\
D_{q1} & \ldots & D_{qq}\end{array}\right],
\end{equation}
the blocks being given by
\begin{equation}\label{eq:D-typical-block}
D_{sr}=
\left[\begin{array}{ccccccc}
d_{sr}^0 & d_{ij}^1 & d_{sr}^2 && d_{sr}^{d_k-3} & d_{sr}^{d_k-2} & d_{sr}^{d_k-1}\\
v^{d_k}d_{sr}^{d_k-1} & d_{sr}^0 & d_{ij}^1 && d_{sr}^{d_k-4} & d_{ij}^{d_k-3} & d_{sr}^{d_k-2}\\
v^{d_k}d_{sr}^{d_k-2} & v^{d_k}d_{sr}^{d_k-1} & d_{sr}^0 && d_{sr}^{d_k-5} &d_{sr}^{d_k-4} & d_{sr}^{d_k-3}\\
&&\ddots&&&&\\
v^{d_k}d_{sr}^{3} & v^{d_k}d_{sr}^4 & v^{d_k}d_{sr}^5 & \ldots & d_{sr}^{0} & d_{sr}^{1} & d_{sr}^{2}\\
v^{d_k}d_{sr}^2 & v^{d_k}d_{sr}^3 & v^{d_k}d_{sr}^4 & \ldots & v^{d_k}d_{sr}^{d_k-1} & d_{sr}^0 & d_{sr}^1\\
v^{d_k}d_{sr}^1 & v^{d_k}d_{sr}^2 & v^{d_k}d_{sr}^3 & \ldots & v^{d_k}d_{sr}^{d_k-2} & v^{d_k}d_{sr}^{d_k-1} & d_{sr}^0
 \end{array}\right].
\end{equation}

As we saw with $C_0+C_1$, the matrix $D_0 + D_1$ is invertible, and its inverse is of the
same form.

For each $r=1,\ldots,q$, let $\B_k^{-1}b_r^\star$ denote the row matrix
$$
\left[\begin{array}{cccccc}b_r^\star & v_k^{-1}b_r^\star & v_k^{-2}b_r^\star & \ldots & v_k^{-(d_k-2)}b_r^\star & v_k^{-(d_k-1)}b_r^\star\end{array}\right].
$$
Set $\B_k^{-1}\boldsymbol{b^\star}$ to be the row matrix
$$
\B_k^{-1}\boldsymbol{b^\star}=\left[\begin{array}{ccccc}\B_k^{-1}b_1^\star & \B_k^{-1}b_2^\star & \ldots & \B_k^{-1}b_{q-1}^\star & \B_k^{-1}b_q^\star\end{array}\right].
$$
To define the elements ${\widehat\varphi}(b_r^\star)$ it is enough to define what the entries of the matrix $\B_k^{-1}{\widehat\varphi}(\boldsymbol{b^\star})$ are. That is, it suffices to define the matrix $\B_k^{-1}{\widehat\varphi}(\boldsymbol{b^\star})$. We do so through the following matrix product:
\begin{equation}\label{eq:phi(b*)}
\B_k^{-1}{\widehat\varphi}(\boldsymbol{b^\star})= (\B_k^{-1}\boldsymbol{b^\star})(D_0+D_1)^{-1}
\end{equation}

It is easy to check that $D^{-1}=(D_0+D_1)^{-1}$ has a decomposition into blocks similar to that of $D$, and this readily implies that $\varphi(v_k^{-\ell}b_r^\star )=v_k^{-\ell}\varphi(b_r^\star)$ for all $\ell=0,\ldots,d_k-1$ and $r=1,\ldots,q$. That is, there is no ambiguity in the definition \eqref{eq:phi(b*)} of $\varphi(v_k^{-\ell}b_r^\star)$.
Now,
\begin{eqnarray}
{\widehat\varphi}\left(\sum_{b \in Q_1 \cap A e_k}\sum_{l=0}^{d_k-1} v_k^{-l}b_q^\star b_qv_k^l\right) & = &
(\B_k^{-1}{\widehat\varphi}(\boldsymbol{b^\star}))(\varphi(\boldsymbol{b})\B_k) \\ \nonumber   & = & (\B_k^{-1}\boldsymbol{b^\star})(D_0+D_1)^{-1}(D_0 + D_1)(\boldsymbol{b}\B_k) \\ \nonumber & = &
\sum_{b \in Q_1 \cap A e_k}\sum_{l=0}^{d_k-1} v_k^{-l}b_q^\star b_qv_k^l.
\end{eqnarray}

Conditions \eqref{eq:preserving-quadratic-terms-a} and \eqref{eq:preserving-quadratic-terms-b} are then clearly
satisfied; the construction also makes clear that the automorphism $\widehat \varphi$
of $\RA{\widehat A}$ preserves
the subalgebra $\RA{\widetilde{\mu}_k (A)}$.
As a consequence of Proposition~\ref{prop:automorphisms},
$\widehat \varphi$  restricts to an automorphism of
$\RA{\widetilde{\mu}_k (A)}$, verifying
\eqref{eq:preserving-tildeA} and completing the proofs of
Lemma~\ref{lemma:extension-of-varphi} and Theorem~\ref{thm:mu-tilde-preserves-isomorphisms}.
\end{proof}

Note that even if an SP $(A,S)$ is assumed to be reduced,
the SP $\widetilde{\mu}_k(A,S) = (\widetilde{\mu}_k (A), \widetilde{\mu}_k (S))$ is
not necessarily reduced because the component
$[S]^{(2)} \in \widetilde A^2$ may be non-zero.
Combining Theorems~\ref{thm:trivial-reduced-splitting} and
\ref{thm:mu-tilde-preserves-isomorphisms}, we obtain the following
corollary.

\begin{coro}
\label{coro:mutations-respect-isom}
Suppose an SP $(A,S)$ satisfies \eqref{eq:no-2-cycles-thru-k}
and \eqref{eq:no-start-in-k}, and let
$\widetilde \mu_k(A,S) = (\widetilde{\mu}_k (A), \widetilde{\mu}_k (S))$.
Let $(\overline A, \overline S)$ be a reduced SP such that
\begin{equation}
\label{eq:mutilde-mu}
(\widetilde{\mu}_k (A), \widetilde{\mu}_k (S)) \cong (\widetilde{\mu}_k (A)_{\rm triv}, \widetilde{\mu}_k (S)^{(2)}) \oplus
(\overline A, \overline S)
\end{equation}
(see \eqref{thm:trivial-reduced-splitting}).
Then the right-equivalence class of $(\overline A, \overline S)$ is determined by the
right-equivalence class of $(A,S)$.
\end{coro}

\begin{defi}
\label{def:reduced-mutation}
In the situation of Corollary~\ref{coro:mutations-respect-isom},
we use the notation $\mu_k(A,S) = (\overline A, \overline S)$
and call the correspondence $(A,S) \mapsto \mu_k(A,S)$ the \emph{mutation at vertex}~$k$.
\end{defi}

Note that if an SP $(A,S)$ satisfies \eqref{eq:no-2-cycles-thru-k}
then the same is true for $\widetilde \mu_k(A,S)$ and for
$\mu_k(A,S)$.
Thus, the mutation $\mu_k$ is a well-defined transformation on the
set of right-equivalence classes of reduced SPs satisfying
\eqref{eq:no-2-cycles-thru-k}.
(With some abuse of notation, we sometimes denote
a right-equivalence class by the same symbol as any of its
representatives). However, it is not necessarily true that if we start with a 2-acyclic SP $(A,S)$, then $\mu_k(A,S)=(\overline A, \overline S)$ is 2-acyclic as well, and so, we cannot conclude that $\overline A$ is the arrow span of $\mu_k(Q,\dtuple)$.

\begin{ex}\label{ex:1213-cycle-potential-mutation-4}
Let $(Q,\dtuple)$ and $A=A_{E/F}$ be as in Examples \ref{ex:1213-cycle} and \ref{ex:1213-cycle-species}.
Let $S=\alpha\beta\gamma\delta+\alpha v^2\beta\gamma v^3\delta$, then $\widetilde{\mu}_4(A,S)=(\widetilde{\mu}_4(A),\widetilde{\mu}_4(S))$, where the arrow span $\widetilde{\mu}_4(A)$ can be visualized as
$$
\xymatrix{F \ar[rr]_\delta^{F_2\otimes_FF} \ar[dd]^{\alpha^*}_{F_4\otimes_FF} & & F_2 \ar[dd]_\gamma^{F\otimes_FF_2}\\
 & & \\
F_4 \ar[rr]^{\beta^*}_{F\otimes_FF_4}& & F \ar[uull] \ar@<-1.0ex>[uull] \ar@<1.0ex>[uull]}
$$
the diagonal arrows being $[\alpha \beta],\comp{\alpha}{v^2}{\beta},\comp{\alpha}{v^4}{\beta}:3\rightarrow 1$ (note that $\widetilde{\mu}_4(A)_{13}=(F\otimes_FF)^3$), and
$$
\widetilde{\mu}_4(S)=[\alpha \beta]\gamma\delta+\comp{\alpha}{ v^2}{\beta}\gamma v^3\delta+[\alpha \beta]\beta^*\alpha^*+\frac{1}{v^6}\comp{\alpha}{v^2}{\beta}\beta^*v^4\alpha^*+\frac{1}{v^6}\comp{\alpha}{v^4}{\beta}\beta^*v^2\alpha^*.
$$
Since $\widetilde{\mu}_4(A)$ is 2-acyclic, $(\widetilde{\mu}_4(A),\widetilde{\mu}_4(S))$ is already reduced, and thus $\mu_4(A,S)=(\widetilde{\mu}_4(A),\widetilde{\mu}_4(S))=(\mu_4(A),\mu_4(S))$.
\end{ex}

\begin{ex}\label{ex:1213-cycle-potential-mutation-4then2}
Let $(A,S)$ be as in the previous example. Let us apply the mutation $\mu_2$ to the SP $\mu_4(A,S)=(\mu_4(A),\mu_4(S))$. We first compute $\widetilde{\mu}_2(\mu_4(A),\mu_4(S))$. The arrow span $\widetilde{\mu}_2(\mu_4(A))$ is given by
$$
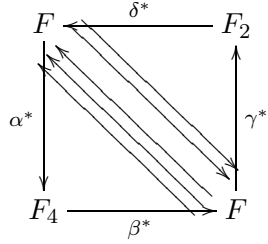
\xymatrix{F \ar[dd]_{\alpha^{*}}  \ar@<2.5ex>[ddrr] \ar@<1.5ex>[ddrr]  & & F_2 \ar[ll]_{\delta^*}\\
 & & \\
 F_4 \ar[rr]_{\beta^{*}} & & F  \ar@<.5ex>[uull] \ar@<1.5ex>[uull] \ar@<2.5ex>[uull] \ar[uu]_{\gamma^*}}
$$
the diagonal arrows being $[\alpha \beta],\comp{\alpha}{v^2}{\beta},\comp{\alpha}{v^4}{\beta}:3\rightarrow 1$, and
$[\gamma\delta],\comp{\gamma}{ v^3}{\delta}:1\rightarrow 3$. Furthermore,
\begin{eqnarray}\nonumber
\widetilde{\mu}_2(\mu_4(S))&=&[\alpha \beta][\gamma\delta]
+\comp{\alpha}{ v^2}{\beta}\comp{\gamma }{v^3}{\delta}
+[\alpha \beta]\beta^*\alpha^*\\
\nonumber
&&+\frac{1}{v^6}\comp{\alpha}{v^2}{\beta}\beta^*v^4\alpha^*
+\frac{1}{v^6}\comp{\alpha}{v^4}{\beta}\beta^*v^2\alpha^*
+\delta^*\gamma^*[\gamma\delta]
+\frac{1}{v^{6}}\delta^*v^3\gamma^*\comp{\gamma}{ v^3}{\delta}.
\end{eqnarray}
The image of $\widetilde{\mu}_2(\mu_4(S))$ under the $R$-algebra automorphism $\varphi$ of $\RA{\widetilde{\mu}_2(\mu_4(A))}$ whose action on the arrows is given by
$$
[\alpha \beta]\mapsto[\alpha \beta]-\delta^*\gamma^*, \
[\gamma\delta]\mapsto[\gamma\delta]-\beta^*\alpha^*, \
\comp{\alpha}{ v^2}{\beta}\mapsto\comp{\alpha}{ v^2}{\beta}-\frac{1}{v^{6}}\delta^*v^3\gamma^*, \
\comp{\gamma}{ v^3}{\delta}\mapsto\comp{\gamma}{ v^3}{\delta}-\frac{1}{v^6}\beta^*v^4\alpha^*
$$
and the identity on the rest of the arrows, is
$$
\varphi(\widetilde{\mu}_2(\mu_4(S)))=
[\alpha \beta][\gamma\delta]
+\comp{\alpha}{ v^2}{\beta}\comp{\gamma}{ v^3}{\delta}
-\delta^*\gamma^*\beta^*\alpha^*
-\frac{1}{v^{12}}\delta^*v^3\gamma^*\beta^*v^4\alpha^*
+\frac{1}{v^6}\comp{\alpha}{v^4}{\beta}\beta^*v^2\alpha^*.
$$
Therefore, $\mu_2\mu_4(A,S)=(\mu_2\mu_4(A),\mu_2\mu_4(S))$, with $\mu_2\mu_4(A)$ given as the arrow span
$$
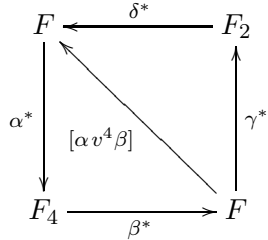
\xymatrix{F \ar[dd]_{\alpha^{*}} & & F_2 \ar[ll]_{\delta^*}\\
 & & \\
 F_4 \ar[rr]_{\beta^{*}} & & F  \ar[uull]^{\ \ \ \ \ \comp{\alpha}{v^4}{\beta}} \ar[uu]_{\gamma^*}}
$$
and $\mu_2\mu_4(S)=-\delta^*\gamma^*\beta^*\alpha^*
-\frac{1}{v^{12}}\delta^*v^3\gamma^*\beta^*v^4\alpha^*
+\frac{1}{v^6}\comp{\alpha}{v^4}{\beta}\beta^*v^2\alpha^*$.
\end{ex}

\begin{ex}\label{ex:1213-cycle-potential-mutation-4then4}
Let $(A,S)$ be as in Example \ref{ex:1213-cycle-potential-mutation-4}.
Let us apply the mutation $\mu_4$ to the SP $\mu_4(A,S)=(\mu_4(A),\mu_4(S))$. We first compute $\widetilde{\mu}_4(\mu_4(A),\mu_4(S))$. The arrow span $\widetilde{\mu}_4(\mu_4(A))$ is given by
$$
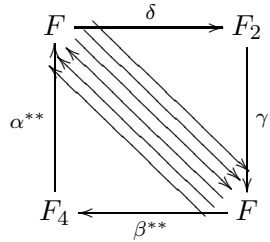
\xymatrix{F \ar[rr]^\delta \ar@<2.5ex>[ddrr] \ar@<1.5ex>[ddrr] \ar@<.5ex>[ddrr] & & F_2 \ar[dd]^\gamma\\
 & & \\
 F_4 \ar[uu]^{\alpha^{**}} & & F \ar[ll]^{\beta^{**}} \ar@<.5ex>[uull] \ar@<1.5ex>[uull] \ar@<2.5ex>[uull]}
$$
the diagonal arrows being $[\alpha \beta],\comp{\alpha}{v^2}{\beta},\comp{\alpha}{v^4}{\beta}:3\to 1$, and $[\beta^*\alpha^*],\comp{\beta^*}{v^2}{\alpha^*},\comp{\beta^*}{v^4}{\alpha^*}:1\to 3$, and
\begin{eqnarray}\nonumber
\widetilde{\mu}_4(\mu_4(S))&=&[\alpha \beta]\gamma\delta+\comp{\alpha}{v^2}{\beta}\gamma v^3\delta+[\alpha \beta][\beta^*\alpha^*]+\frac{1}{v^6}\comp{\alpha}{v^2}{\beta}\comp{\beta^*}{v^4}{\alpha^*}+
\frac{1}{v^6}\comp{\alpha}{v^4}{\beta}\comp{\beta^*}{v^2}{\alpha^*}\\
\nonumber
&&+\alpha^{**} \beta^{**}[\beta^*\alpha^*]
+\frac{1}{v^6}\alpha^{**} v^2\beta^{**}\comp{\beta^*}{v^4}{\alpha^*}
+\frac{1}{v^6}\alpha^{**} v^4\beta^{**}\comp{\beta^*}{v^2}\alpha^*.
\end{eqnarray}
The image of $\widetilde{\mu}_4(\mu_4(S)))$ under the $R$-algebra automorphism $\varphi$ of $\RA{\widetilde{\mu}_3(\mu_3(A))}$ whose action on the arrows is given by
$$
[\alpha \beta]\mapsto[\alpha \beta]-\alpha^{**}\beta^{**}, \ [\beta^*\alpha^*]\mapsto[\beta^*\alpha^*]-\gamma\delta,
$$
$$
\comp{\alpha}{v^2}{\beta}\mapsto\comp{\alpha}{v^2}{\beta}-\alpha^{**}v^2\beta^{**}, \ \comp{\beta^*}{v^4}{\alpha^*}\mapsto\comp{\beta^*}{v^4}{\alpha^*}-v^6\gamma v^3\delta, \ \comp{\alpha}{v^4}{\beta}\mapsto \comp{\alpha}{v^4}{\beta}-\alpha^{**}v^4\beta^{**}
$$
and the identity on the rest of the arrows, is
$$
\varphi(\widetilde{\mu}_4(\mu_4(S))))=
[\alpha \beta][\beta^*\alpha^*]+\frac{1}{v^6}\comp{\alpha}{v^2}{\beta}\comp{\beta^*}{v^4}{\alpha^*}
+\frac{1}{v^6}\comp{\alpha}{v^4}{\beta}\comp{\beta^*}{v^2}{\alpha^*}
-\alpha \beta \gamma\delta-\alpha v^2\beta \gamma v^3\delta.
$$
Hence
$\mu_4(\mu_4(A,S))$ is right-equivalent to $(A,S)$.
\end{ex}

As suggested by Example \ref{ex:1213-cycle-potential-mutation-4then4}, our next result is that every mutation is an involution up to right-equivalence.

\begin{thm}
\label{thm:mutation-involutive}
The correspondence $\mu_k: (A,S) \to (\overline A, \overline S)$
acts as an involution on the set of right-equivalence classes of reduced SPs
satisfying \eqref{eq:no-2-cycles-thru-k}, that is,
$\mu_k^2(A,S)$ is right-equivalent to $(A,S)$.
\end{thm}

\begin{proof} We follow the proof of \cite[Theorem 5.7]{DWZ1}, adapting it to our setup.
Let $(A, S)$ be a reduced SP satisfying \eqref{eq:no-2-cycles-thru-k}
and \eqref{eq:no-start-in-k}.
Let ${\widetilde \mu}_k (A,S)= ({\widetilde A}, {\widetilde S})$ and
${\widetilde \mu}_k^2 (A, S)= {\widetilde \mu}_k({\widetilde A}, {\widetilde S}) =
(\widetilde {\widetilde A}, \widetilde {\widetilde S})$.
In view of Theorem~\ref{thm:trivial-reduced-splitting} and
Proposition~\ref{prop:mu-tilde-trivial-summand}, it is enough to show that
\begin{equation}
\label{eq:tilde-twice}
\text{$(\widetilde {\widetilde A}, \widetilde {\widetilde S})$ is
right-equivalent to $(A,S) \oplus (C,T)$, where $(C,T)$ is a trivial SP.}
\end{equation}
Identifying $(e_k A)^\star$ with
$A^\star e_k$, and $(A e_k )^\star$ with
$e_k A^\star$, where $A^\star$ is the dual $R$-$R$-bimodule of $A$,
we conclude that
\begin{equation}
\label{eq:tilde-tilde-A}
\widetilde {\widetilde A} =
A \oplus  A e_k A \oplus
A^\star e_k A^\star.
\end{equation}
Furthermore, the basis of arrows in
$\widetilde {\widetilde A}$ consists of the original set of arrows $Q_1$ in $A$
together with the arrows $\compbwa \in  A e_k A$ and
$\comp{a^\star}{\omega}{ b^\star} \in A^\star e_k A^\star$
for $a \in Q_1 \cap  e_k A$, $b \in Q_1 \cap A e_k$, and $\omega\in\B_k$.
We see then that the
potential $\widetilde{\mu}_k\widetilde{\mu}_{k}(S)=\widetilde {\widetilde S}$ is
cyclically equivalent to
\begin{equation}
\label{eq:S1}
S_1= [S] + \sum_{a, b \in Q_1: \ h(a) = t(b) = k}\left(([ba]+ba)[a^\star b^\star]+\sum_{\ell=1}^{d_k-1} \frac{1}{v_k^{d_k}}(\comp{b}{v_k^\ell} {a}
+ bv_k^{\ell}a)\comp{a^\star}{v_k^{d_k-\ell}} {b^\star} \right).
\end{equation}
Let us abbreviate
$$
(C,T) = (A e_k A \oplus A^\star e_k A^\star,
\sum_{a, b \in Q_1: \ h(a) = t(b) = k}\left([ba][a^\star b^\star]+\sum_{\ell=1}^{d_k-1} \frac{1}{v^{d_k}}\comp{b}{v^\ell}{a}\comp{a^\star}{v^{d_k-\ell}}{ b^\star} \right).
$$
This is a trivial SP (cf. Proposition~\ref{prop:trivial-potential});
therefore to prove Theorem~\ref{thm:mutation-involutive}
it suffices to show that
the SP $(\widetilde {\widetilde A}, S_1)$ given by
\eqref{eq:tilde-tilde-A} and \eqref{eq:S1} is right-equivalent to
$(A,S) \oplus (C,T)$.
We proceed in several steps.

{\bf Step 1:} Let $\varphi_1$ be the change of arrows
automorphism of
$\RA{\widetilde {\widetilde A}}$
(see Definition~\ref{def:automorphisms}) multiplying
each arrow $b \in Q_1 \cap A e_k$ by $-1$, and fixing the rest of
the arrows in $\widetilde {\widetilde A}$.
Then the potential $S_2 = \varphi_1(S_1)$ is given by
$$
S_2= [S] + \sum_{a, b \in Q_1: \ h(a) = t(b) = k}\left(([ba]-ba)[a^\star b^\star]+\sum_{\ell=1}^{d_k-1} \frac{1}{v^{d_k}}(\comp{b}{v^\ell}{a}
- bv^{\ell}a)\comp{a^\star}{v^{d_k-\ell}}{b^\star} \right).
$$

{\bf Step 2:}
Let $\varphi_2$ be the unitriangular automorphism of
$\RA{\widetilde {\widetilde A}}$
(see Definition~\ref{def:automorphisms})
sending each arrow $\comp{b}{v^\ell}{a} \in  A e_k A$ to $\comp{b}{v^{\ell}}{a} + bv^{\ell}a$,
and fixing the rest of
the arrows in $\widetilde {\widetilde A}$.
Remembering the definition of $[S]$, it is easy to see that
the potential $\varphi_2(S_2)$ is cyclically equivalent
to a potential of the form
$$
S_3 = S+\sum_{a, b \in Q_1: \ h(a) = t(b) = k}\left([ba]([a^\star b^\star]+f(a,1,v))+\sum_{\ell=1}^{d_k-1} \frac{1}{v^{d_k}}\comp{b}{v^\ell}{a}
(\comp{a^\star}{v^{d_k-\ell}}{b^\star}+ f(a,v^\ell,b))\right)
$$
for some elements $f(a,v^\ell,b) \in {\mathfrak m}(A \oplus  A e_k A)^2$, $\ell=0,\ldots,d_k-1$.

{\bf Step 3:}
Let $\varphi_3$ be the unitriangular automorphism of
$\RA{ \widetilde {\widetilde A}}$
sending each arrow $\comp{a^\star}{v^\ell}{b^\star} \in A^\star e_k A^\star$ ($\ell=0,\ldots,d_k-1$) to
$\comp{a^\star}{v^\ell}{b^\star} - f(a,v^{d_k-\ell},b)$ (we take $f(a,v^{d^k},b)=f(a,1,b)$),
and fixing the rest of
the arrows in $\widetilde {\widetilde A}$.
Then we have $\varphi_3(S_3) = S+T$.

Combining these three steps, we conclude that
the SP $(\widetilde {\widetilde A}, S_1)$
is right-equivalent to
$(\widetilde {\widetilde A}, S+T) = (A,S) \oplus (C,T)$, finishing
the proof of Theorem~\ref{thm:mutation-involutive}.
\end{proof}

Let $(A_{E/F},S)$ be an SP. Notice that the potential \eqref{eq:mu-k-S} remains the same if we replace $A_{E/F}$ with $A_{KE/K}$ for any finite-degree extension $KE/K$ linearly disjoint from $E/F$. 
In view of Proposition \ref{prop:red-commutes-with-ext-of-scalars}, this implies the following.

\begin{prop}\label{prop:SPmut-commutes-with-ext-of-scalars} Let $(A_{E/F},S)$ be an SP, and let $(\widetilde{Q},\dtuple)$ be the weighted quiver underlying the arrow span of $\mu_k(A_{E/F},S)$. For any finite-degree extension $K/F$ linearly disjoint from $E/F$, the SP $\mu_k(A_{KE/K},S)$ can be obtained from $\mu_k(A_{E/F},S)$ by replacing the arrow span underlying the latter SP with the arrow span of $(\widetilde{Q},\dtuple)$ over $KE/K$.
\end{prop}

\section{Nondegeneracy}\label{sec:nondegeneracy}

Following \cite{DWZ1} we make the following definition.

\begin{defi}\label{def:locally-non-degenerate} Let $(k_1,\ldots,k_\ell)$ be a finite sequence of vertices of $Q$ such that $k_t\neq k_{t+1}$ for $t=1,\ldots,\ell-1$. We say that an SP $(A,S)$ is \emph{$(k_\ell,\ldots,k_1)$-nondegenerate} if all the SPs $(A,S)$, $\mu_{k_1}(A,S)$, $\mu_{k_2}\mu_{k_1}(A,S)$, $\ldots$, $\mu_{k_\ell}\ldots\mu_{k_1}(A,S)$ are 2-acyclic (hence well-defined). An SP $(A,S)$ will be called \emph{nondegenerate} if it is \emph{$(k_\ell,\ldots,k_1)$-nondegenerate} for every sequence of vertices as above.
\end{defi}

The following Proposition can be seen as a (non-trivial) refinement of \cite[Proposition 7.3]{DWZ1}.

\begin{prop}\label{prop:exist-pols-measuring-local-nondeg} Let $(Q,\dtuple)$ be a 2-acyclic weighted quiver and $(k_1,\ldots,k_\ell)$ be a finite sequence of vertices as in Definition \ref{def:locally-non-degenerate}.
There exists a non-empty finite subset $\mathcal{F}=\mathcal{F}_{(k_1,\ldots,k_\ell)}\subseteq E[X_c\suchthat c\in\mathcal{C}]$, consisting only of non-zero polynomials, such that for every finite-degree extension $K/F$ linearly disjoint from $E/F$ and every $S\in K^{\mathcal{C}}$, the species with potential $(A_{KE/K},S)$ is $(k_\ell,\ldots,k_1)$-nondegenerate if and only if $f(S)\neq 0$ for some element $f\in \mathcal{F}$.
\end{prop}

\begin{proof} By induction on $\ell\geq 1$. For every extension $K/F$ as in the statement of the proposition, let $\widetilde{\mu}_{k_1}(A_{KE/K})$ be the arrow span of the weighted quiver $\widetilde{\mu}_{k_1}(Q,\dtuple)=(\widetilde{\mu}_{k_1}(Q),\dtuple)$ over $KE/K$. By Theorem \ref{thm:reduction-is-regular-function}, there exists a non-empty finite subset $\mathcal{G}\subseteq E[X_\gamma\suchthat \gamma\in\mathcal{C}(\widetilde{\mu}_{k_1}(Q),\dtuple)]$, consisting only of non-zero polynomials, such that for every extension $K/F$ as in the statement of the proposition, and every $W\in K^{\mathcal{C}(\widetilde{\mu}_{k_1}(Q),\dtuple)}$, the reduced part of $(\widetilde{\mu}_{k_1}(A_{KE/K}),W)$ is 2-acyclic if and only if $g(W)\neq 0$ for some $g\in\mathcal{G}$.

On the other hand, there obviously exist polynomials $p_\gamma\in F[X_c\suchthat c\in\mathcal{C}(Q,\dtuple)]$, $\gamma\in\mathcal{C}(\widetilde{\mu}_{k_1}(Q),\dtuple)$, such that for every extension $K/F$ as in the statement of the proposition, and every $S\in K^{\mathcal{C}(Q,\dtuple)}$, we have  $\widetilde{\mu}_{k_1}(S)=\sum_{\gamma\in\mathcal{C}(\widetilde{\mu}_{k_1}(Q),\dtuple)}p_\gamma(S)\in K^{\mathcal{C}(\widetilde{\mu}_{k_1}(Q),\dtuple)}$. Let
$$
\mathcal{F}_{(k_1)}=\left\{g\left(\sum_{\gamma\in\mathcal{C}(\widetilde{\mu}_{k_1}(Q),\dtuple)}p_\gamma\right) \ \Big| \ g\in\mathcal{G} \ \text{and} \ g\left(\sum_{\gamma\in\mathcal{C}(\widetilde{\mu}_{k_1}(Q),\dtuple)}p_\gamma\right) \ \text{is not the zero polynomial}\right\}.
$$
The set $\mathcal{F}_{(k_1)}$ certainly has the property that for every extension $K/F$ as in the statement of the proposition, and every $S\in K^{\mathcal{C}(Q,\dtuple)}$, the species with potential $(A_{KE/K},S)$ is $(k_1)$-nondegenerate if and only if $f(S)\neq 0$ for some $f\in\mathcal{F}_{(k_1)}$. Also, $\mathcal{F}_{(k_1)}$ consists entirely of non-zero polynomials with coefficients in $E$. The fact that $\mathcal{F}_{(k_1)}$ is not empty follows from the fact that all 2-cycles of $(\widetilde{\mu}_{k_1}(Q),\dtuple)$ are of the form $\omega_0\comp{b}{\omega_1}{a}\omega_2c\omega_3$ with $\omega_0b\omega_1a\omega_2c\omega_3$ a 3-cycle of $(Q,\dtuple)$ such that $h(a)=k_1=t(b)$. This establishes the inductive basis of our proof.

Let $(Q',\dtuple)=\mu_{k_1}(Q,\dtuple)$. Before entering the inductive step of our proof we need some preparation. Recall from the proof of Theorem \ref{thm:reduction-is-regular-function}, that for each $g\in\mathcal{G}$ there exist rational functions $H'_{g,\delta}\in F(X_c\suchthat c\in\mathcal{C}(\widetilde{\mu}_{k_1}(Q,\dtuple)))$, $\delta\in\mathcal{C}(Q',\dtuple)$, such that for every extension $K/F$ as in the statement of the proposition, and every $W\in U_K(g)\subseteq K^{\mathcal{C}(\widetilde{\mu}_{k_1}(Q,\dtuple))}$, the reduced part of $(\widetilde{\mu}_{k_1}(A_{KE/K}),W)$ is right-equivalent to $(A'_{KE/K},H'_g(W))$, where $H'_g(W)=\sum_{\gamma\in\mathcal{C}(Q',\dtuple)}H'_{g,\delta}(W)\in K^{\mathcal{C}(Q',\dtuple)}$, and $A'_{KE/K}$ is the arrow span of $(Q',\dtuple)=\mu_{k_1}(Q,\dtuple)$ over the field extension $KE/K$. For each $f=g\left(\sum_{\gamma\in\mathcal{C}(\widetilde{\mu}_{k_1}(Q),\dtuple)}p_\gamma\right)\in\mathcal{F}_{(k_1)}$, and each $\delta\in \mathcal{C}(Q',\dtuple)$ define
$H_{f,\delta}=H'_{g,\delta}\left(\sum_{\gamma\in\mathcal{C}(\widetilde{\mu}_{k_1}(Q),\dtuple)}p_\gamma\right)$. Then each $H_{f,\delta}$ is a rational function $H_{f,\delta}\in F(X_c\suchthat c\in\mathcal{C}(Q,\dtuple))$, such that for every $K/F$ the evaluation of $H_{f,\delta}$ on $U_K(f)\subseteq K^{\mathcal{C}(Q,\dtuple)}$ is well-defined, and furthermore, for every $S\in U_{K}(f)$, the mutation $\mu_k(A_{KE/K},S)$ is right-equivalent to $(A'_{KE/K},H_f(S))$, where $H_f(S)=\sum_{\delta\in\mathcal{C}(Q',\dtuple)}H_{f,\delta}(S)\in K^{\mathcal{C}(Q,\dtuple)}$.

For the inductive step, let $(Q',\dtuple)=\mu_{k_1}(Q,\dtuple)$, and suppose that there exists $\mathcal{F}_{(k_2,\ldots,k_\ell)}\subseteq E[X_c\suchthat c\in\mathcal{C}(Q',\dtuple)]$ with the desired properties for the sequence $(k_2,\ldots,k_\ell)$. For each $f_1\in\mathcal{F}_{(k_1)}\subseteq F[X_c\suchthat c\in\mathcal{C}(Q,\dtuple)]$ and each $f_2\in\mathcal{F}_{(k_2,\ldots,k_\ell)}$, we have a rational function $f_2(H_{f_1})\in F(X_c\suchthat c\in\mathcal{C}(Q,\dtuple))$. Let $h_{f_1,f_2}$ be the numerator of this rational function $f_2(H_{f_1})$. Set
$$
\mathcal{F}_{(k_1,\ldots,k_\ell)}=\left\{f_1h_{f_1,f_2}\suchthat f_1\in\mathcal{F}_{(k_1)}, \ f_2\in\mathcal{F}_{(k_2,\ldots,k_\ell)}, \ \text{and} \ h_{f_1,f_2} \ \text{is not the zero polynomial}\right\}.
$$
This subset of $F[X_c\suchthat c\in\mathcal{C}(Q,\dtuple)]$ is obviously finite and consists of non-zero polynomials. Let $K/F$ be any finite-degree extension linearly disjoint from $E/F$. For $S\in K^{\mathcal{C}(Q,\dtuple)}$, if $h_{f_1,f_2}(S)\neq 0$, then $f_1(S)\neq 0$ and $(f_2(H_{f_1}))(S)\neq 0$, and this implies that $S$ is $(k_1,\ldots,k_\ell)$-nondegenerate. Conversely, if $S\in K^{\mathcal{C}(Q,\dtuple)}$ is $(k_1,\ldots,k_\ell)$-nondegenerate, then there exists $f_1\in\mathcal{F}_{(k_1)}$ such that $f_1(S)\neq 0$, and furthermore, $(A'_{KE/K},H_{f_1}(S))$, being right-equivalent to $\mu_k(A_{KE/K},S)$, is $(k_2,\ldots,k_\ell)$-non-degenerate. Thus there exists $f_2\in \mathcal{F}_{(k_2,\ldots,k_\ell)}$ such that $f_2(H_{f_1}(S))\neq 0$, and hence $(f_1h_{f_1,f_2})(S)\neq 0$.

It only remains to show that $\mathcal{F}_{(k_1,\ldots,k_\ell)}$ is not the empty set. Fix an algebraic closure $\overline{F}$ of $F$.
The set $L$ defined as
$$
L=\begin{cases}\{\alpha\in\overline{F}\suchthat F(\alpha)\cap E=F\} & \text{if $F$ is finite;}\\
F & \text{if $F$ is infinite}.
\end{cases}
$$
is an infinite field. Therefore, since the polynomials belonging to $\mathcal{F}_{(k_2,\ldots,k_\ell)}$ are non-zero and depend on only finitely many indeterminates and since $\mathcal{F}_{(k_2,\ldots,k_\ell)}$ is finite, there exist $\alpha\in L$ and $W\in F(\alpha)^{\mathcal{C}(Q',\dtuple)}$ such that none of the polynomials belonging to $\mathcal{F}_{(k_2,\ldots,k_\ell)}$ vanishes when evaluated at $W$, and such that $\mu_{k_1}(A'_{E(\alpha)/F(\alpha)},W)$ is 2-acyclic. Let $S\in F(\alpha)^{\mathcal{C}(Q,\dtuple)}$ be such that $(A_{E(\alpha)/F(\alpha)},S)$ is right-equivalent to $\mu_{k_1}(A'_{E(\alpha)/F(\alpha)},W)$, then $(A_{E(\alpha)/F(\alpha)},S)$ is $(k_1,\ldots,k_\ell)$-nondegenerate, and this implies that $\mathcal{F}_{(k_1,\ldots,k_\ell)}\neq\varnothing$. This finishes the proof of the Proposition \ref{prop:exist-pols-measuring-local-nondeg}.
\end{proof}

\begin{remark} In the proof of Proposition \ref{prop:exist-pols-measuring-local-nondeg} we have corrected an inaccuracy in the proof of   \cite[Proposition 7.3]{DWZ1}.
\end{remark}

\begin{coro}\label{coro:locally-nondeg-exist} Let $(Q,\dtuple)$ be a 2-acyclic weighted quiver and $(k_1,\ldots,k_\ell)$ be a finite sequence of vertices as in Definition \ref{def:locally-non-degenerate}. Suppose that the ground field extension $E/F$ is one of the extensions described in Examples \ref{ex:finite-fields-satisfy} and \ref{ex:p-adic-fields-satisfy}. There exists a finite-degree extension $K/F$, linearly disjoint from $E/F$, such that the arrow span $A_{KE/K}$ admits a $(k_\ell,\ldots,k_1)$-nondegenerate potential. If $E/F$ is as in Example \ref{ex:p-adic-fields-satisfy}, $K$ can be taken to be $F$ itself.
\end{coro}

\begin{coro}\label{coro:nondeg-exist-over-padic} Let $(Q,\dtuple)$ be a 2-acyclic weighted quiver. If the ground field extension $E/F$ is as in Example \ref{ex:p-adic-fields-satisfy}, then the arrow span $A_{E/F}$ admits a nondegenerate potential.
\end{coro}

\begin{proof} For each finite sequence $(k_1,\ldots,k_\ell)$ of vertices as in Definition \ref{def:locally-non-degenerate}, Proposition \ref{prop:exist-pols-measuring-local-nondeg} guarantees the existence of a non-zero polynomial $f_{(k_1,\ldots,k_\ell)}\in E[X_c\suchthat c\in\mathcal{C}]$ such that every $S\in U_{F}(f_{(k_1,\ldots,k_\ell)})\subseteq F^{\mathcal{C}}$ is $(k_\ell,\ldots,k_1)$-nondegenerate (note that $F/F$ is obviously linearly disjoint from $E/F$).

Denote
$$
\mathcal{F}=\{f_{(k_1,\ldots,k_\ell)}\suchthat (k_1,\ldots,k_\ell) \ \text{is a sequence as in Definition \ref{def:locally-non-degenerate}}\},
$$
which is a countable set of polynomials that belong to $E[X_c\suchthat c\in\mathcal{C}]$. Since $\mathcal{C}$ is countable, we can identify $E[X_c\suchthat c\in\mathcal{C}]$ with the commutative ring of polynomials with coefficients in $E$ on countably many variables $X_1,X_2,X_3,\ldots$. For each $n>0$ define
$$
\mathcal{F}_n=\mathcal{F}\cap E[X_1,\ldots,X_n].
$$
We clearly have $\mathcal{F}_1\subseteq\mathcal{F}_2\subseteq\mathcal{F}_3\subseteq\ldots\subseteq \mathcal{F}_n\subseteq\mathcal{F}_{n+1}\subseteq\ldots\subseteq \mathcal{F}=\bigcup_{n>0}\mathcal{F}_n$. For each $n>0$ define sets $\mathcal{F}_{n,n-1},\mathcal{F}_{n,n-2}\ldots,\mathcal{F}_{n,1}$, recursively as follows:
\begin{itemize}
\item $\mathcal{F}_{n,n-1}$ is the set of non-zero elements of $E[X_1,\ldots,X_{n-1}]$ that appear as coefficients of the elements of $\mathcal{F}_n$ when we write the latter as polynomials in $X_n$ with coefficients in $E[X_1,\ldots,X_{n-1}]$;
\item For $j=1,\ldots,n-2$, once $\mathcal{F}_{n,n-j}$ has been defined, we define $\mathcal{F}_{n,n-(j+1)}$ to be the set of non-zero elements of $E[X_1,\ldots,X_{n-(j+1)}]$ that appear as coefficients of the elements of $\mathcal{F}_{n,n-j}$ when we write the latter as polynomials in $X_{n-j}$ with coefficients in $E[X_1,\ldots,X_{n-(j+1)}]$.
\end{itemize}
Since $\mathcal{F}_n$ is countable, each one of the sets $\mathcal{F}_{n,j}$ is countable. Therefore, for each $j>0$, the set
$$
\mathcal{G}_j=\bigcup_{n>j}\mathcal{F}_{n,j}
$$
is countable. Note that $\mathcal{F}_n\subseteq \mathcal{G}_n\subseteq E[X_1,\ldots,X_n]$ for all $n$, and that the elements of $\mathcal{G}_n$ are precisely the non-zero elements of $E[X_1,\ldots,X_{n}]$ that appear as coefficients of the elements of $\mathcal{G}_{n+1}$ when we write the latter as polynomials in $X_{n+1}$ with coefficients in $E[X_1,\ldots,X_{n}]$.

Since $G_1$ is a countable subset of $E[X_1]$, while $F$ is uncountable (see Example \ref{ex:p-adic-fields-satisfy}), we can find $x_1\in F$ such that $f(x_1)\neq 0$ for all $f\in G_1$. By the last sentence of the previous paragraph, the polynomial $f(x_1,X_2)\in E[X_2]$ is non-zero for every $f\in\mathcal{G}_2$. We can therefore choose $x_2\in F$ such that $f(x_1,x_2)\neq 0$ for all $f\in\mathcal{G}_2$ (because $\mathcal{G}_2$ is countable, but $F$ is not). Continuing in this fashion, which we can do by the last line of the previous paragraph, we see that there exists $S=(x_1,x_2,x_3,\ldots)\in F^{\mathcal{C}}$ such that $f(S)\neq 0$ for all $f\in\mathcal{F}$.
\end{proof}

Over finite fields, we have the following natural question.

\begin{question}\label{question:nondeg-pots-exist?} Let $(Q,\dtuple)$ be a 2-acyclic weighted quiver, and let $F$ be as in Example \ref{ex:finite-fields-satisfy}. Does there exist a finite-degree extension $K$ of $F$, linearly disjoint from $E/F$, such that the arrow span $A_{KE/K}$ admits a nondegenerate potential?
\end{question}

\section{A mutation invariant}\label{sec:mutation-invariants}

The main aim of this section is to show that finite-dimensionality of Jacobian algebras is invariant under mutations. Following \cite{DWZ1}, we write $\overline{e_k}=1-e_k\in R$, and for an $R$-$R$-bimodule $B$ we write
\begin{equation}
\label{eq:k-excluded}
B_{\hat k, \hat k} = \overline e_k B \overline e_k =
\bigoplus_{i,j \neq k} B_{i,j}
\end{equation}

\begin{prop}
\label{pr:k-excluded-mutation-invariant}
Suppose an SP $(A,S)$ satisfies \eqref{eq:no-2-cycles-thru-k}
and \eqref{eq:no-start-in-k}.
Then the algebras ${\mathcal P}(A,S)_{\hat k, \hat k}$
and ${\mathcal P}(\widetilde{\mu}_k (A,S))_{\hat k, \hat k}$
are isomorphic to each other.
\end{prop}

\begin{proof}
Notice that, as an $R$-algebra, $R\langle \langle \widetilde{\mu}_k (A)_{\hat k, \hat k}\rangle \rangle$
is generated by the arrows of $Q$ that are not incident to $k$, and the composite arrows $\compbwa$ for
$a \in Q_1 \cap e_k A$, $b \in Q_1 \cap A e_k$ and $\omega\in\B_k$.
The following fact is immediate from the definitions.

\begin{lemma}
\label{lem:[]-isomorphism}
The correspondence sending each $c \in Q_1 \cap A_{\hat k, \hat k}$ to itself,
and each composite arrow $\compbwa$ to the product $b\omega a$, extends to an algebra isomorphism
$$
\RA{\widetilde{\mu}_k (A)_{\hat k, \hat k}}\to{\RA{A}}_{\hat k, \hat k}.
$$
\end{lemma}

Let $u \mapsto [u]$ denote the isomorphism
${\RA{A}}_{\hat k, \hat k}
\to \RA{\widetilde{\mu}_k (A)_{\hat k, \hat k}}$
inverse of that in Lemma~\ref{lem:[]-isomorphism}.
It acts in the same way as
the correspondence $S \mapsto [S]$ in \eqref{eq:mu-k-S}:
$[u]$ is obtained by substituting $\comp{a_p}{\omega_p}{ a_{p+1}}$ for each
factor $a_p\omega_p a_{p+1}$ with $t(a_p) = h(a_{p+1}) = k$ of any path
$\omega_0a_1\omega_1 \cdots \omega_{\ell-1}a_\ell\omega_\ell$ occurring in the path expansion of~$u$.

\begin{lemma}
\label{lem:[]-Jacobian-epimorphism}
The correspondence $u \mapsto [u]$ above induces a surjective algebra homomorphism
${\mathcal P}(A,S)_{\hat k, \hat k} \to
{\mathcal P}(\widetilde{\mu}_k(A,S))_{\hat k, \hat k}$.
\end{lemma}

\begin{proof} The algebra $\RA{ A}_{\hat k, \hat k}$ is canonically embedded as a (non-unital) subalgebra of $\RA{ A}$, and the inclusion $\RA{A}_{\hat k, \hat k}\hookrightarrow \RA{A}$ sends $J(S)_{\hat k, \hat k}$ into $J(S)$. Thus we have a well-defined (non-unital) algebra homomorphism $\RA{A}_{\hat k, \hat k}/J(S)_{\hat k, \hat k}\rightarrow \mathcal{P}(A,S)$. This map is easily seen to carry $\RA{A}_{\hat k, \hat k}/J(S)_{\hat k, \hat k}$ isomorphically onto $\mathcal{P}(A,S)_{\hat k, \hat k}$ (the latter being a non-unital subalgebra of $\mathcal{P}(A,S)$). The very same reasoning shows that there is also a canonical isomorphism between $\RA{\widetilde{\mu}_k(A)}_{\hat k, \hat k}/J(\widetilde{\mu}_k(S))_{\hat k, \hat k}$ and $\mathcal{P}(\widetilde{\mu}_k(A,S))_{\hat k, \hat k}$. Now, we have the diagram
$$
\xymatrix{
\RA{A}_{\hat k, \hat k} \ar[r]^{\cong} \ar[d] & \RA{\widetilde{\mu}_k(A)_{\hat k, \hat k}} \ar@{^{(}->}[r] \ar[d] & \RA{ \widetilde{\mu}_k(A)}_{\hat k, \hat k} \ar[d]\\
\frac{\RA{A}_{\hat k, \hat k}}{J(S)_{\hat k, \hat k}} \ar[r]_{\cong} &
\frac{\RA{\widetilde{\mu}_k(A)_{\hat k, \hat k}}}{[J(S)_{\hat k, \hat k}]} \ar@{.>}[r] & \frac{\RA{\widetilde{\mu}_k(A)}_{\hat k, \hat k}}{J(\widetilde{\mu}_k(S))_{\hat k, \hat k}}
}
$$
whose first square commutes, the indicated isomorphisms being induced by the correspondence $u\mapsto [u]$. We want to show that the dotted arrow is a well-defined epimorphism induced by the inclusion $\RA{\widetilde{\mu}_k(A)_{\hat k, \hat k}}\hookrightarrow \RA{\widetilde{\mu}_k(A)}_{\hat k, \hat k}$. To do so, it is enough to prove the following two facts:
\begin{equation}
\label{eq:J-k-excluded-intersection}
[J(S)_{\hat k, \hat k}] \subseteq
R\langle \langle {\widetilde{\mu}_k(A)}_{\hat k, \hat k} \rangle \rangle
\cap  J(\widetilde{\mu}_k(S))_{\hat k, \hat k}.
\end{equation}
\begin{equation}
\label{eq:J-complements-k-excluded}
R\langle \langle \widetilde{\mu}_k(A)\rangle \rangle_{\hat k, \hat k}
= R\langle \langle \widetilde{\mu}_k(A)_{\hat k, \hat k} \rangle \rangle
+ J(\widetilde{\mu}_k(S))_{\hat k, \hat k};
\end{equation}

To show \eqref{eq:J-k-excluded-intersection}, note first that
$J(S)_{\hat k, \hat k}$ is the closure of the ideal in
${R\langle \langle A \rangle \rangle}_{\hat k, \hat k}$ generated
by the elements $\partial_c S$ for all arrows $c \in Q_1$ not incident to $k$, together with the elements
$(\partial_{a} S)\omega a'$ and
$b'\omega(\partial_{b}S)$, for $a, a' \in Q_1 \cap e_k A$, $b, b' \in Q_1 \cap A e_k $, $\omega\in\B_k$.
Let us apply the map $u \mapsto [u]$ to these generators.
First, we have:
\begin{equation}
\label{eq:[]-partial-c}
[\partial_c S] = \partial_c (\widetilde{\mu}_k (S)) \ \ \ \text{for $c\in Q_1$ not incident to $k$.}
\end{equation}
Using the equality
\begin{equation}
\label{eq:partial-ba}
\partial_{\comp{b}{\nu}{a}}[S]=
\partial_{\comp{b}{\nu}{ a}}(\widetilde{\mu}_k(S)) -a^\star \nu^{-1} b^\star  ,
\end{equation}
which follows from the definition of $\widetilde{\mu}_k(S)$ (cf. \eqref{eq:mu-k-S}), we obtain
\begin{align}
\nonumber
[(\partial_{a} S)\omega a'] &= \sum_{t(b) = k}\sum_{\nu\in\B_k} (\partial_{\comp{b}{\nu}{a}}[S])[b\nu \omega a']\\
\label{eq:[]-partial-a}
&=  \sum_{t(b) = k}\sum_{\nu\in\B_k} (\partial_{\comp{b}{\nu}{a}}(\widetilde{\mu}_k(S)) -a^\star \nu^{-1} b^\star)[b\nu \omega a']\\
\nonumber
&= \left(\sum_{t(b) = k}\sum_{\nu\in\B_k} \partial_{\comp{b}{\nu }{a}}(\widetilde{\mu}_k(S))[b\nu \omega a']\right) - \left(\sum_{t(b) = k}\sum_{\nu\in\B_k} a^\star \omega \nu^{-1} b^\star\comp{b}{\nu} {a'}\right)\\
\nonumber
&= \left(\sum_{t(b) = k}\sum_{\nu\in\B_k} \partial_{\comp{b}{\nu }{a}}(\widetilde{\mu}_k(S))[b\nu \omega a']\right) - a^\star \omega\partial_{a'^\star}(\widetilde{\mu}_k(S)),
\end{align}
and
\begin{align}
\nonumber
[b'\omega(\partial_{b}S)] &= \sum_{h(a) = k}\sum_{\nu\in\B_k} [b'\omega \nu a] (\partial_{\comp{b}{\nu}{a}}[S])\\
\label{eq:[]-partial-b}
&=  \sum_{h(a) = k}\sum_{\nu\in\B_k} [b'\omega \nu a] (\partial_{\comp{b}{\nu} {a}}(\widetilde{\mu}_k(S)) -a^\star \nu^{-1} b^\star)\\
\nonumber
&=  \left(\sum_{h(a) = k}\sum_{\nu\in\B_k} [b'\omega \nu a] \partial_{\comp{b}{\nu}{ a}}\widetilde{\mu}_k(S)\right) -\left(\sum_{h(a) = k}\sum_{\nu\in\B_k} \comp{b'}{\nu}{a} a^\star \nu^{-1}\omega b^\star\right)\\
\nonumber
&=  \left(\sum_{h(a) = k}\sum_{\nu\in\B_k} [b'\omega \nu a] \partial_{\comp{b}{\nu}{ a}}\widetilde{\mu}_k(S)\right)
-\partial_{b'^\star}(\widetilde{\mu}_k(S))\omega b^\star.
\end{align}
This implies the desired inclusion in
\eqref{eq:J-k-excluded-intersection}.

To show \eqref{eq:J-complements-k-excluded}, we note that if a
path $\omega_0\widetilde a_1 \omega_1 \cdots \omega_{\ell-1}\widetilde a_\ell \omega_\ell \in
R\langle \langle \widetilde{\mu}_k(A)\rangle \rangle_{\hat k, \hat k}$
does not belong to
$R\langle \langle \widetilde{\mu}_k(A)_{\hat k, \hat k} \rangle \rangle$
then it must contain one or more factors of the form $a^\star\omega
b^\star$ with $h(a) = t(b) = k$. Remembering \eqref{eq:partial-ba} every time such a factor appears, we deduce that
$\omega_0\widetilde a_1 \omega_1 \cdots \omega_{\ell-1}\widetilde a_\ell \omega_\ell \in
R\langle \langle \widetilde{\mu}_k(A)_{\hat k, \hat k} \rangle \rangle
+ J(\widetilde{\mu}_k(S))_{\hat k, \hat k}$, as desired.
\end{proof}

To finish the proof of
Proposition~\ref{pr:k-excluded-mutation-invariant},
it is enough to show that the epimorphism in
Lemma~\ref{lem:[]-Jacobian-epimorphism} (let us denote it by $\alpha$)
is in fact an isomorphism.
To do this, we construct the left inverse algebra homomorphism
$\beta: {\mathcal P}(\widetilde{\mu}_k(A,S))_{\hat k, \hat k} \to
{\mathcal P}(A,S)_{\hat k, \hat k}$ (so that $\beta \alpha$ is the identity
map on ${\mathcal P}(A,S)_{\hat k, \hat k}$).
We define $\beta$ as the composition of three maps.
First, we apply the epimorphism ${\mathcal P}(\widetilde{\mu}_k(A,S))_{\hat k, \hat k}
\to {\mathcal P}(\widetilde{\mu}_k\widetilde{\mu}_k(A,S))_{\hat k, \hat k}$
defined in the same way as $\alpha$.
Remembering the proof of Theorem~\ref{thm:mutation-involutive} and
using the notation introduced there, we then apply the isomorphism
${\mathcal P}(\widetilde{\mu}_k\widetilde{\mu}_k(A,S))_{\hat k, \hat k}
\to {\mathcal P}(A \oplus C, S + T)_{\hat k, \hat k}$ induced by
the automorphism $\varphi_3 \varphi_2 \varphi_1$ of
$R\langle \langle A \oplus C \rangle \rangle$.
Finally, we apply the isomorphism
${\mathcal P}(A \oplus C, S + T)_{\hat k, \hat k}
\to {\mathcal P}(A, S)_{\hat k, \hat k}$
given in Proposition~\ref{prop:jacobian-algebra-invariant}.

Since all the maps involved are algebra homomorphisms, it is
enough to check that $\beta \alpha$ fixes the generators
$p(c)$ and $p(b\omega a)$ of ${\mathcal P}(A,S)_{\hat k, \hat k}$,
where  $p$ is the projection
$R\langle \langle A \rangle \rangle \to {\mathcal P}(A,S)$, and
$a, b, c$ and $\omega$ have the same meaning as above.
This is done by direct tracing of the definitions.
\end{proof}

\begin{prop}
\label{prop:fin-dim}
Suppose an SP $(A,S)$ satisfies \eqref{eq:no-2-cycles-thru-k}
and \eqref{eq:no-start-in-k}.
If the Jacobian algebra ${\mathcal P}(A,S)$ is finite-dimensional
then so is ${\mathcal P}(\widetilde A,\widetilde S)$.
\end{prop}

\begin{proof}
We start by showing that finite dimensionality of
${\mathcal P}(A,S)$ follows from a seemingly weaker
condition.

\begin{lemma}
\label{lem:fin-dim-reduction}
Let $J \subseteq {\mathfrak m}(A)$ be a closed ideal
in $R\langle \langle A \rangle \rangle$.
Then the quotient algebra $R\langle \langle A \rangle \rangle/J$
is finite dimensional provided the subalgebra
$R\langle \langle A \rangle \rangle_{\hat k, \hat k}/J_{\hat k, \hat k}$
is finite dimensional.
In particular, the Jacobian algebra ${\mathcal P}(A,S)$ is finite-dimensional
if and only if so is the subalgebra ${\mathcal P}(A,S)_{\hat k, \hat k}$.
\end{lemma}

\begin{proof}
Similarly to \eqref{eq:k-excluded},
for an $R$-$R$-bimodule $B$, we denote
$$
B_{k, \hat k} = e_k B \overline e_k =
\bigoplus_{j \neq k} B_{k,j}, \quad
B_{\hat k, k} = \overline e_k B e_k =
\bigoplus_{i \neq k} B_{i,k}, \quad
\text{and} \quad B_{k, k} = e_k B e_k
$$
We need to show that if
$R\langle \langle A \rangle \rangle_{\hat k, \hat k}/J_{\hat k, \hat k}$
is finite dimensional then so is each of the spaces
$R\langle \langle A \rangle \rangle_{k, \hat k}/J_{k, \hat k}$,
$R\langle \langle A \rangle \rangle_{\hat k, k}/J_{\hat k, k}$ and
$R\langle \langle A \rangle \rangle_{k, k}/J_{k, k}$.
Let us treat $R\langle \langle A \rangle \rangle_{k, k}/J_{k, k}$;
the other two cases are done similarly (and a little simpler).

Let
$$
Q_1 \cap A_{k,\hat k} = \{a_1, \dots, a_s\}, \quad
Q_1 \cap A_{\hat k, k} = \{b_1, \dots, b_t\}.
$$
We have
$$
R\langle \langle A \rangle \rangle_{k, k} = \field_k e_k \oplus
\bigoplus_{\ell, m,\omega,\varpi} \omega a_\ell R\langle \langle A \rangle \rangle_{\hat k, \hat k} b_m\varpi,
$$
It follows that there is a surjective map
$\alpha: \field_k \times \Mat_{d_ks \times d_kt}(R\langle \langle A \rangle \rangle_{\hat k, \hat k})
\to R\langle \langle A \rangle \rangle_{k, k}/J_{k, k}$ given by
$$\alpha(u, C) = 
p\left(ue_k+(\B_k\boldsymbol{a})C(\boldsymbol{b}\B_k)\right)
$$
where $\Mat_{d_ks \times d_kt}(B)$ stands for the space of $d_ks \times d_kt$
matrices with entries in~$B$, $p$ is the projection
$R\langle \langle A \rangle \rangle \to R\langle \langle A \rangle \rangle/J$, and $\B_k\boldsymbol{a}$ and $\boldsymbol{b}\B_k$ are the matrices defined in the proof of Lemma \ref{lemma:extension-of-varphi}.
The kernel of $\alpha$ contains the space $\Mat_{d_ks \times d_kt}(J_{\hat k, \hat k})$,
hence $R\langle \langle A \rangle \rangle_{k, k}/J_{k, k}$ is
isomorphic to a quotient of the finite-dimensional space
$K \times \Mat_{s \times t}(R\langle \langle A \rangle \rangle_{\hat k, \hat k}/J_{\hat k, \hat k})$.
Thus, $R\langle \langle A \rangle \rangle_{k, k}/J_{k, k}$ is
finite dimensional, as desired.
\end{proof}

To finish the proof of Proposition~\ref{prop:fin-dim},
suppose that ${\mathcal P}(A,S)$ is finite dimensional.
Then ${\mathcal P}(\widetilde A,\widetilde S)_{\hat k, \hat k}$
is finite dimensional by Proposition~\ref{pr:k-excluded-mutation-invariant}.
Applying Lemma~\ref{lem:fin-dim-reduction} to the SP
$(\widetilde A,\widetilde S)$, we conclude that
${\mathcal P}(\widetilde A,\widetilde S)$ is finite dimensional, as desired.
\end{proof}

Remembering \eqref{eq:mutilde-mu} and using
Proposition~\ref{prop:jacobian-algebra-invariant}, we see that
Propositions~\ref{pr:k-excluded-mutation-invariant} and
\ref{prop:fin-dim} imply the following.

\begin{thm}
\label{thm:mutation-preserves-fin-dim}
Suppose $(A,S)$ is a reduced SP satisfying \eqref{eq:no-2-cycles-thru-k}.
Then the algebras ${\mathcal P}(A,S)_{\hat k, \hat k}$
and ${\mathcal P}(\mu_k(A,S))_{\hat k, \hat k}$
are isomorphic to each other, and
${\mathcal P}(A,S)$ is finite-dimensional
if and only if so is ${\mathcal P}(\mu_k(A,S))$.
\end{thm}

We see that the class of SPs with finite dimensional Jacobian
algebras is invariant under mutations.

\section{Restriction}\label{sec:restriction}

\begin{defi}\label{def:restriction} Let $(Q,\dtuple)$ be a strongly primitive weighted quiver, and $I\subseteq Q_0$ be a non-empty subset of the vertex set of $Q$. On the vertex set $Q_0$ we define a weighted quiver $(Q|_I,\dtuple)$ by deleting from $Q$ every arrow $c$ which is incident to at least one vertex outside $I$. If $S$ is a potential on the arrow span $A$ of $(Q,\dtuple)$ over $E/F$, we define the \emph{restriction of $(A,S)$ to $I$} to be the SP $(A,S)|_I=(A|_I,S|_I)$, where $A|_I$ is the arrow span of $(Q|_I,\dtuple)$ over $E/F$, and $S|_I$ is the image of $S$ under the $R$-algebra homomorphism $\rho:\RA{A}\rightarrow\RA{A|_I}$ defined by
$$
\rho(a)=\begin{cases}a & \text{if $a\in Q_1$ is an arrow connecting vertices in $I$;}\\
0 & \text{if $a\in Q_1$ is incident to some vertex in $Q_0\setminus I$}.
\end{cases}
$$
\end{defi}

Given any element $u\in\RA{A}$, we shall denote by $u|_I$ the image of $u$ under the homomorphism $\rho$ of Definition \ref{def:restriction}.

\begin{lemma}\label{lemma:restriction-preserves-right-equiv} Let $(A,S)$ and $(A',S')$ be SPs with the same underlying weighted quiver $(Q,\dtuple)$, and let $I$ be a non-empty finite subset of the vertex set $Q_0$. If $\varphi:\RA{A}\rightarrow\RA{A}$ is a right-equivalence $(A,S)\to (A,S')$, then the $R$-algebra homomorphism $\varphi|_I:\RA{A|_I}\rightarrow\RA{A|_I}$ defined by the rule $u\mapsto\varphi(u)|_I$ is a right-equivalence $(A|_I,S|_I)\to(A|_I,S'|_I)$. In other words, restriction preserves right-equivalences.
\end{lemma}

\begin{proof} We claim that $\varphi|_I(S|_I)=\varphi(S)|_I$. To see this, write $S = S|_I +W$, where $W\in\RA{A}$ is a potential each of whose constituting cycles passes through a vertex not in $I$. Then each term of $\varphi(W)$ passes through a vertex not in $I$, which means that $\varphi(W)|_I=0$, and hence $\varphi(S)|_I=\varphi(S|_I +W)|_I=\varphi(S|_I)|_I$. Now, the $R$-algebra homomorphism
$\rho:\RA{A}\rightarrow\RA{A|_I}$, being continuous, sends cyclically equivalent potentials to cyclically equivalent
ones, from which it follows that $\varphi(S)|_I$ is cyclically equivalent to $S'|_I$. Therefore, $\varphi|_I$ is a right-equivalence $(A|_I,S|_I)\to (A'|_I,S'|_I)$.
\end{proof}

\begin{prop}\label{prop:mut-&-restr-commute} Let $(Q,\dtuple)$ be a strongly primitive weighted quiver, and $I\subseteq Q_0$ be a non-empty subset of the vertex set of $Q$. Let $(A,S)$ be an SP on the arrow span of $(Q,\dtuple)$, and suppose that $k\in I$ is a vertex such that $Q$ does not have 2-cycles incident to $k$. Then $\mu_k(A|_I,S|_I)$ is right-equivalent to the restriction of $\mu_k(A,S)$ to $I$. In other words, SP-mutation commutes with restriction.
\end{prop}

\begin{proof} Let $\varphi:(\widetilde{\mu}_k(A),\widetilde{\mu}_k(S))\rightarrow
(\widetilde{\mu}_k(A)_{\red},\widetilde{\mu}_k(S)_{\red})\oplus(\widetilde{\mu}_k(A)_{\triv},\widetilde{\mu}_k(S)_{\triv})$ be a right-equivalence, with $(\widetilde{\mu}_k(A)_{\red},\widetilde{\mu}_k(S)_{\red})$ a reduced SP and $(\widetilde{\mu}_k(A)_{\triv},\widetilde{\mu}_k(S)_{\triv})$ a trivial SP. By Lemma \ref{lemma:restriction-preserves-right-equiv}, we have a right-equivalence $\varphi|_I:(\widetilde{\mu}_k(A)|_I,\widetilde{\mu}_k(S)|_I)\rightarrow
(\widetilde{\mu}_k(A)_{\red}|_I,\widetilde{\mu}_k(S)_{\red}|_I)\oplus(\widetilde{\mu}_k(A)_{\triv}|_I,\widetilde{\mu}_k(S)_{\triv}|_I)$. Since $(\widetilde{\mu}_k(A)_{\red}|_I,\widetilde{\mu}_k(S)_{\red}|_I)$ is clearly reduced and $(\widetilde{\mu}_k(A)_{\triv}|_I,\widetilde{\mu}_k(S)_{\triv}|_I)$ is clearly trivial, we deduce that $(\widetilde{\mu}_k(A)_{\red}|_I,\widetilde{\mu}_k(S)_{\red}|_I)$ is the reduced part of $(\widetilde{\mu}_k(A)|_I,\widetilde{\mu}_k(S)|_I)$.
The proposition then follows from the observation that $\widetilde{\mu}_k(A)|_I=\widetilde{\mu}_k(A|_I)$ and $[S]|_I+\triangle_k(A)|_I=[S|_I]+\triangle_k(A|_I)$.
\end{proof}

The next corollary as an obvious consequence of Proposition \ref{prop:mut-&-restr-commute}.

\begin{coro} If $(A,S)$ is a non-degenerate SP, then for any non-empty subset $I\subseteq Q_0$, the restriction $(A|_I,S|_I)$ is non-degenerate. In other words, non-degeneracy is stable under restriction.
\end{coro}

\section{Decorated representations and their mutations}\label{sec:reps-and-repmutations}

The definition and use of decorated representations goes back to \cite{MRZ}. Further use of them has been made in \cite{DWZ1} and \cite{DWZ2}. Among their features are the fact that they allow to keep track of vector spaces that would otherwise be lost, for example, when applying reflection functors, and the fact that they allow to keep track of initial seeds in cluster algebras.

\begin{defi}\label{def:decorated-representation} Let $(A,S)$ be an SP. A \emph{decorated representation}, or simply an \emph{SP-representation}, of
$(A,S)$ consists of
\begin{enumerate}\item A pair $M=((M_i)_{i\in Q_0},(M_{a})_{a\in Q_1})$ such
that:
\begin{itemize}
\item[(a)] For each $i\in Q_0$, $M_i$ is a finite-dimensional $F_i$-vector
space;
\item[(b)] for each arrow $a\in Q_1$, $M_{a}$ is an
$F$-linear map from $M_{t(a)}$ to $M_{h(a)}$.
\item[(c)] there exists $r\geq 0$ such that for every path $\omega_0a_1\omega_1a_2\ldots \omega_{\ell-1}a_\ell\omega_\ell$ of length $\ell$ greater than $r$ the corresponding linear transformation
$\omega_0M_{a_1}\omega_1M_{a_2}\ldots M_{a_\ell}\omega_\ell$ is the zero map;
\item[(d)] all relations of the form $\partial_a(S)$ with $a\in Q_1$ are
satisfied by $M$.
\end{itemize}
\item A tuple $V=(V_{i})_{i\in Q_0}$ such that $V_i$ is a
finite-dimensional $F_i$-vector space for every $i\in Q_0$.
\end{enumerate}
We will write $\mathcal{M}=(M,V)$ or $\mathcal{M}=((M_i)_{i\in Q_0},(M_{a})_{a\in Q_1},(V_i)_{i\in Q_0})$ to denote such an SP-representation.
\end{defi}

Equivalently, $M=\bigoplus_{i\in Q_0} M_i$ is a $\mathcal{P}(A,S)$-module and $V$ an $R$-module, both finite-dimensional over $F$.

\begin{remark} The linear maps $M_a$ in Definition \ref{def:decorated-representation} will often be denoted by $a_M$ as well.
\end{remark}

We define the notion of right-equivalence of decorated representations following \cite{DWZ1}.

\begin{defi} Let $(A,S)$ and $(A',S')$ be Ss on the same set of vertices, with the same vertex-weight function, and over the same ground field. Let $\mathcal{M}=(M,V)$ and $\mathcal{M}'=(M',V')$ be decorated representations of $(A,S)$ and $(A',S')$, respectively. A \emph{right-equivalence} between $\mathcal{M}$ and $\mathcal{M}'$ is a triple $(\varphi,\psi,\eta)$, where:
\begin{itemize}
\item $\varphi:(A,S)\rightarrow(A',S')$ is a right-equivalence;
\item $\psi:M\rightarrow M'$ is an $F$-vector space isomorphism such that $\psi\circ u_M=\varphi(u)_{M'}\circ\psi$ for all $u\in\completeQdcoprime$;
\item $\eta:V\rightarrow V'$ is an isomorphism of $R$-modules.
\end{itemize}
\end{defi}

Let $\mathcal{M}=(M,V)$ be a decorated representation of $(A,S)$, and let $\varphi:\RA{A_{\operatorname{red}}\oplus A_{\operatorname{triv}}}\rightarrow \completeQdcoprime$, be a right-equivalence of SPs $\varphi:(A_{\operatorname{red}},S_{\operatorname{red}})\oplus(A_{\operatorname{triv}},S_{\operatorname{triv}})\to(A,S)$, where $(A_{\operatorname{red}},S_{\operatorname{red}})$ and $(A_{\operatorname{triv}},S_{\operatorname{triv}})$ are a reduced and a trivial SP, respectively. We define an $\RA{A_{\operatorname{red}}\oplus A_{\operatorname{triv}}}$-module $M'$ by setting $M=M'$ as $F$-vector space, with the action of $\RA{A_{\operatorname{red}}}$ given by $u_{M'}=\varphi(u)_{M}$. Then $\mathcal{M}_{\operatorname{red}}=(M',V)$ is a decorated representation of $(A_{\operatorname{red}},S_{\operatorname{red}})$.

\begin{prop} The right-equivalence class of $\mathcal{M}_{\operatorname{red}}$ is determined by the right-equivalence class of $\mathcal{M}$.
\end{prop}

\begin{proof} The proof of \cite[Proposition 10.5]{DWZ1} is valid here.
\end{proof}

The following elementary fact concerning tensor products will be needed in order to define the notion of \emph{mutations of decorated representations}.

\begin{lemma}\label{lemma:tensor-products} Let $i$ and $j$ be different vertices of $Q$. There exist $F$-vector space isomorphisms
$$
\Hom_{F_i}(F_i\otimes_F M, N)\cong\Hom_{F}(M,N)\cong\Hom_{F_j}(M,F_j\otimes_FN),
$$
natural in the $F_j$-vector space $M$ and the $F_i$-vector space $N$.
\end{lemma}

\begin{proof} Although the proof of this lemma is rather elementary and well known, in order to establish some notation we explicitly exhibit the stated isomorphisms.

\begin{itemize}
\item[(1)] $\Hom_{F_i}(F_i\otimes_F M, N)\cong\Hom_{F}(M,N)$.
\end{itemize}
Define maps $\overrightarrow{\bullet}:\Hom_{F_i}(F_i\otimes_F M, N)\rightarrow\Hom_{F}(M,N)$ and $\overleftarrow{\bullet}:\Hom_{F}(M,N)\rightarrow\Hom_K(K\otimes_FM,N)$ according to the rules
$$
\overrightarrow{f}(m)=f(1\otimes m) \ \ \text{for} \ \ f\in\Hom_{F_i}(F_i\otimes_F M, N), \ \ \ \ \text{and} \ \ \ \
\overleftarrow{a}(e\otimes m)=ea(m) \ \ \text{for} \ \ a\in\Hom_{F}(M,N).
$$
A straightforward computation shows that $\overrightarrow{\bullet}$ and $\overleftarrow{\bullet}$ are mutually inverse $F$-vector space isomorphisms.

\begin{itemize}
\item[(2)] $\Hom_{F}(M,N)\cong\Hom_{F_j}(M,F_j\otimes_FN)$.
\end{itemize}
First of all, note that given $f\in \Hom_{F_j}(M,F_j\otimes_FN)$ and $m\in M$, it is possible to write
$$
f(m)=\sum_{\omega\in\mathcal{B}_j}\omega^{-1}\otimes 1\otimes n_{f,m,\omega^{-1}}
$$
for some elements $n_{f,m,\omega^{-1}}\in N$ uniquely determined by $f$ and $m$.
With this in mind, define maps $\overrightarrow{\bullet}:\Hom_{F}(M,N)\rightarrow\Hom_{F_j}(M,F_j\otimes_FN)$ and $\overleftarrow{\bullet}:\Hom_{F_j}(M,F_j\otimes_FN)\rightarrow\Hom_{F}(M,N)$ according to the rules
$$
\overrightarrow{b}(m)=\sum_{\omega\in\mathcal{B}_j}\omega^{-1}\otimes b(\omega m) \ \ \text{for} \ \ b\in\Hom_{F}(M,N) \ \ \ \ \text{and} \ \ \ \ \overleftarrow{f}(m)=n_{f,m,1} \ \ \text{for} \ \ f\in\Hom_{F_j}(M,F_j\otimes_FN).
 $$
Straightforward calculations show that the maps $\overrightarrow{\bullet}$ and $\overleftarrow{\bullet}$ are well-defined, mutually inverse $F$-vector space isomorphisms, independent of the eigenbasis $\mathcal{B}_j$ of $F_j/F$ chosen.

\begin{itemize}
\item[(3)] $\Hom_{F_j}(M,F_j\otimes_FN)\cong \Hom_{F_i}(F_i\otimes_F M, N)$.
\end{itemize}
The maps $\overleftarrow{\overleftarrow{\bullet}}:\Hom_{F_j}(M,F_j\otimes_FN)\rightarrow\Hom_{F_i}(F_i\otimes_F M, N)$ and $\overrightarrow{\overrightarrow{\bullet}}:\Hom_{F_i}(F_i\otimes_F M, N)\rightarrow\Hom_{F_j}(M,F_j\otimes_FN)$, obtained by composing the isomorphisms from (1) and (2) above, are obviously inverse to each other. For the convenience of the reader we display the rules these maps obey:
$$
\overleftarrow{\overleftarrow{f}}(x\otimes m)=xn_{f,m,1}, \ \ \text{and} \ \ \overrightarrow{\overrightarrow{g}}(m)=\sum_{\omega\in\mathcal{B}_j}\omega^{-1}\otimes g(1\otimes \omega m).
$$
\end{proof}

\begin{defi} Let $(Q,\dtuple)$ be a strongly primitive weighted quiver and let $k\in Q_0$ be a vertex such that $Q$ does not have 2-cycles incident to $k$. Let $A$ be the arrow span of $(Q,\dtuple)$ over $E/F$.
For each pair of arrows $a,b\in Q_1$ such that $h(a)=t(b)$, and each element $\omega\in\B_k$, we define
\begin{equation}
\label{eq:2nd-order-cyclic-derivative}
\partial_{b\omega a}(c_1\omega_1 c_2\cdots c_\ell\omega_\ell) =
\sum_{t=1}^\ell \delta_{b\omega a,c_t\omega_t c_{t+1}} \omega_{t+1} c_{t+2} \cdots c_\ell\omega_\ell c_1 \cdots c_{t-1}\omega_{t-1}
\end{equation}
for every cycle on $(Q,\dtuple)$ having the form $c_1\omega_1 c_2\cdots c_\ell\omega_\ell$, where $\delta_{b\omega a,c_t\omega_t c_{t+1}}$ is the \emph{Kronecker delta} (the $\ell^{\operatorname{th}}$ summand is obtained by setting $c_{\ell+1}=c_1$). We extend $\partial_{b\omega a}$ by linearity and continuity to the space of all potentials on $A$ (this is possible by \eqref{eq:class-of-representative-cycles} and \eqref{eq:only-0-is-cyc-equiv-to-0}).
\end{defi}

We are now ready to turn to the definition of \emph{mutations of decorated representations}.
Let $(A,S)$ be an SP, $\mathcal{M}=(M,V)$ be a decorated representation of
$(A,S)$, and $k$ a vertex of $Q$. For the rest of this section we will assume \eqref{eq:no-2-cycles-thru-k} and \eqref{eq:no-start-in-k}.

Let $a_1,\ldots,a_p$ be the
arrows of $Q$ whose head is $k$, and $b_1,\ldots,b_q$ the arrows whose tail is
$k$. Define the $F_k$-vector spaces $M_{\operatorname{in}}$ and $M_{\operatorname{out}}$ by
\begin{equation}
M_{\operatorname{in}}=\underset{s=1}{\overset{p}{\bigoplus}}
{F_k}\otimes_FM_{t(a_s)}\ \ \
\text{and} \ \ \
M_{\operatorname{out}}=\underset{r=1}{\overset{q}{\bigoplus}}{F_k}\otimes_FM_{h(b_r)}.
\end{equation}

For each $(s,r)\in[1,p]\times[1,q]$, define an $F_k$-linear map
\begin{equation}
\gamma_{sr}:{F_k}\otimes_FM_{h(b_r)}
\longrightarrow {F_k}\otimes_FM_{t(a_s)}
\end{equation}
\begin{center}
according to the rule
\end{center}
\begin{equation}\label{eq:c-rule}
\gamma_{sr}:x\otimes m\mapsto
\underset{\omega\in\B_{k}}{\sum}
x\omega\otimes\partial_{b_r\omega a_s}(S)(m).
\end{equation}

Assembling all the maps $\gamma_{sr}$ we obtain an $F_k$-linear map
$$
\gamma=\left[\begin{array}{ccc}
\gamma_{11} & \ldots & \gamma_{1q}\\
\vdots & \ddots & \vdots\\
\gamma_{p1} & \ldots & \gamma_{pq}\end{array}\right]:
M_{\operatorname{out}}\longrightarrow M_{\operatorname{in}}.
$$

We thus have a triangle of $F_k$-linear maps
\begin{equation}\label{eq:triangle-alpha-beta-gamma}
\xymatrix{
 & M_k \ar[dr]^{\beta} & \\
M_{\operatorname{in}} \ar[ur]^{\alpha} & & M_{\operatorname{out}} \ar[ll]^{\gamma},
}
\end{equation}
where the components of the maps
\begin{equation}\label{eq:alpha-&-beta-assembled}
\alpha=\left(\begin{array}{cccc}\overleftarrow{a_1} & \overleftarrow{a_2} & \ldots & \overleftarrow{a_p}\end{array}\right):\underset{s=1}{\overset{p}{\bigoplus}}{F_k}\otimes_FM_{t(a_s)}\rightarrow M_k
\ \ \ \text{and} \ \ \ \beta=\left(\begin{array}{c}\overrightarrow{b_1} \\ \overrightarrow{b_2} \\ \vdots \\ \overrightarrow{b_q}\end{array}\right):M_k\rightarrow\underset{r=1}{\overset{q}{\bigoplus}}
{F_k}\otimes_FM_{h(b_r)}
\end{equation}
are defined by means of Lemma \ref{lemma:tensor-products}.

\begin{prop}\label{prop:alphagamma-gammabeta-are-zero} The function $\gamma$ is an $F_k$-linear such
that $\gamma\beta=0$ and $\alpha\gamma=0$.
\end{prop}

\begin{proof}
Recall from the proof of Lemma \ref{lemma:tensor-products} that for each $r=1,\ldots,q$, we have
\begin{equation}\label{eq:explicitMk->FkMi}
\overrightarrow{b_r}(m)=\underset{v\in \mathcal{B}_k}{\sum}v^{-1}\otimes b_r(vm).
\end{equation}

Let us compute $\gamma\beta$:
\begin{eqnarray*}
\gamma\beta(m) & = &
\gamma(\sum_{r=1}^q\sum_{\nu\in\mathcal{B}_k}\nu^{-1}\otimes b_r(\nu m))=
\sum_{s=1}^p\sum_{r=1}^q\sum_{\nu\in\mathcal{B}_k}\gamma_{sr}(\nu^{-1}\otimes b_r(\nu m))= \\
 & = &
 \sum_{s=1}^p\sum_{r=1}^q\sum_{\nu\in\mathcal{B}_k}\underset{\omega\in\B_k}{\sum}
\nu^{-1}\omega\otimes\partial_{b_r\omega a_s}(S)(b_r(\nu m))= \\
 & = &
\sum_{s=1}^p\sum_{r=1}^q\sum_{u\in\mathcal{B}_k}\underset{\omega\in\B_k}{\sum}
u\otimes\partial_{b_r\omega a_s}(S)(b_r(\omega u^{-1}m))=  \\
 & = &
\sum_{s=1}^p\sum_{u\in\mathcal{B}_k}
u\otimes\sum_{r=1}^q\underset{\omega\in\B_k}{\sum}\partial_{b_r\omega a_s}(S)(b_r(\omega u^{-1}m))= \\
 & = &
\sum_{s=1}^p\sum_{u\in\mathcal{B}_k}
u\otimes\partial_{a_s}(S)(u^{-1}m))= 0.
\end{eqnarray*}

Showing that $\alpha\gamma=0$ is easier. For $r\in\{1,\ldots,q\}$ fixed and $x\otimes m\in {F_k}\otimes_FM_{h(b_r)}$ we have:
\begin{eqnarray*}
\alpha\gamma(x\otimes m) & = &
\alpha\left(\sum_{s=1}^p\underset{\omega\in\B_{k}}{\sum}
x\omega\otimes\partial_{b_r\omega a_s}(S)(m)\right)= \\
 & = &
\sum_{s=1}^p\underset{\omega\in\B_{k}}{\sum}
x\omega a_s\partial_{b_r\omega a_s}(S)(m)= \\
 & = &
x\partial_{b_r}(S)(m)=0.
\end{eqnarray*}
\end{proof}

We associate to $\mathcal{M}=(A,S,M,V)$ another SP-representation $\widetilde{\mu}_k(\mathcal{M})=(\widetilde{\mu}_k(A),\widetilde{\mu}_k(S),\widetilde{\mu}_k(M),\widetilde{\mu}_k(V))$ as follows.
First of all, we set
\begin{equation}
\widetilde{\mu}_k(M)_i=M_i \ \ \ \text{and} \ \ \ \widetilde{\mu}_k(V)_i=V_i \ \ \ \text{for all} \ i\neq k.
\end{equation}
We define $\widetilde{\mu}_k(M)_k$ and $\widetilde{\mu}_k(V)_k$ by
\begin{equation}
\widetilde{\mu}_k(M)_k=\frac{\ker\gamma}{\image\beta}\oplus\image\gamma\oplus\frac{\ker\alpha}{\image\gamma}\oplus V_k \ \ \ \text{and} \ \ \ \widetilde{\mu}_k(V)_k=\frac{\ker\beta}{\ker\beta\cap\image\alpha}.
\end{equation}

We now define the action on $\widetilde{\mu}_k(M)$ of all arrows in $\widetilde{\mu}_k(A)$. First, we set $c_{\widetilde{\mu}_k(M)}=c_M$ for every arrow of $Q$ not incident to $k$, and
$$
\comp{b_r}{\omega} {a_s}_{\widetilde{\mu}_k(M)}=(b_r\omega a_s)_M
$$
for all $r$, $s$ and $\omega\in\B_k$. We define the action of the remaining arrows $a_s^\star$ and $b_r^\star$ collectively through operators
$$
\overline{\alpha}=(\overleftarrow{b_1^\star} \ \ \overleftarrow{b_2^\star} \ \ \ldots \ \ \overleftarrow{b_q^\star}) \ \ \ \text{and} \ \ \
\overline{\beta}=\left(\begin{array}{c}\overrightarrow{a_1^\star}\\ \overrightarrow{a_2^\star}\\ \vdots\\ \overrightarrow{a_p^\star}\end{array}\right)
$$
Thus, we need to define linear maps
$$
\overline{\alpha}:M_{\operatorname{out}}=\widetilde{\mu}_k(M)_{\operatorname{in}}\rightarrow\widetilde{\mu}_k(M)_k \ \ \ \text{and} \ \ \ \overline{\beta}:\widetilde{\mu}_k(M)_k\rightarrow\widetilde{\mu}_k(M)_{\operatorname{out}}=M_{\operatorname{in}}.
$$
We will use the following notational convention: whenever we have
a pair $U_1 \subseteq U_2$ of vector spaces,
denote by $\iota:U_1 \to U_2$ the inclusion map, and by $\pi: U_2 \to U_2/U_1$
the natural projection.
We now introduce the following \emph{splitting data}:
\begin{align}
\label{eq:rho}
& \text{Choose an $F_k$-linear map $\rho: M_{\rm out} \to \ker \gamma$
such that $\rho \iota = {\rm id}_{\ker \gamma}$.}\\
\label{eq:sigma}
&\text{Choose an $F_k$linear map $\sigma: \ker \alpha / {\rm im}\ \gamma
\to \ker \alpha$ such that $\pi \sigma = {\rm id}_{\ker \alpha/ {\rm im}\ \gamma}$.}
\end{align}
Then we define:
\begin{equation}
\label{eq:alpha-beta-mutated}
\overline \alpha =
\begin{pmatrix}
- \pi \rho\\
- \gamma\\
0\\
0
\end{pmatrix}, \quad
\overline \beta  =
\begin{pmatrix}
0 & \iota & \iota \sigma & 0\end{pmatrix}.
\end{equation}

Having defined the action of all arrows in $\widetilde{\mu}_k (A)$ on $\widetilde{\mu}_k(M)$, we can view
$\widetilde{\mu}_k(M)$ as a module over the path
algebra $R \langle \widetilde{\mu}_k(A) \rangle$.
The fact that $M$ is annihilated by ${\mathfrak m}(A)^n$ for $n \gg
0$ implies that $\widetilde{\mu}_k(M)$ is annihilated by $\widetilde{\mu}_k(A)^n$ for $n \gg 0$.
This allows us to view $\widetilde{\mu}_k(M)$ as a module over the completed path
algebra $R \langle \langle \widetilde{\mu}_k(M) \rangle \rangle$.

\begin{prop}
\label{prop:mutation-well-defined}
The above definitions make of $\widetilde \mu_k ({\mathcal M})=(\widetilde{\mu}_k(M), \widetilde{\mu}_k(V))$
a decorated representation of $(\widetilde{\mu}_k(A), \widetilde{\mu}_k(S))$.
\end{prop}

\begin{proof}
We need to show that
$(\partial_c \widetilde{\mu}_k(S))_{\widetilde{\mu}_k(M)} = 0$ for every arrow~$c$
in $\widetilde{\mu}_k(A)$.
If~$c$ is not incident to~$k$, the desired statement follows from the following equalities
$$
\partial_c(\widetilde{\mu}_k(S))=\partial_c([S])=[\partial_c(S)].
$$

Suppose $c$ is one of the arrows $\comp{b_r}{\omega}{ a_s}$, and take $m\in M_{h(b_r)}$. We are going to compute $(a_s^\star\omega^{-1}b_r^\star)_{\widetilde{\mu}_k(M)}(m)=
(a_s^\star)_{\widetilde{\mu}_k(M)}(\omega^{-1})_{\widetilde{\mu}_k(M)}(b_r^\star)_{\widetilde{\mu}_k(M)}(m)$. To do so, we first compute $\overrightarrow{a_s^\star}\omega^{-1}\overleftarrow{b_r^\star}(1\otimes m)$. According to \eqref{eq:alpha-beta-mutated},
$$
\omega^{-1}\overleftarrow{b_r^\star}(1\otimes m)=
\begin{pmatrix}
- \omega^{-1}\pi \rho(1\otimes m)\\
\hline
- \omega^{-1}\gamma_{1r}(1\otimes m)\\
\vdots\\
-\omega^{-1}\gamma_{pr}(1\otimes m)\\
\hline
0\\
\hline
0
\end{pmatrix}=
\begin{pmatrix}
- \omega^{-1}\pi \rho(1\otimes m)\\
\hline
- \sum_{\nu\in\B_k}\omega^{-1}\nu\otimes\partial_{b_r\nu a_1}(S)(m)\\
\vdots\\
-\sum_{\nu\in\B_k}\omega^{-1}\nu\otimes\partial_{b_r\nu a_p}(S)(m)\\
\hline
0\\
\hline
0
\end{pmatrix}
$$
and hence
$$
\overrightarrow{a_s^\star}\omega^{-1}\overleftarrow{b_r^\star}(1\otimes m)=-\sum_{\nu\in\B_k}\omega^{-1}\nu\otimes\partial_{b_r\nu a_s}(S)(m).
$$ 
This means that
$$
(a_s^\star)_{\widetilde{\mu}_k(M)}(\omega^{-1})_{\widetilde{\mu}_k(M)}(b_r^\star)_{\widetilde{\mu}_k(M)}(m)=
(a_s^\star)_{\widetilde{\mu}_k(M)}(\omega^{-1}\overleftarrow{b_r^\star}(1\otimes m))=
-\partial_{b_r\omega a_s}(S)(m).
$$
The desired equality $(\partial_{\comp{b_r}{\omega}{ a_s}}(\widetilde{\mu}_k(S)))_{\widetilde{\mu}_k(M)}=0$ follows then from the equalities
$$
\partial_{\comp{b_r}{\omega}{ a_s}}([S])_{\widetilde{\mu}_k(M)}=\partial_{b_r\omega a_s}(S)_{\widetilde{\mu}_k(M)}
\ \ \ \text{and} \ \ \
\partial_{\comp{b_r}{\omega}{ a_s}}(\widetilde{\mu}_k(S))=a_s^\star\omega^{-1}b_r^\star+\partial_{\comp{b_r}{\omega}{ a_s}}([S]).
$$

It remains to show that $(\partial_{a_s^\star} (\widetilde{\mu}_k(S)))_{\widetilde{\mu}_k(M)} = 0$
and $(\partial_{b_r^\star} (\widetilde{\mu}_k(S)))_{\widetilde{\mu}_k(M)} = 0$ for all $s$ and $r$.
We first deal with $(\partial_{a_s^\star} (\widetilde{\mu}_k(S)))_{\widetilde{\mu}_k(M)}$.
By definition of the potential $\widetilde{\mu}_k(S)$ (cf. \eqref{eq:mu-k-S}
) and the action of the composite arrows $\comp{b_r}{\omega}{ a_s}$, we have
$$
(\partial_{a_s^\star} (\widetilde{\mu}_k(S)))_{\widetilde{\mu}_k(M)} =
(\sum_r \sum_{\omega}\omega^{-1}b_r^\star \comp{b_r}{\omega}{ a_s})_{\widetilde{\mu}_k(M)} =
\left(\sum_r \sum_\omega (\omega^{-1}b_q^\star)_{\widetilde{\mu}_k(M)} (b_r\omega)_M\right) (a_s)_M.
$$
Thus it suffices to show that $\sum_r \sum_\omega (\omega^{-1}b_q^\star)_{\widetilde{\mu}_k(M)} (b_r\omega)_M = 0$.
For $m\in M_k$ we have
\begin{eqnarray}\nonumber
\left(\sum_r \sum_\omega (\omega^{-1}b_q^\star)_{\widetilde{\mu}_k(M)} (b_r\omega)_M\right)(m) & = &
\sum_r \sum_\omega (\omega^{-1}b_q^\star)_{\widetilde{\mu}_k(M)} ((b_r)_M(\omega m))\\
\nonumber
=
\sum_r \sum_\omega (\omega^{-1}\overleftarrow{b_q^\star}) (1\otimes(b_r)_M(\omega m)) & = &
\begin{pmatrix}
- \pi \rho\left(\sum_r \sum_\omega \omega^{-1}\otimes(b_r)_M(\omega m)\right)\\
\hline
- \sum_r \sum_\omega \sum_{\nu\in\B_k}\omega^{-1}\nu\otimes\partial_{b_r\nu a_1}(S)((b_r)_M(\omega m))\\
\vdots\\
-\sum_r \sum_\omega \sum_{\nu\in\B_k}\omega^{-1}\nu\otimes\partial_{b_r\nu a_p}(S)((b_r)_M(\omega m))\\
\hline
0\\
\hline
0
\end{pmatrix}\\
\nonumber
=
\begin{pmatrix}
- \pi \rho\beta (m)\\
\hline
-\gamma\beta(m)\\
\hline
0\\
\hline
0
\end{pmatrix} & = & 0.
\end{eqnarray}

Finally, we shall show that $(\partial_{b_r^\star} (\widetilde{\mu}_k(S)))_{\widetilde{\mu}_k(M)} = 0$. We have
$$
(\partial_{b_r^\star} (\widetilde{\mu}_k(S)))_{\widetilde{\mu}_k(M)}=
(\sum_s\sum_\omega[{b_r}_\omega a_s]a_s^\star\omega^{-1})_{\widetilde{\mu}_k(M)}=
(b_r)_M\sum_s\sum_\omega(\omega a_s)_M(a_s^\star\omega^{-1})_{\widetilde{\mu}_k(M)}.
$$
Thus it suffices to show that $\sum_s\sum_\omega(\omega a_s)_M(a_s^\star\omega^{-1})_{\widetilde{\mu}_k(M)}=0$. Since $\alpha\gamma=0$, we have
$$
\alpha \overline \beta = \begin{pmatrix}
0 & \alpha \iota & \alpha \iota \sigma & 0\end{pmatrix} = 0.
$$
Take $m=(m_1+\image\beta,m_2,m_3+\image\gamma,m_4)
\in\frac{\ker\gamma}{\image\beta}\oplus\image\gamma\oplus\frac{\ker\alpha}{\image\gamma}\oplus V_k=\widetilde{\mu}_k(M)_k$, assume without loss of generality that $m_3=\sigma(m_3+\image\gamma)$,
and write
$$
m_2=\begin{pmatrix}
\sum_{\nu\in\B_k}\nu\otimes n_{1\nu}\\
\vdots\\
\sum_{\nu\in\B_k}\nu\otimes n_{p\nu}
\end{pmatrix},
m_3=\begin{pmatrix}
\sum_{\nu\in\B_k}\nu\otimes n'_{1\nu}\\
\vdots\\
\sum_{\nu\in\B_k}v\otimes n'_{p\nu}
\end{pmatrix}\in \bigoplus_{s=1}^pF_k\otimes_FM_{t(a_s)}=M_{\operatorname{in}}.
$$
Then
$(a_s^\star\omega^{-1})_{\widetilde{\mu}_k(M)}(m_1+\image\beta,m_2,m_3+\image\gamma,m_4)=
(a_s^\star)_{\widetilde{\mu}_k(M)}(\omega^{-1}m_1+\image\beta,\omega^{-1}m_2,\omega^{-1}m_3+\image\gamma,\omega^{-1}m_4)=
n_{s\omega}+n'_{s\omega}$, and hence
\begin{eqnarray}\nonumber
&& \sum_s\sum_\omega(\omega a_s)_M(a_s^\star\omega^{-1})_{\widetilde{\mu}_k(M)}(m_1+\image\beta,m_2,m_3+\image\gamma,m_4)  = \sum_s\sum_\omega(\omega a_s)_M(n_{s\omega}+n'_{s\omega})\\
\nonumber & = &
\sum_s\sum_\omega\overleftarrow{a_s}(\omega \otimes n_{s\omega}+\omega \otimes n'_{s\omega})  =
\sum_s\overleftarrow{a_s}\left(\sum_\omega\omega \otimes n_{s\omega}\right)+\sum_s\overleftarrow{a_s}\left(\sum_\omega\omega \otimes n'_{s\omega}\right)\\
\nonumber  & = &
\alpha(m_2)+\alpha(m_3)=\alpha\iota(m_2)+\alpha\iota\sigma(m_3+\image\gamma)  =  0.
\end{eqnarray}

Proposition \ref{prop:mutation-well-defined} is proved.
\end{proof}

\begin{prop} The isomorphism class of the decorated representation $\widetilde{\mu}_k(\mathcal{M})$ does not depend on the choice of the splitting data \eqref{eq:rho} -- \eqref{eq:sigma}.
\end{prop}

\begin{proof} The proof of \cite[Proposition 10.9]{DWZ1} is valid here.
\end{proof}

\begin{prop}\label{prop:rep-right-equiv-class-well-defined} The right-equivalence class of the decorated representation $\widetilde{\mu}_k(\mathcal{M})$ is determined by the right-equivalence class of $\mathcal{M}$
\end{prop}

\begin{proof} Let $\varphi$ be an $R$-algebra automorphism of $\RA{A}$. Define a decorated representation $\mathcal{M}'=(M',V')$ as follows: $V'=V$ and $M'=M$ as left $R$-modules, while the actions of $\RA{A}$ on $M$ and $M'$ are related by
$$
u_{M'}=\varphi^{-1}(u)_M
$$
for $u\in \RA{A}$. By Proposition \ref{prop:automorphism-respects-jacobian}, $\mathcal{M}'$ is a decorated representation of $(A,\varphi(S))$. To prove Proposition \ref{prop:rep-right-equiv-class-well-defined} we need to show that there exist $R$-linear maps $\widehat{\varphi}:\RA{\widetilde{\mu}_k(A)}\rightarrow \RA{\widetilde{\mu}_k(A)}$, $\psi:\widetilde{\mu}_k(M)\rightarrow\widetilde{\mu}_k(M)$ and $\eta:\widetilde{\mu}_k(V)\rightarrow \widetilde{\mu}_k(V)$, such that the triple $(\widehat{\varphi},\psi,\eta)$ is a right-equivalence $\mathcal{M}\rightarrow\mathcal{M}$. To start the proof, define matrices
$$
\B_k\boldsymbol{a}, \ \boldsymbol{b}\B_k, \ C, \ \text{and} \ D,
$$
as in the proof of Lemma \ref{lemma:extension-of-varphi}.

In particular, $C$ can be divided into $p^2$ blocks of size $d_k\times d_k$ as in \eqref{eq:C-divided-in-blocks}, the blocks possessing the form \eqref{eq:C-typical-block}, and $D$ can be divided into $q^2$ blocks of size $d_k\times d_k$ as in \eqref{eq:D-divided-in-blocks}, with the blocks having the form \eqref{eq:D-typical-block}. We denote by $D^\intercal$ the matrix obtained from $D$ by transposing each block $D_{ij}$. Thus, $D^\intercal$ is a $d_kq\times d_kq$ matrix divided into $q\times q$ blocks, each of size $d_k\times d_k$, the $ij$-block being the transpose of $D_{ij}$.

Let $\Gamma$ be the $d_kp\times d_kq$ matrix organized in $p\times q$ blocks
$$
\Gamma=\left[\begin{array}{ccc}
\Gamma_{11} & \ldots & \Gamma_{1q}\\
\vdots & \ddots & \\
\Gamma_{p1} & \ldots & \Gamma_{pq}\end{array}\right],
$$
the blocks being given by
{\small$$
\Gamma_{ij}=\left[\begin{array}{rrrrrrrr}
\partial_{b_ja_i}(S) & v_k^{d_k}\partial_{b_jv_k^{d_k-1}a_i}(S) & v_k^{d_k}\partial_{b_jv_k^{d_k-2}a_i}(S) & & 
 & v_k^{d_k}\partial_{b_jv_k^2a_i}(S) & v_k^{d_k}\partial_{b_jv_ka_i}(S)\\
\partial_{b_jv_ka_i}(S) & \partial_{b_ja_i}(S) & v_k^{d_k}\partial_{b_jv_k^{d_k-1}a_i}(S) & \ldots & 
 & v_k^{d_k}\partial_{b_jv_k^{3}a_i}(S) & v_k^{d_k}\partial_{b_jv_k^{2}a_i}(S)\\
\partial_{b_jv_k^2a_i}(S) & \partial_{b_jv_ka_i}(S) & \partial_{b_ja_i}(S) & & 
& v^{d_k}\partial_{b_jv_k^{4}a_i}(S) &v_k^{d_k}\partial_{b_jv_k^{3}a_i}(S)\\
 & \vdots & & \ddots & & \vdots &\\
\partial_{b_jv_k^{d_k-3}a_i}(S) & \partial_{b_jv_k^{d_k-4}a_i}(S) & \partial_{b_jv_k^{d_k-5}a_i}(S) & & 
& v_k^{d_k}\partial_{b_jv_k^{d_k-1}a_i}(S) & v_k^{d_k}\partial_{b_jv_k^{d_k-2}a_i}(S)\\
\partial_{b_jv_k^{d_k-2}a_i}(S) & \partial_{b_jv^{d_k-3}a_i}(S) & \partial_{b_jv_k^{d_k-4}a_i}(S) & \ldots & 
&  \partial_{b_ja_i}(S) & v_k^{d_k}\partial_{b_jv_k^{d_k-1}a_i}(S)\\
\partial_{b_jv_k^{d_k-1}a_i}(S) & \partial_{b_jv_k^{d_k-2}a_i}(S) & \partial_{b_jv_k^{d_k-3}a_i}(S) & & 
& \partial_{b_jv_ka_i}(S) & \partial_{b_ja_i}(S)
\end{array}\right].
$$}
We define a matrix $\Gamma'$ in a similar way by taking cyclic derivatives of $\varphi(S)$ (rather than cyclic derivatives of $(S)$). By $\varphi(\Gamma)$ we denote the matrix obtained by applying $\varphi$ componentwise to $\Gamma$.

We use the matrices
\begin{equation}\label{eq:matrices-ABCDGamma}
\B_k\boldsymbol{a}, \ \boldsymbol{b}\B_k, \ C, \ D, \ \Gamma, \ \Gamma' \ \text{and} \ \varphi(\Gamma),
\end{equation}
to define $F_k$-linear maps between the spaces $M_k$, $M_{\operatorname{in}}$ and $M_{\operatorname{out}}$. This requires some preparation.

First of all, we note that he map
$$
\sum_{\ell=0}^{d_k-1}v_k^\ell\otimes m_\ell\mapsto(m_0,m_1,\ldots,m_{d_k-1})
$$ induces an $F$-vector space decomposition
\begin{equation}\label{eq:Min-dec-type-I}
M_{\operatorname{in}}=\bigoplus_{p=1}^sF_k\otimes_FM_{t(a_p)}
\cong\bigoplus_{p=1}^s\bigoplus_{\ell=0}^{d_k-1}M_{t(a_p)}
\end{equation}
 and an $F$-vector space decomposition
\begin{equation}\label{eq:Mout-dec-type-I}
M_{\operatorname{out}}=\bigoplus_{q=1}^tF_k\otimes_FM_{h(b_q)}
\cong\bigoplus_{q=1}^t\bigoplus_{\ell=0}^{d_k-1}M_{h(b_q)}.
\end{equation}
We refer to these as \emph{$\B_k$-type decompositions} (for lack of a better term).
Similarly, the map
$$
\sum_{\ell=0}^{d_k-1}v_k^{-\ell}\otimes m_\ell\mapsto(m_0,m_1,\ldots,m_{d_k-1})
$$
(notice the negative powers of $v_k$) induces an $F$-vector space decomposition
\begin{equation}\label{eq:Min-dec-type-II}
M_{\operatorname{in}}=\bigoplus_{p=1}^sF_k\otimes_FM_{t(a_p)}
\cong\bigoplus_{p=1}^s\bigoplus_{\ell=0}^{d_k-1}M_{t(a_p)}
\end{equation}
and an $F$-vector space decomposition
\begin{equation}\label{eq:Mout-dec-type-II}
M_{\operatorname{out}}=\bigoplus_{q=1}^tF_k\otimes_FM_{h(b_q)}
\cong\bigoplus_{q=1}^t\bigoplus_{\ell=0}^{d_k-1}M_{h(b_q)}.
\end{equation}
We refer to these as \emph{$B_k^{-1}$-type decompositions} (also for lack of a better term).

All of the matrices \eqref{eq:matrices-ABCDGamma} have entries in the complete path algebra $\RA{A}$. Every such entry $u$ belongs to $e_{k_1}\RA{A}e_{k_2}$ for some pair of vertices $k_1,k_2\in Q_0$, and it hence induces an $F$-linear map $u_{M}:M_{k_2}\to M_{k_1}$ (resp. $u_{M'}:M_{k_2}'\to M_{k_1}'$) given by $m\mapsto um$.
The matrices \eqref{eq:matrices-ABCDGamma} thus give rise to matrices $(\B_k\boldsymbol{a})_M$, $(\boldsymbol{b}\B_k)_M$, $C_M$, $D_M$, $\Gamma_M$, $\Gamma'_M$ and $\varphi(\Gamma)_M$ (resp. $(\B_k\boldsymbol{a})_{M'}$, $(\boldsymbol{b}\B_k)_{M'}$, $C_{M'}$, $D_{M'}$, $\Gamma_{M'}$, $\Gamma'_{M'}$ and $\varphi(\Gamma)_{M'}$), whose entries are $F$-linear maps between the spaces attached by $M$ (resp. $M'$) to the vertices of $Q$.
The sizes of these matrices of $F$-linear maps match the number of direct summands in the right-hand side of the decompositions \eqref{eq:Min-dec-type-I}, \eqref{eq:Min-dec-type-II}, \eqref{eq:Mout-dec-type-I}, \eqref{eq:Mout-dec-type-II} of $M_{\operatorname{in}}$ and
$M_{\operatorname{out}}$.

With respect to the $\B_k$-type decompositions of $M_{\operatorname{in}}$ and $M_{\operatorname{out}}$ we define $F$-linear maps
\begin{eqnarray}
\overline{(\B_k\boldsymbol{a})_{M}}, \ \ \ \overline{(\B_k\boldsymbol{a})_{M'}}&:&M_{\operatorname{in}}\to M_k,\\
\nonumber
\overline{C_{M'}}&:&M_{\operatorname{in}}\to M_{\operatorname{in}},\\
\nonumber
\overline{D^\intercal_{M'}}&:&M_{\operatorname{out}}\to M_{\operatorname{out}},\\
\nonumber
\overline{\Gamma_M}, \ \ \ \overline{\Gamma'_{M'}}, \ \ \ \overline{\varphi(\Gamma)_{M'}}&:&M_{\operatorname{out}}\to M_{\operatorname{in}},
\end{eqnarray}
via matrix multiplication, evaluating the entries of the matrices $(\B_k\boldsymbol{a})_{M}$, $(\B_k\boldsymbol{a})_{M'}$, $C_{M'}$,
$D^\intercal_{M'}$, $\Gamma_M$, $\Gamma'_{M'}$ and $\varphi(\Gamma)_{M'}$ at the corresponding elements of $M_{\operatorname{in}}=\bigoplus_{p=1}^s\bigoplus_{\ell=0}^{d_k-1}M_{t(a_p)}$ and $M_{\operatorname{out}}=\bigoplus_{q=1}^t\bigoplus_{\ell=0}^{d_k-1}M_{h(b_q)}$.

With respect to the $\B_k^{-1}$-type decompositions of $M_{\operatorname{in}}$ and $M_{\operatorname{out}}$ we define $F$-linear maps
\begin{eqnarray}
\underline{(\boldsymbol{b}\B_k)_{M}}, \ \ \ \underline{(\boldsymbol{b}\B_k)_{M'}}&:&M_k\to M_{\operatorname{out}},\\
\nonumber
\underline{D_{M'}}&:&M_{\operatorname{out}}\to M_{\operatorname{out}},
\end{eqnarray}
also via matrix multiplication and evaluation of the $F$-linear maps which are entries of the corresponding matrices.

An easy check shows that all of these maps are actually $F_k$-linear and that we have the following identities:
$$
\overline{(\B_k\boldsymbol{a})_M}=\alpha, \ \ \ \ \overline{(\B_k\boldsymbol{a})_{M'}}=\alpha', \ \ \ \ \alpha=
\alpha'\overline{C_{M'}},
$$
$$
\underline{(\boldsymbol{b}\B_k)_{M}}=\beta, \ \ \ \ \underline{(\boldsymbol{b}\B_k)_{M'}}=\beta', \ \ \ \ \beta=
\underline{D_{M'}}\beta',
$$
$$
\overline{D^\intercal_{M'}}=\underline{D_{M'}}, \ \ \ \ \overline{\Gamma_M}=\overline{\varphi(\Gamma)_{M'}}=\gamma, \ \ \ \ \overline{\Gamma'_{M'}}=\gamma'.
$$

Now, for each $i\in\{1,\ldots,p\}$, each $j\in\{1,\ldots,q\}$, and each $\ell\in\{0,\ldots,d_k-1\}$, we have the following crucial identity in the complete path algebra $\RA{A}$:
\begin{eqnarray}\label{eq:mutreps-preserves-rightequiv-crucial-identity}
\Gamma'_{d_k(i-1)+\ell,d_k(j-1)} & = & 
(C\varphi(\Gamma)(D^\intercal))_{d_k(i-1)+\ell,d_k(j-1)}\\
\nonumber &+ &\sum_{x=1}^q[\Delta_{[b_jv^\ell a_i]}(\varphi(b_x))\square\varphi(\partial_{b_x}(S))]\\
\nonumber &+& \sum_{y=1}^p[\Delta_{[b_jv^\ell a_i]}(\varphi(a_y))\square\varphi(\partial_{a_y}(S))].
\end{eqnarray}
Hence
\begin{eqnarray}\nonumber
&&(\Gamma'_{d_k(i-1)+\ell,d_k(j-1)})_{M'}=((C\varphi(\Gamma)(D^\intercal))_{d_k(i-1)+\ell,d_k(j-1)})_{M'}\\
\nonumber
&=&  (C_{M'}\varphi(\Gamma)_{M'}(D^\intercal_{M'}))_{d_k(i-1)+\ell,d_k(j-1)}=(C_{M'}\Gamma_M(D^\intercal_{M'}))_{d_k(i-1)+\ell,d_k(j-1)},
\end{eqnarray}
and therefore, $\gamma'=\overline{C_{M'}}\gamma\overline{D^{\intercal}_{M'}}$.

At this point the proof starts going in pretty much the same way as the proof of \cite[Proposition 10.10]{DWZ1}. Indeed, 
we can now deduce that
$$
\ker\alpha=(\overline{C_{M'}})^{-1}(\ker\alpha'), \ \ \ \image\alpha=\image\alpha',
$$
$$
\ker\beta=\ker\beta', \ \ \  \image\beta=\overline{D^{\intercal}_{M'}}(\image\beta')
$$
$$
\ker\gamma=\overline{D^{\intercal}_{M'}}(\ker\gamma') \ \ \  \image\gamma=(\overline{C_{M'}})^{-1}(\image\gamma'),
$$
and from this we see that the sections $\sigma:\ker\alpha/\image\gamma\to\ker\alpha$, $\sigma':\ker\alpha'/\image\gamma'\to\ker\alpha'$, and the retractions $\rho:M_{\operatorname{out}}\to\ker\gamma$, $\rho':M_{\operatorname{out}}\to\ker\gamma'$, can be chosen in such a way that the diagrams
$$
\xymatrix{\ker\alpha/\image\gamma \ar[r]^{\overline{C_{M'}}} \ar[d]_{\sigma} & \ker\alpha'/\image\gamma' \ar[d]^{\sigma'}\\
\ker\alpha \ar[r]_{\overline{C_{M'}}} & \ker\alpha'
} \ \ \ \ \
\xymatrix{M_{\operatorname{out}} \ar[r]^{(\overline{D^{\intercal}_{M'}})^{-1}} \ar[d]_{\rho} & M_{\operatorname{out}} \ar[d]^{\rho'}\\
\ker\gamma \ar[r]_{(\overline{D^{\intercal}_{M'}})^{-1}} & \ker\gamma'
}
$$
commute (here we are making a small notational abuse by denoting the map $\ker\alpha/\image\gamma\to\ker\alpha'/\image\gamma'$ with the symbol $\overline{C_{M'}}$). The action of the arrows $a_1^\star,\ldots,a_p^\star$ and $b_1^\star,\ldots,b_q^\star$ on $\widetilde{\mu}_k(M)$ and $\widetilde{\mu}_k(M')$ is then given by
$$
\overline{\alpha}=\left[\begin{array}{c}
-\pi\rho\\
-\gamma\\
0\\
0\end{array}\right], \ \ \ \ \
\overline{\alpha'}=\left[\begin{array}{c}
-\pi(\overline{D^{\intercal}_{M'}})^{-1}\rho\overline{D^{\intercal}_{M'}}\\
-\overline{C_{M'}}\gamma\overline{D^{\intercal}_{M'}}\\
0\\
0\end{array}\right]
$$
$$
\overline{\beta}=\left[\begin{array}{cccc}0 \\ \iota \\ \iota\sigma \\ 0\end{array}\right], \ \ \ \ \
\overline{\beta'}=\left[\begin{array}{cccc}0 \\ \iota \\ \iota\overline{C_{M'}}\sigma(\overline{C_{M'}})^{-1} \\ 0\end{array}\right].
$$

We can now define a right-equivalence $(\widehat{\varphi},\psi,\eta)$ between $\widetilde{\mu}_k(\mathcal{M})$ and
$\widetilde{\mu}_k(\mathcal{M}')$. We take $\widehat{\varphi}:\RA{\widetilde{\mu}_k(A)}\to
\RA{\widetilde{\mu}_k(A)}$  to be the right-equivalence
$(\widetilde{\mu}_k(A),\widetilde{\mu}_k(S))\to(\widetilde{\mu}_k(A),\widetilde{\mu}_k(\varphi(S)))$ constructed in the proof of Lemma
\ref{lemma:extension-of-varphi}. Next, we define $\psi:\widetilde{\mu}_k(M)\to\widetilde{\mu}_{k}(M')$ as the identity map on
$\bigoplus_{i\neq k}M_i=\bigoplus_{i\neq k}M_i'$, and the restriction $\psi|_{\widetilde{\mu}_k(M)_k}:\widetilde{\mu}_k(M)_k\to\widetilde{\mu}_k(M')_k$ given by the block-diagonal matrix
$$
\psi|_{\widetilde{\mu}_k(M)_k}=\left[
\begin{array}{cccc}
(\overline{D^{\intercal}_{M'}})^{-1} & 0 & 0 & 0\\
0 & \overline{C_{M'}} & 0 & 0\\
0 & 0 & \overline{C_{M'}} & 0\\
0 & 0 & 0 & \mathbf{1}
\end{array}\right]
$$
(here we are making a notational abuse when denoting by $(\overline{D^{\intercal}_{M'}})^{-1}$ the linear map $\ker\gamma/\image\beta\to\ker\gamma'/\image\beta'$ induced by $(\overline{D^{\intercal}_{M'}})^{-1}:\ker\gamma\to\ker\gamma'$). Finally, we set $\eta:\widetilde{\mu}_k(V)\to\widetilde{\mu}_k(V')$ to be the identity map.

We need to check the equality $\psi\circ (c_{\widetilde{\mu}_k(M)})=(\widehat{\varphi}(c)_{\widetilde{\mu}_k(M')})\circ\psi$ for every arrow $c$ of $\widehat{\mu}_k(Q)$. The only arrows for which the equality may not be clear are those incident to $k$. Consider the matrices $\boldsymbol{a^\star}\B_k^{-1}$ and $\B_k^{-1}\boldsymbol{b^\star}$, whose entries clearly lie in $\RA{\widetilde{\mu}_k(A)}$,
and notice that, with respect to the $\B_k$-type decomposition of $M_{\operatorname{in}}$, we have the equalities
$$
\overline{\beta}=\overline{(\boldsymbol{a^\star}\B_k^{-1})_{\widetilde{\mu}_k(M)}} \ \ \  \text{and} \ \ \ \overline{\beta'}=\overline{(\boldsymbol{a^\star}\B_k^{-1})_{\widetilde{\mu}_k(M')}},
$$
whereas with respect to the $\B_k^{-1}$-type decomposition of $M_{\operatorname{out}}$, we have the equalities
$$
\overline{\alpha}=\underline{(\B_k^{-1}\boldsymbol{b^\star})_{\widetilde{\mu}_k(M)}} \ \ \ \text{and} \ \ \ \overline{\alpha'}=\underline{(\B_k^{-1}\boldsymbol{b^\star})_{\widetilde{\mu}_k(M')}}.
$$
Note also that $\left(\overline{C_{M'}}\right)^{-1}=\overline{C_{M'}^{-1}}$ and $\left(\overline{D^{\intercal}_{M'}}\right)^{-1}=\overline{(D^{\intercal}_{M'})^{-1}}$. Therefore,
\begin{eqnarray}\nonumber
&&\overline{(\boldsymbol{a^\star}\B_k^{-1})_{\widetilde{\mu}_k(M)}}=
\overline{\beta}=\left(\overline{C_{M'}}\right)^{-1}\overline{\beta'}\circ\psi|_{\widetilde{\mu}_k(M)_k}=
\overline{C_{M'}^{-1}}\overline{(\boldsymbol{a^\star}\B_k^{-1})_{\widetilde{\mu}_k(M')}}\circ\psi|_{\widetilde{\mu}_k(M)_k}\\
\nonumber &=&
\overline{C_{M'}^{-1}(\boldsymbol{a^\star}\B_k^{-1})_{\widetilde{\mu}_k(M')}}\circ\psi|_{\widetilde{\mu}_k(M)_k}=
\overline{\widehat{\varphi}((\boldsymbol{a^\star}\B_k^{-1}))_{\widetilde{\mu}_k(M')}}\circ\psi|_{\widetilde{\mu}_k(M)_k}
\end{eqnarray}
and
\begin{eqnarray}
&& \psi|_{\widehat{\mu}_k(M)_k}\circ\underline{(\B_k^{-1}\boldsymbol{b^\star})_{\widetilde{\mu}_k(M)}}=
\psi|_{\widehat{\mu}_k(M)_k}\circ\overline{\alpha}=
\overline{\alpha'}\left(\overline{D^{\intercal}_{M'}}\right)^{-1}=
\underline{(\B_k^{-1}\boldsymbol{b^\star})_{\widetilde{\mu}_k(M')}}\left(\underline{D_{M'}}\right)^{-1}\\
\nonumber
&=&
\underline{(\B_k^{-1}\boldsymbol{b^\star})_{\widetilde{\mu}_k(M')}}\underline{D_{M'}^{-1}}=
\underline{(\B_k^{-1}\boldsymbol{b^\star})_{\widetilde{\mu}_k(M')}D_{M'}^{-1}}=
\underline{\widehat{\varphi}(\B_k^{-1}\boldsymbol{b^\star})_{\widetilde{\mu}_k(M')_k}},
\end{eqnarray}
and this proves that the equality $\psi\circ (c_{\widetilde{\mu}_k(M)})=(\widehat{\varphi}(c)_{\widetilde{\mu}_k(M')})\circ\psi$ holds for every arrow $c$ of $\widehat{\mu}_k(Q)$ which is incident to $k$.
\end{proof}

\begin{thm} The mutation $\mu_k$ of decorated representations is an involution; that is, for every decorated representation $\mathcal{M}$ of a reduced SP $(A,S)$ satisfying \eqref{eq:no-2-cycles-thru-k}, the decorated representation $\mu^2_k(\mathcal{M})$ of $\mu^2_k(A,S)$ is right-equivalent to $\mathcal{M}$.
\end{thm}

\begin{proof} The crucial identity $\overline{\gamma}=\beta\alpha$ 
can be proved using the $\B_k$-type decompositions of the spaces $M_{\operatorname{in}}$ and $M_{\operatorname{out}}$. Once this identity is known to hold, the proof of \cite[Theorem 10.13]{DWZ1} can be applied here \emph{as is}.
\end{proof}

Just as \cite{DWZ1}, let us note here that direct sums of decorated representations of a given $(A,S)$ are defined in the obvious way, that the relation of right-equivalence respects direct sums and indecomposability, and that mutations send direct sums to direct sums. This, combined with the involutivity of $\mu_k$, give the following.

\begin{coro} Any mutation $\mu_k$ is an involution on the set of right-equivalence classes of indecomposable decorated representations of reduced SPs satisfying \ref{eq:no-2-cycles-thru-k}.
\end{coro}

Following \cite{DWZ1}, we say that a decorated representation $\mathcal{M}=(M,V)$ is \emph{positive} if $V=0$, and \emph{negative} if $M=0$. Thus, all indecomposable decorated representation are either positive or negative, and the indecomposable positive ones are just indecomposable $\mathcal{P}(A,S)$-modules. In particular, for every vertex $k$, the \emph{simple} representation $\mathcal{S}_k(A,S)$ is the indecomposable positive representation of $(A,S)$ such that $\dim M_i=\delta_{i,k}$, and the \emph{negative simple} representation $\mathcal{S}_k^-(A,S)$ is the indecomposable negative representation such that $\dim V_i=\delta_{i,k}$.

\begin{prop} Any indecomposable decorated representation is either positive, or negative simple. Furthermore,
$$
\mu_k(\mathcal{S}_k(A,S))=\mathcal{S}_k^-(\mu_k(A,S)), \ \ \ \mu_k(\mathcal{S}_k^-(A,S))=\mathcal{S}_k(\mu_k(A,S));
$$
and this is the only mutation that interchanges positive and negative indecomposable representations.
\end{prop}

To close the section, we show that mutations of decorated representations commute with  extension of scalars.

\begin{prop} Let $\mathcal{M}=(M,V)$ be a decorated representation of $(A_{E/F},S)$. For each finite-degree extension $K/F$ linearly disjoint from $E/F$, set $K\otimes_F\mathcal{M}=(K\otimes_FM, K\otimes_FV)$, where $K\otimes_FM=((K\otimes_FM_i)_{i\in Q_0},(a_{K\otimes_FM})_{a\in Q_1})$ and $K\otimes_FV=(K\otimes_FV_i)_{i\in Q_0}$. Then $K\otimes_F\mathcal{M}$ is a decorated representation of $(A_{KE/K},S)$. Furthermore, if $k\in Q_0$ is a vertex such that $Q$ does not have 2-cycles incident to $k$, then $\mu_k(K\otimes_F\mathcal{M})$ and $K\otimes_F\mu_k(\mathcal{M})$ are right-equivalent.
\end{prop}

\begin{proof} Since $M_i$ is an $F_i$-vector space for every $i\in Q_0$, the left $K$-vector space $K\otimes_FM_i$ is actually a left $KF_i$-vector space if we define $y(x\otimes m)=x\otimes ym$ for $y\in F_i$ (that this gives a well-defined left action of $F_i$ on $K\otimes_FM_i$ follows from the fact that $K\cap F_i=F$). Furthermore, for any given arrow $a\in Q_1$, we have $a_{K\otimes_FM}=K\otimes_Fa_M$, that is, $a_{K\otimes_FM}(x\otimes m)=x\otimes a_{M}(m)\in K\otimes_FM_{h(a)}$ for $x\in K$ and $m\in M_{t(a)}$. Hence, the fact that $K\otimes_F\mathcal{M}$ is a decorated representation of $(A_{KE/K},S)$ follows by directly checking that the conditions in Definition \ref{def:decorated-representation} are satisfied.

We have the following facts:
\begin{itemize}
\item Given the triangle of $F_k$-linear maps \eqref{eq:triangle-alpha-beta-gamma} induced by $M$ as a representation of $(A,S)$,  the corresponding triangle of $KF_k$-linear maps induced by $K\otimes_FM$ as a representation of $(A_{KE/K},S)$ can be obtained from \eqref{eq:triangle-alpha-beta-gamma} by applying $K\otimes_F-$ to all spaces and maps involved.
\item $K\otimes_F-$ is an exact functor.
\item The splitting data \eqref{eq:rho} and \eqref{eq:sigma} corresponding to $M$ can be chosen in such a way that applying $K\otimes_F-$ to it produces a splitting data for $K\otimes_FM$.
\end{itemize}
These facts, together with Propositions \ref{prop:red-commutes-with-ext-of-scalars} and \ref{prop:SPmut-commutes-with-ext-of-scalars}, imply that $\mu_k(K\otimes_F\mathcal{M})$ and $K\otimes_F\mu_k(\mathcal{M})$ are right-equivalent.
\end{proof}

\section{Relation to Dlab-Ringel reflection functors}\label{sec:Dlab-Ringel}

In this section we show that in the special case where the mutating vertex $k\in Q_0$ is a sink or a source, our definition of mutation of representations specializes to the classical reflection functors of Dlab-Ringel (at the level of objects, for we have not defined mutations functorially). Our reference for the definition of Dlab-Ringel reflection functors is \cite{DR}.

Let us first recall the definition of Dlab-Ringel reflection functors at sinks. Let $A$ be the arrow span of $(Q,\dtuple)$ over $E/F$,
and suppose that $k$ is a sink of $Q$. 
Let
$M=((M_i)_{i\in Q_0},(M_{a})_{a\in Q_1})$ be a representation of $A$. That is, for each $i\in Q_0$, $M_i$ is a finite-dimensional (left) $F_i$-vector
space, and for each arrow $a\in Q_1$, $M_{a}$ is an $F$-linear map from $M_{t(a)}$ to $M_{h(a)}$. Let $\mu_k(A)$ be the arrow span of $\mu_k(Q,\dtuple)$ over $E/F$ (note that since $k$ is a sink of $Q$, the weighted quiver $\mu_k(Q,\dtuple)$ is obtained from $(Q,\dtuple)$ by simply reversing the arrows incident to $k$).
Dlab-Ringel define a representation $\DR^+_k(M)$ of $\mu_k(Q,\dtuple)$ over $E/F$ as follows (they denote $\DR^+_k(M)$ by $S^+_k(M)$ instead).
Let $a_1,\ldots,a_p$ be the arrows of $Q$ that point towards $k$. The vector spaces attached in $\DR^+_k(M)$ to the vertices of $(Q,\dtuple)$ are
$$
\DR^+_k(M)_i=\begin{cases}\ker\alpha & \text{if $i=k$;}\\
M_i & \text{if $i\neq k$;}
\end{cases}
$$
where $\alpha:M_{\Min}=\bigoplus_{s=1}^pF_k\otimes_FM_{t(a_s)}\rightarrow M_k$ is the $F_k$-linear map in \eqref{eq:alpha-&-beta-assembled}, whose components are determined by $M_{a_1},\ldots,M_{a_p}$, by means of Lemma \ref{lemma:tensor-products}. The linear maps attached by $\DR^+_k(M)$ to the arrows of $(Q,\dtuple)$ are given by
$$
\DR^+_{k}(M)_c=\begin{cases}\overleftarrow{\pi_s\iota}  & \text{if $c=a_s^\star$ for some $s=1,\ldots,p$;}\\
M_c & \text{if $c$ is not incident to $k$;}
\end{cases}
$$
where $\iota$ is the inclusion $\ker\alpha\hookrightarrow M_{\Min}$, $\pi_s$ is the canonical projection $M_{\Min}\rightarrow F_k\otimes_FM_{t(a_s)}$ (remember that $F_k\otimes_FM_{t(a_s)}$ is a direct summand of $M_{\Min}$), and $\overleftarrow{\pi_s\iota}$ is the $F$-linear map $\ker\alpha\to M_{t(a_s)}$ corresponding to the $F_k$-linear map $\pi_s\iota: \ker\alpha\rightarrow F_k\otimes_FM_{t(a_s)}$ in Lemma \ref{lemma:tensor-products}.

Now suppose that instead of a representation of $A$ we are given a decorated representation of $(A,S)$ for any potential $S$, say $\mathcal{M}=(M,V)$. The space $M_{\Mout}$ is $0$ since $k$ is a sink of $Q$. Thus the triangle of $F_k$-linear maps \eqref{eq:triangle-alpha-beta-gamma} becomes
\begin{equation}\nonumber
\xymatrix{
 & M_k \ar[dr]^{\beta=0} & \\
M_{\operatorname{in}} \ar[ur]^{\alpha} & & 0 \ar[ll]^{\gamma=0},
}
\end{equation}
where $\alpha:M_{\Min}=\bigoplus_{s=1}^pF_k\otimes_FM_{t(a_s)}\rightarrow M_k$ is given as in the previous paragraph. Since $\frac{\ker\gamma}{\image\beta}=0=\image\gamma$ and $\frac{\ker\beta}{\ker\beta\cap\image\alpha}=\coker\alpha$, the vector spaces attached by $\mu_k(\mathcal{M})=(\overline{M},\overline{V})$ to the vertices of $(Q,\dtuple)$ are
$$
\overline{M}_i=\begin{cases}\ker\alpha\oplus V_k & \text{if $i=k$;}\\
M_i & \text{if $i\neq k$;}
\end{cases}
\ \ \ \text{and} \ \ \
\overline{V}_i=\begin{cases}\coker\alpha & \text{if $i=k$;}\\
V_i & \text{if $i\neq k$.}
\end{cases}
$$
As for the maps $\overline{M}_c$, we first note that the choice of a section $\frac{\ker\alpha}{\image\gamma}\to\ker\alpha$ and a retraction $M_{\Mout}\to\ker\gamma$ become immaterial (both of them are forced to be the corresponding identity map). Therefore,
the linear maps attached by $\overline{M}$ to the arrows of $(Q,\dtuple)$ are given by
$$
\overline{M}_c=\begin{cases}[\overleftarrow{\pi_s\iota} \ 0]  & \text{if $c=a_s^\star$ for some $s=1,\ldots,p$;}\\
M_c & \text{if $c$ is not incident to $k$.}
\end{cases}
$$

These considerations imply the following.

\begin{lemma} Suppose $k$ is a sink of $Q$. Let $\mathcal{M}=(M,V)$ be an indecomposable decorated representation of $(A,S)$, so that either $\mathcal{M}=(M,0)$ or $\mathcal{M}=\mathcal{S}^-_j(A,S)$ for some $j\in Q_0$.
\begin{enumerate}
\item If $\mathcal{M}=(M,0)$, then $\mu_k(\mathcal{M})=(\DR^+_k(M),0)\oplus\left(\mathcal{S}^-_k(A,S)\right)^{\delta_{\mathcal{M},\mathcal{S}_k(A,S)}}$, where
$\delta_{\mathcal{M},\mathcal{S}_k(A,S)}$ is the \emph{Kronecker delta} between $\mathcal{M}$ and $\mathcal{S}_k(A,S)$.
\item If $\mathcal{M}=\mathcal{S}^-_j(A,S)$, then $\mu_k(\mathcal{M})=(\mathcal{S}_k(A,S))^{\delta_{j,k}}\oplus(\mathcal{S}^-_j(A,S))^{1-\delta_{j,k}}$.
\end{enumerate}
\end{lemma}

Now suppose that $k$ is a source of $Q$. As before, let $\mu_k(A)$ be the arrow span of $\mu_k(Q,\dtuple)$ over $E/F$ (since $k$ is a source of $Q$, the weighted quiver $\mu_k(Q,\dtuple)$ is obtained from $(Q,\dtuple)$ by simply reversing the arrows incident to $k$).
Dlab-Ringel define a representation $\DR^-_k(M)$ of $\mu_k(Q,\dtuple)$ over $E/F$ as follows (they denote $\DR^-_k(M)$ by $S^-_k(M)$ instead).
Let $b_1,\ldots,b_q$ be the arrows of $Q$ that point towards $k$. The vector spaces attached in $\DR^-_k(M)$ to the vertices of $(Q,\dtuple)$ are
$$
\DR^-_k(M)_i=\begin{cases}\coker\beta & \text{if $i=k$;}\\
M_i & \text{if $i\neq k$;}
\end{cases}
$$
where $\beta:M_k\rightarrow \bigoplus_{s=1}^q F_k\otimes_FM_{h(b_s)}=M_{\Mout}$ is the $F_k$-linear map in \eqref{eq:alpha-&-beta-assembled}, whose components are determined by $M_{b_1},\ldots,M_{b_q}$, by means of Lemma \ref{lemma:tensor-products}. The linear maps attached by $\DR^-_k(M)$ to the arrows of $(Q,\dtuple)$ are given by
$$
\DR^-_{k}(M)_c=\begin{cases}\overrightarrow{\pi\iota_s}  & \text{if $c=b_s^\star$ for some $s=1,\ldots,q$;}\\
M_c & \text{if $c$ is not incident to $k$;}
\end{cases}
$$
where $\iota_s$ is the inclusion $F_{k}\otimes_FM_{h(b_s)}\hookrightarrow M_{\Mout}$, $\pi$ is the canonical projection $M_{\Mout}\rightarrow \coker\beta$, and $\overrightarrow{\pi\iota_s}$ is the $F$-linear map $M_{h(b_s)}\to\coker\beta$ corresponding to the $F_k$-linear map $\pi\iota_s: F_k\otimes_FM_{h(b_s)}\rightarrow \coker\beta$ in Lemma \ref{lemma:tensor-products}.

Now suppose that instead of a representation of $A$ we are given a decorated representation of $(A,S)$ for any potential $S$, say $\mathcal{M}=(M,V)$. The space $M_{\Min}$ is $0$ since $k$ is a source of $Q$. Thus the triangle of $F_k$-linear maps \eqref{eq:triangle-alpha-beta-gamma} becomes
\begin{equation}\nonumber
\xymatrix{
 & M_k \ar[dr]^{\beta} & \\
0 \ar[ur]^{\alpha=0} & & M_{\Mout} \ar[ll]^{\gamma=0},
}
\end{equation}
where $\beta:M_k\rightarrow \bigoplus_{s=1}^q F_k\otimes_FM_{h(b_s)}=M_{\Mout}$ is given as in the previous paragraph. Since $\ker\alpha=0=\image\gamma$ and $\frac{\ker\beta}{\ker\beta\cap\image\alpha}=\ker\beta$, the vector spaces attached by $\mu_k(\mathcal{M})=(\overline{M},\overline{V})$ to the vertices of $(Q,\dtuple)$ are
$$
\overline{M}_i=\begin{cases}\coker\beta\oplus V_k & \text{if $i=k$;}\\
M_i & \text{if $i\neq k$;}
\end{cases}
\ \ \ \text{and} \ \ \
\overline{V}_i=\begin{cases}\ker\beta & \text{if $i=k$;}\\
V_i & \text{if $i\neq k$.}
\end{cases}
$$
As for the maps $\overline{M}_c$, we first note that the choice of a section $\frac{\ker\alpha}{\image\gamma}\to\ker\alpha$ and a retraction $M_{\Mout}\to\ker\gamma$ become immaterial (both of them are forced to be the corresponding identity map). Therefore,
the linear maps attached by $\overline{M}$ to the arrows of $(Q,\dtuple)$ are given by
$$
\overline{M}_c=\begin{cases}\left[\begin{array}{c}\overrightarrow{\pi\iota_s}\\ 0\end{array}\right]  & \text{if $c=b_s^\star$ for some $s=1,\ldots,q$;}\\
M_c & \text{if $c$ is not incident to $k$.}
\end{cases}
$$

These considerations imply the following.

\begin{lemma} Suppose $k$ is a source of $Q$. Let $\mathcal{M}=(M,V)$ be an indecomposable decorated representation of $(A,S)$, so that either $\mathcal{M}=(M,0)$ or $\mathcal{M}=\mathcal{S}^-_j(A,S)$ for some $j\in Q_0$.
\begin{enumerate}
\item If $\mathcal{M}=(M,0)$, then $\mu_k(\mathcal{M})=(\DR^-_k(M),0)\oplus\left(\mathcal{S}^-_k(A,S)\right)^{\delta_{\mathcal{M},\mathcal{S}_k(A,S)}}$, where
$\delta_{\mathcal{M},\mathcal{S}_k(A,S)}$ is the \emph{Kronecker delta} between $\mathcal{M}$ and $\mathcal{S}_k(A,S)$.
\item If $\mathcal{M}=\mathcal{S}^-_j(A,S)$, then $\mu_k(\mathcal{M})=(\mathcal{S}_k(A,S))^{\delta_{j,k}}\oplus(\mathcal{S}^-_j(A,S))^{1-\delta_{j,k}}$.
\end{enumerate}
\end{lemma}

We see that, for species satisfying \eqref{eq:coprime-assumption}, \eqref{eq:d-root-of-unity} and \eqref{eq:E/F-Galois-simple-spectrum}, the mutations of species with potentials we have defined here constitute a generalization of Dlab-Ringel reflection functors which is completely analogous to the generalization of Bernstein-Gelfand-Ponomarev reflection functors provided by the mutations of quivers with potentials in the simply laced case. Furthermore, the way we have generalized the QP-mutation theory of \cite{DWZ1} is analogous to the way that Dlab-Ringel reflection functors generalize Bernstein-Gelfand-Ponomarev reflection functors.

\begin{remark} Dlab-Ringel define reflection functors for sinks and sources in a setup that is far more general than that of finite or $p$-adic fields, allowing division rings to be attached to the vertices of a weighted quiver. Moreover, the tuple $\dtuple$ is allowed to be arbitrary. In this more general context the present paper does not provide a generalization of Dlab-Ringel reflection functors.
\end{remark}

\section{Unfoldings}\label{sec:unfoldings}

This paper has followed \cite{DWZ1} extending its setup to \emph{strongly primitive} skew-symmetrizable integer matrices, that is, integer matrices $B$ for which there exist pairwise coprime positive integers $d_1, \dots, d_n$ with $d_i b_{ij} = - d_j b_{ji}$ for all $i$ and
$j$. 
We intend to eventually apply this framework to cluster algebras in the spirit of \cite{DWZ2}.
However, there is a general method for reducing statements about skew-symmetrizable cluster algebras to the skew-symmetric ones based on foldings/unfoldings, and we would like to
be sure that our restrictions still cover some matrices $B$ \emph{not} allowing any unfoldings from the cluster algebra point of view. That is, we would like to show that there exist strongly primitive skew-symmetrizable matrices whose unfoldings are not compatible with matrix mutations.

The general notion of unfolding for skew-symmetrizable matrices goes back at least to \cite[Step~1, Section~2.4]{FZ-Y-systems}, and was then used by
G.~Dupont \cite{Dupont} and L.~Demonet \cite{Demonet}.
We present it in the  form suggested by the second author several years ago (unpublished), and later reproduced in \cite{FeST}.

\begin{defi}
For a skew-symmetrizable $n \times n$ integer matrix $B$ (not necessarily strongly primitive), an \emph{unfolding} of
$B$ is a triple $(C,(e_i)_{1\leq i\leq n},(E_i)_{1\leq i\leq n})$ consisting of a choice of:
\begin{itemize}
\item positive integers $e_1, \dots, e_n$ such that $b_{ij} e_j = - b_{ji} e_i$;
\item disjoint index sets $E_1, \dots, E_n$ with $|E_i| = e_i$; and
\item a skew-symmetric  integer matrix $C$ with rows and columns
indexed by the union of all $E_i$,
\end{itemize}
satisfying the following conditions:
\begin{enumerate}
\item the sum of entries in each column of each $E_i \times E_j$ block of
$C$ is equal to $b_{ij}$;
\item if $b_{ij} \geq 0$ then the $E_i \times E_j$ block of $C$ has all
entries nonnegative.
\end{enumerate}
\end{defi}

Note that (2) implies a weaker condition:

\smallskip

$(2')$ each $E_i \times E_i$ diagonal block of $C$ is $0$.

\smallskip

This condition makes each \emph{composite mutation}
$\mu_k = \prod_{\bar k \in E_k} \mu_{\bar k}$ well-defined for $C$ since the factors in the product commute with each other.
However $(\mu_k(C),(e_i)_{1\leq i\leq n},(E_i)_{1\leq i\leq n})$ is not necessarily an unfolding of $\mu_k(B)$ since condition (2) may
become violated.

\begin{defi}
We say that $B$ \emph{admits a global unfolding} if an unfolding $(C,(e_i)_{1\leq i\leq n},(E_i)_{1\leq i\leq n})$ of $B$ exists with the property that for every finite sequence $(k_1,\ldots,k_\ell)$ of indices from $\{1,\ldots,n\}$, the data $(\mu_{k_\ell}\ldots\mu_{k_1}(C),(e_i)_{1\leq i\leq n},(E_i)_{1\leq i\leq n})$ is an unfolding of $\mu_{k_\ell}\ldots\mu_{k_1}(B)$, that is, if condition (2) is never violated.
\end{defi}

Here is an example of a family of global unfoldings.


\begin{ex}\label{ex:unfoldable-class} Suppose $B$ is an $n\times n$ skew-symmetrizable matrix satisfying the following condition: there exist positive integers $e_1, \dots, e_n$ such that
\begin{itemize}
\item $b_{ij} e_j = - b_{ji} e_i$ for all $i,j$;
\item $e_i$ divides $b_{ij}$ for all $i,j$.
\end{itemize}
We claim that $B$ admits a global unfolding. Let $E_1, \dots, E_n$ be disjoint sets with $|E_i| = e_i$.
Define the matrix $C$ with rows and columns
indexed by the union of all $E_i$ by setting
$$
c_{\bar i \ \bar j} = b_{ij}/e_i \quad (\bar i \in E_i, \ \bar j \in E_j) \ .
$$
It is easy to see that these choices define an unfolding $(C,(e_i)_{1\leq i\leq n},(E_i)_{1\leq i\leq n})$ of $B$.
For every $k = 1, \dots, n$, it is easy to check that $(\mu_k(C),(e_i)_{1\leq i\leq n},(E_i)_{1\leq i\leq n})$ is an unfolding of $\mu_k(B)$ that is obtained from $\mu_k(B)$ in the same way as $C$ from $B$.
\end{ex}

However it is even more essential for us to give a series of examples (including some primitve ones) that do \emph{not} admit a global unfolding.

\begin{ex}\label{ex:non-unfoldable}
Let $a$ and $b$ be two positive integers such that $a < b$, and $b$ is \emph{not} a multiple of $a$.
We associate with $a$ and $b$ a skew-symmetrizable integer $4 \times 4$ matrix
$$B = \begin{pmatrix} 0 & -a & 0 & b \\ 1 & 0 & -1 & 0 \\ 0 &  a & 0 &
-b \\ -1 & 0 & 1 & 0 \end{pmatrix} \ .$$
Note that we can choose the skew-symmetrizing tuple $(d_1,d_2,d_3,d_4)$ as $(1,a,1,b)$; in particular, it is primitive
if $a$ and $b$ are coprime.
We claim that even under a weaker condition that $a < b$, and $a$ does not divide $b$, the matrix $B$ does \emph{not} admit a global unfolding.

Note that $(e_1,e_2,e_3,e_4) = (N, N/a, N, N/b)$, where $N$ is a common multiple of $a$ and $b$.
We choose the corresponding disjoint sets $E_1, \dots, E_4$, and let $C$ be a skew-symmetric integer matrix that provides an unfolding of $B$.
Let $C_{ij}$ denote the $E_i \times E_j$ block of $C$.
Thus the blocks $C_{21}, C_{32}, C_{43}, C_{14}$ consist of nonnegative integer entries, the opposite blocks $C_{12}, \dots C_{41}$
consist of nonpositive integer entries, and the remaning blocks are $0$.
Furthermore, since $b_{21} = b_{43} = 1$, the blocks $C_{21}$ and $C_{43}$ consist of $0$'s and $1$'s; by the same token,
since $b_{23} = b_{41} = -1$, the blocks $C_{23}$ and $C_{41}$ consist of $0$'s and $-1$'s.
Using the fact that $C$ is skew-symmetric, we conclude that all entries of $C$ belong to $\{-1, 0, 1\}$.

Now let $B' = \mu_2 \mu_4 (B)$ (since $b_{24} = 0$, the order of factors does not matter).
By a direct calculation, we have
$$B' = \begin{pmatrix} 0 & a & b-a & -b \\ -1 & 0 & 1 & 0 \\ a-b &  -a & 0 &
b \\ 1 & 0 & -1 & 0 \end{pmatrix} \ .$$
Let $C' = \mu_2 \mu_4 (C)$.
We claim that $C'$ cannot satisfy condition (2) in the definition of an unfolding.
More precisely, we will show that the block $C'_{13}$ has to contain at least one positive entry and at least one negative entry.

By the definition, $C'_{13}$ is a $N \times N$ matrix given by $C'_{13} = C_{14} C_{43} - C_{12} C_{23}$.
We note that $C_{14} C_{43}$ is a $(0,1)$ matrix that can be described as follows: relabeling $E_1$ and $E_3$ if necessary,
$C_{14} C_{43}$ becomes a block diagonal matrix with $N/b$ blocks, each block being equal to a $b \times b$ matrix $\mathbf{1}_b$ filled with $1$'s.
Similarly, relabeling $E_1$ and $E_3$ if necessary,
$C_{12} C_{23}$ becomes a block diagonal matrix with $N/a$ blocks, each block being equal to $\mathbf{1}_a$.
Now $C_{14} C_{43}$ has $Nb$ entries equal to $1$, while  $C_{12} C_{23}$ has $Na$ entries equal to $1$.
Since $a < b$, it follows that $C'_{13}$ has at least one entry equal to $1$.

On the other hand, sice $b$ is not a multiple of $a$, there exists a block $\mathbf{1}_a$ of $C_{12} C_{23}$ not contained
in any block $\mathbf{1}_b$ of $C_{14} C_{43}$.
It follows that there is an entry in this block $\mathbf{1}_a$ not contained in the union of all blocks of $C_{14} C_{43}$.
Thus, the corresponding entry of $C'_{13}$ is equal to $-1$, finishing the proof that $B$ does not admit a global unfolding.
\end{ex}

\begin{remark}
The global non-unfoldability of the above matrix $B$ for $(a,b) = (2,3)$ was found 
by D.~Speyer and the second author several years ago (unpublished).
The idea to include it into the above family of examples was suggested by F.~Petrov.
\end{remark}

We conclude this section with a brief discussion of how global unfoldings can be used to reduce the proofs of certain properties
of skew-symmetrizable cluster algebras to he corresponding properties of skew-symmetric ones.
For simplicity we consider only the coefficient-free cluster algebras although the reduction procedure can be generalized to various coefficient systems.

Let $B$ be as above.
Recall that the coefficient-free cluster algebra $\mathcal{A}(B)$ is a subring of the \emph{ambient field} $\mathcal{F} = \Q(u_1, \dots, _n)$ generated by the union of all
\emph{clusters}.
We start with the \emph{initial cluster} $\{x_1, \dots, x_n\}$ which is a free generating set for $\mathcal{F}$.
Together with the initial exchange matrix $B$, this cluster forms an initial \emph{seed}.
For each $k = 1, \dots, n$, the mutation $\mu_k$ transforms this seed into a new one, with the exchange matrix $B' = \mu_k(B)$ and the new cluster obtained from $\{x_1, \dots, x_n\}$
by exchanging $x_k$ with $x'_k$ given by the \emph{exchange relation}
$$x_k x'_k = \prod_i x_i^{[b_{ik}]_+} +  \prod_i x_i^{[- b_{ik}]_+} \ .$$
Iterating this mutation procedure, we obtain all clusters of $\mathcal{A}(B)$.

Now suppose that a skew-symmetric matrix $C$ as above defines a global unfolding of $B$.
Consider all clusters of $\mathcal{A}(C)$ obtained from the initial one by a sequence of ``composite mutations" $\mu_k$
(recall that $\mu_k = \prod_{\bar k \in E_k} \mu_{\bar k}$), and let $\mathcal{A}(C)^\circ$ denote the subalgebra of
$\mathcal{A}(C)$ generated by the union of all these clusters.
The conditions in the definition of an unfolding imply that the map of initial clusters given by $x_{\bar i} \mapsto x_i$ for $\bar i \in E_i$, extends to a surjective ring
homomorphism $\mathcal{A}(C)^\circ \to \mathcal{A}(B)$.
Therefore $\mathcal{A}(B)$ can be identified with a subquotient of $\mathcal{A}(C)$.
It is this identification that allows to deduce various properties of $\mathcal{A}(B)$ from the corresponding properties of $\mathcal{A}(C)$.

\section{Other developments}\label{sec:other-developments}

Earlier approaches to non--skew-symmetric cluster algebras via representations of quivers or species include:
\begin{itemize}
\item L. Demonet's generalization \cite{Demonet} of Derksen-Weyman-Zelevinsky's mutation theory of quivers with potentials through a mutation theory of \emph{group species with potentials} that succesfully covers all skew-symmetrizable matrices in Example \ref{ex:unfoldable-class} above, and all skew-symmetrizable matrices that are mutation-equivalent to acyclic ones.
\item G. Dupont's approach \cite{Dupont} to cluster algebras whose exchange matrices admit global unfoldings.
\item B. Nguefack's mutations of species with potentials in a quite general setting \cite{Nguefack}.
\item D. Rupel's approach \cite{Rupel,Rupel2} to acyclic skew-symmetrizable cluster algebras and their quantum versions, using representations of species over finite fields. 
\item D. Speyer and H. Thomas' description \cite{ST} of $c$-vectors for acyclic seeds of
acyclic skew-symmetrizable cluster algebras, using derived categories of species over finite fields.
\end{itemize}
Let us say a few words about the works of Demonet \cite{Demonet} and Nguefack \cite{Nguefack}, for both of them involve species and potentials.

In \cite{Nguefack}, B. Nguefack develops a quite general approach to a mutation
theory of species with potentials. He considers a notion of species far
more general than the one we have considered in the present paper.
However, in that generality it is quite difficult to address the problem of existence of non-degenerate potentials or the definition of
mutations of representations.

Through a mutation theory of \emph{group species with potentials} \cite{Demonet}, a generalization of Derksen-Weyman-Zelevinsky's mutation theory of quivers with potentials is developed by L. Demonet that covers every matrix mutation-equivalent to an acyclic skew-symmetrizable matrix, as well as every skew-symmetrizable matrix $B$ satisfying the condition that
\begin{eqnarray}\label{eq:Demonet-condition}
&&\text{there exist positive integers $c_1,\ldots,c_n$, such that $C^{-1}B$ is skew-symmetric}\\
\nonumber
&&\text{and has integer entries, where $C=\diag(c_1,\ldots,c_n)$}.
\end{eqnarray}

As we saw in Section \ref{sec:unfoldings}, the approach we have taken in the present paper covers a large class of skew-symmetrizable matrices $B$ that do not admit global unfoldings, and hence do not satisfy \eqref{eq:Demonet-condition}.
This is the main difference between our approach and that of Demonet.
Another significant difference concerns the algebraic-combinatorial setups. We give a brief comparison of these.
Let $d=\lcm(d_j\suchthat 1\leq j\leq n)$, and for each $i\in Q_0=\{1,\ldots,n\}$, let $d_i'=\frac{d}{d_i}$. Then the matrix $B\overline{D}$ is skew-symmetric, where $\overline{D}=\diag(d_1',\ldots,d_n')$. To each $i\in Q_0$, Demonet attaches the group algebra $K[\Gamma_i]$, where $K$ is an algebraically closed field and $\Gamma_i=\Z/d_i'\Z$ is the cyclic group of order $d_i'$. To each pair $(i,j)\in Q_0\times Q_0$ such that $b_{ij}>0$ he furthermore attaches the $K[\Gamma_j]$-$K[\Gamma_i]$-bimodule $K[\Z/d_j'b_{ij}\Z]$ (since he works with right modules, arrows are implicitly interpreted to go from $j$ to $i$ whenever $b_{ij}>0$, although this is not explicitly stated in \cite{Demonet}).

So, we see that while in the present paper the rings attached to the vertices of $Q$ are fields, the rings attached by Demonet are not fields but group algebras (which are never fields unless the group is trivial). Moreover, while the dimension of the field $F_i$ a vector space over the ground field is $d_i$, the dimension of the group algebra $K[\Gamma_i]$ is $d_i'=\frac{d}{d_i}$.

Now, for every skew-symmetrizable matrix $B$ that either satisfies \eqref{eq:Demonet-condition} or is mutation-equivalent to an acyclic one, Demonet shows that the corresponding \emph{group species} over an algebraically closed field $K$ always admits a non-degenerate potential. On the other hand, for every strongly primitive skew-symmetrizable matrix $B$
we have proved here that the corresponding species over a $p$-adic field always admits a non-degenerate potential, but over a finite field we have not been able to prove such ``global" existence of non-degenerate potentials, but only a ``local" existence (along finite sequences of mutations).

It is worth pointing out that Demonet ultimately proves several conjectures from Fomin-Zelevinsky's \cite{FZ4} for all skew-symmetrizable matrices that either satisfy \eqref{eq:Demonet-condition} or are mutation-equivalent to acyclic ones.

The following table outlines a basic comparison of approaches between Demonet's paper \cite{Demonet} and the present one.
\begin{center}
{\renewcommand{\arraystretch}{1.5}
\renewcommand{\tabcolsep}{0.2cm}
{\small
\begin{tabular}{|r|c|c|}
\hline
 & Demonet (cf. \cite{Demonet}) & Here \\
\hline
Ground field & Algebraically closed field $K$ & Finite or $p$-adic field $F$ \\
\hline
Ring attached to vertex $i\in Q_0$ & Group algebra $K[\Gamma_i]=K[\Z/d_i'\Z]$ & Field $F_i$ \\ \cline{2-3}
Dim. of such ring over ground field & $d_i'=d/d_i$
 & $d_i$\\
\hline
Bimodule attached whenever $b_{ij}>0$ & $K[\Z/d_j'b_{ij}\Z]$, attached as  & $(F_i\otimes_FF_j)^{b_{ij}/d_j}$, attached as  \\
 & $K[\Gamma_j]$-$K[\Gamma_i]$-bimodule & $F_i$-$F_j$-bimodule \\ \cline{2-3}
Arrows going from & $j$ to $i$ (left to right, implicitly) & $j$ to $i$ (right to left, explicitly) \\ \cline{2-3}
Representations are & right modules & left modules \\
\hline
Non-deg. potentials shown to exist & globally & locally \\ \cline{2-3}
provided & $C^{-1}B$ is skew-symmetric and   & $B$ admits a skew-symmetrizer  \\
 & has integer entries for some & $D=\diag(d_1,\ldots,d_n)$ such that \\
& $C=\diag(c_1,\ldots,c_n)$  & $\gcd(d_i,d_j)=1$ for all $i\neq j$ \\
 & such that $c_1,\ldots,c_n\in \Z_{>0}$;  or  &   \\\cline{2-2}
 & $B$ is mutation-equivalent & \\
 & to an acyclic matrix & \\
\hline
Cluster mutation defined along & columns of $B$ (implicit) & columns of $B$ (explicit) \\
\hline
\end{tabular}
}}
\end{center}

\end{document}